\documentclass[a4paper, 12pt, reqno,final]{amsart}
\usepackage[foot]{amsaddr} 
\usepackage[utf8x]{inputenc}
\usepackage[T1]{fontenc}
\usepackage[english]{babel}
\usepackage{amsmath, amsfonts, amsthm, amssymb, amscd}
\usepackage[final]{graphicx}
\usepackage[below]{placeins}
\usepackage{subfig}
\usepackage{multirow}
\usepackage{mathtools}

\usepackage[hidelinks]{hyperref}
\hypersetup{
  pdfauthor = {Stefan Metzger, G\"unther Gr\"un},
  pdftitle = {SPDE 2d}
}

\usepackage{showkeys} 
\usepackage{esint}
\usepackage{caption}
\usepackage{dsfont}

\usepackage{tikz}
\usetikzlibrary{patterns}
\usetikzlibrary{shapes.geometric}

\usepackage{diagbox}

\newtheorem{theorem}{Theorem}[section]
\newtheorem{lemma}[theorem]{Lemma}
\newtheorem*{lemma*}{Lemma}
\newtheorem{corollary}[theorem]{Corollary}
\newtheorem{proposition}[theorem]{Proposition}
\newtheorem{remark}[theorem]{Remark}
\newtheorem{definition}[theorem]{Definition}
\numberwithin{equation}{section}



\makeatletter
\newcommand{\labitem}[2]{%
\def\@itemlabel{\textbf{#1}}
\item
\def\@currentlabel{#1}\label{#2}}
\makeatother

\newcommand{\norm}[1]{\left\|{#1}\right\|}
\newcommand{\abs}[1]{\left|{#1}\right|}

\newcommand{\rkla}[1]{{\left(#1\right)}}
\newcommand{\trkla}[1]{{(#1)}}
\newcommand{\gkla}[1]{{\left\{#1\right\}}}
\newcommand{\tgkla}[1]{{\{#1\}}}
\newcommand{\skla}[1]{{\left\langle#1\right\rangle}}
\newcommand{\ekla}[1]{{\left[#1\right]}}
\newcommand{\tekla}[1]{{[#1]}}
\newcommand{\tabs}[1]{|{#1}|}

\newcommand{\bs}[1]{\boldsymbol{#1}}

\newcommand{\iOmega}{\int_{\Omega}}
\newcommand{\Om}{\mathcal{O}}
\newcommand{\Ox}{\Om^x}
\newcommand{\Oy}{\Om^y}
\newcommand{\Lx}{L_x}
\newcommand{\Ly}{L_y}
\newcommand{\iO}{\int_{\Om}}
\newcommand{\iOx}{\int_{\Ox}}
\newcommand{\iOy}{\int_{\Oy}}

\newcommand{\itTh}{\int_{0}^{t\wedge T\h}}
\newcommand{\itTht}{\int_{0}^{t\wedge \tilde{T}\h}}

\newcommand{\iOOTTh}{\int_{0}^{T\wedge T\h} \!\!\! \iO}
\newcommand{\Tmax}{T_{\text{max}}}
\newcommand{\iTmaxTh}{\int_{0}^{\Tmax\wedge T\h}}
\newcommand{\chiTh}{\chi_{T\h}}
\newcommand{\chiTth}{\chi_{\tilde{T}\h}}

\newcommand{\para}[1]{\partial _{#1}}

\newcommand{\dx}{\, \mathrm{d}{x}}
\newcommand{\dy}{\, \mathrm{d}{y}}
\newcommand{\dt}{\, \mathrm{d}t}
\newcommand{\ds}{\, \mathrm{d}s}
\newcommand{\dtau}{\, \mathrm{d}\tau}
\newcommand{\parx}{\partial_x}
\newcommand{\pary}{\partial_y}

\renewcommand{\div}{\operatorname{div}}

\newcommand{\ids}{I}

\newcommand{\h}{_{h}}

\newcommand{\weak}{\rightharpoonup}
\newcommand{\weakstar}{\stackrel{*}{\rightharpoonup}}

\newcommand{\Uhx}{U_h^x}
\newcommand{\Uhy}{U_h^{y}}
\newcommand{\Uhxy}{U_{h}}
\newcommand{\Thx}{\mathcal{T}_h^x}
\newcommand{\Thy}{\mathcal{T}_h^y}
\newcommand{\Qh}{\mathcal{Q}_h}

\newcommand{\R}{\mathds{R}}

\newcommand{\N}{\mathds{N}}
\newcommand{\Z}{\mathds{Z}}

\newcommand{\Kx}{{K^x}}
\newcommand{\Ky}{{K^y}}

\newcommand{\Q}{Q}

\newcommand{\hx}{h_x}
\newcommand{\hy}{h_y}

\newcommand{\Ihxop}{\mathcal{I}_h^x}
\newcommand{\Ihyop}{\mathcal{I}_h^y}
\newcommand{\Ihxyop}{\mathcal{I}_h^{xy}}

\newcommand{\Ihx}[1]{\Ihxop\gkla{#1}}
\newcommand{\Ihy}[1]{\Ihyop\gkla{#1}}
\newcommand{\Ihxy}[1]{\Ihxyop\gkla{#1}}

\newcommand{\Ihxlocop}{\mathcal{I}_{h,\operatorname{loc}}^x}
\newcommand{\Ihylocop}{\mathcal{I}_{h,\operatorname{loc}}^y}

\newcommand{\Ihxloc}[1]{\Ihxlocop\gkla{#1}}
\newcommand{\Ihyloc}[1]{\Ihylocop\gkla{#1}}

\newcommand{\restr}[2]{\ensuremath{
  \left.\kern-\nulldelimiterspace 
  #1 
  \vphantom{\big|} 
  \right|_{#2} 
  }}
  
\newcommand{\extend}[2]{\ensuremath{
  \left.\kern-\nulldelimiterspace 
  #1 
  \vphantom{\big|} 
  \right|^{#2} 
  }}
  
\newcommand{\diam}{\operatorname{diam}}

\newcommand{\lx}[1]{{\lambda^x_{#1}}}
\newcommand{\ly}[1]{{\lambda^y_{#1}}}
\newcommand{\lalpha}[1]{{\lambda^{\alpha}_{#1}}}
\newcommand{\g}[1]{\mathfrak{g}_{#1}}
\newcommand{\gt}[1]{\tilde{\mathfrak{g}}_{h,#1}}

\newcommand{\meanGinvX}[1]{\ekla{{G^{\prime\prime}\trkla{#1}}}^{-1}_x}
\newcommand{\meanGinvY}[1]{\ekla{{G^{\prime\prime}\trkla{#1}}}^{-1}_y}
\newcommand{\meanFX}[1]{\ekla{F^{\prime\prime}\trkla{#1}}_x}
\newcommand{\meanFY}[1]{\ekla{F^{\prime\prime}\trkla{#1}}_y}
\newcommand{\meanspX}[1]{\ekla{\tabs{#1}^{-p-2}}_x}
\newcommand{\meanspY}[1]{\ekla{\tabs{#1}^{-p-2}}_y}

\newcommand{\Zx}{I^x_h}
\newcommand{\Zy}{I^y_h}

\newcommand{\bx}[1]{\beta_{#1}^x}
\newcommand{\by}[1]{\beta_{#1}^y}
\newcommand{\balpha}[1]{\beta_{#1}^{\alpha}}
\newcommand{\bthx}[1]{\tilde{\beta}_{h,#1}^x}
\newcommand{\bthy}[1]{\tilde{\beta}_{h,#1}^y}
\newcommand{\bthalpha}[1]{\tilde{\beta}_{h,#1}^{\alpha}}
\newcommand{\btx}[1]{\tilde{\beta}_{#1}^x}
\newcommand{\bty}[1]{\tilde{\beta}_{#1}^y}
\newcommand{\btalpha}[1]{\tilde{\beta}_{#1}^{\alpha}}
\newcommand{\dbx}[1]{\,\operatorname{d}\!\bx{#1}}
\newcommand{\dby}[1]{\,\operatorname{d}\!\by{#1}}
\newcommand{\dbalpha}[1]{\,\operatorname{d}\!\balpha{#1}}
\newcommand{\dbthx}[1]{\,\operatorname{d}\!\bthx{#1}}
\newcommand{\dbthy}[1]{\,\operatorname{d}\!\bthy{#1}}
\newcommand{\dbtx}[1]{\,\operatorname{d}\!\btx{#1}}
\newcommand{\dbty}[1]{\,\operatorname{d}\!\bty{#1}}
\newcommand{\dbtalpha}[1]{\,\operatorname{d}\!\btalpha{#1}}

\newcommand{\ex}[1]{\mathfrak{e}^x_{#1}}
\newcommand{\ey}[1]{\mathfrak{e}^y_{#1}}

\newcommand{\kbalpha}{\bs{b}_{\alpha}}

\newcommand{\p}{{\overline{p}}}

\newcommand{\Emax}{\mathcal{E}_{\text{max},h}}

\newcommand{\dxp}{\partial_x^{+\hx}}
\newcommand{\dxm}{\partial_x^{-\hx}}
\newcommand{\dyp}{\partial_y^{+\hy}}
\newcommand{\dym}{\partial_y^{-\hy}}

\newcommand{\expected}[1]{\mathds{E}\ekla{#1}}
\newcommand{\expectedt}[1]{\tilde{\mathds{E}}\ekla{#1}}
\newcommand{\Prob}{\mathds{P}}
\newcommand{\prob}[1]{\Prob\ekla{#1}}
\newcommand{\probt}[1]{\tilde{\Prob}\ekla{#1}}

\newcommand{\Xu}{\mathcal{X}_u}
\newcommand{\Xddu}{\mathcal{X}_{\Delta u}}
\newcommand{\XJx}{\mathcal{X}_{J^x}}
\newcommand{\XJy}{\mathcal{X}_{J^y}}
\newcommand{\XW}{\mathcal{X}_{\bs{W}}}
\newcommand{\XuInit}{\mathcal{X}_{u_0}}
\newcommand{\X}{\mathcal{X}}

\newcommand{\projop}{\mathcal{P}_{\Uhxy}}
\newcommand{\proj}[1]{\projop\gkla{#1}}

\newcommand{\per}{{\operatorname{per}}}

\newcommand{\charac}[1]{\chi_{\tekla{#1}}}

\newcommand{\dist}{\operatorname{dist}}

\newcommand{\marth}{\mathcal{M}_{h,v}}
\newcommand{\martht}{\tilde{\mathcal{M}}_{h,v}}
\newcommand{\martt}{\tilde{\mathcal{M}}_{v}}
\newcommand{\qvar}[1]{\skla{\!\skla{#1}\!}}
\newcommand{\crossvar}[2]{\skla{\!\skla{#1,\,#2}\!}}

\newcommand{\Ritzop}{\mathcal{R}}
\newcommand{\Ritz}[1]{\Ritzop\gkla{#1}}

\newcommand{\esssup}{\operatorname{ess}\,\sup}

\newcommand{\Ito}{It\^{o}}

\definecolor{m1}{rgb}{0.1,0.5,1}
\definecolor{m2}{rgb}{0.9,0.7,0.1}
\definecolor{m3}{rgb}{.1,1,1}
\definecolor{m4}{rgb}{.9,.2,.1}
\definecolor{m5}{rgb}{.0,.4,.0}
\definecolor{m6}{rgb}{0.3,0.,0.1}
\definecolor{g1}{rgb}{0.1,0.5,1}

\newcommand{\revx}[1]{#1}
\newcommand{\revy}[1]{#1}

\begin{document}
\title[]{Existence of nonnegative solutions to stochastic thin-film equations in two space dimensions}
\date{\today}
\author[S.~Metzger]{Stefan Metzger}
\address[S.~Metzger, (corresponding author)]{Friedrich--Alexander Universität Erlangen--Nürnberg,~Cauerstraße 11,~91058~Erlangen,~Germany}
\email{stefan.metzger@fau.de}

\author[G.~Gr\"un]{G\"unther Gr\"un}
\address[G.~Gr\"un]{Friedrich--Alexander Universität Erlangen--Nürnberg,~Cauerstraße 11,~91058~Erlangen,~Germany}
\email{gruen@math.fau.de}


\keywords{stochastic thin-film equation, thermal noise, nonnegativity preserving scheme, stochastic partial differential equations, finite element method}
\subjclass[2010]{Primary 60H15, 76D08, 35K65; Secondary 65M12, 35K25, 35Q35}

%
\selectlanguage{english}

\allowdisplaybreaks

\begin{abstract}
We prove the existence of martingale solutions to stochastic thin-film equations in the physically relevant space dimension $d=2$. Conceptually, we rely on a stochastic Faedo-Galerkin approach using tensor-product linear finite elements in space. Augmenting the physical energy on the approximate level by a curvature term weighted by positive powers of the spatial discretization parameter $h$, we combine \Ito's formula with inverse estimates and appropriate stopping time arguments to derive stochastic counterparts of the energy and entropy estimates known from the deterministic setting. 
In the limit $h\searrow 0$, we prove our strictly positive finite element solutions to converge towards  nonnegative martingale solutions --- making use of compactness arguments based on Jakubowski's generalization of Skorokhod's theorem and subtle exhaustion arguments to identify third-order spatial derivatives in the flux terms. 
\end{abstract}
\maketitle

\section{Introduction}\label{sec:introduction}
We are concerned with stochastic thin-film equations of the generic form
\begin{align}\label{eq:stfe}
\text{d}u=-\div\tgkla{m\trkla{u}\nabla\trkla{\Delta u-F^\prime\trkla{u}}}\dt+\div\tgkla{\sqrt{m\trkla{u}}\text{d}\bs{W}}
\end{align}
on a space-time cylinder $\mathcal O\times (0,T]$ where $\mathcal O$ is a bounded rectangular domain in $\R^2.$
Such kind of equations have been introduced to model dewetting of unstable liquid films under the influence of thermal fluctuations.
Here, the mobility $m(\cdot)$ may be chosen as $m(u)=u^3+\beta u^2$,  a prototypical example for the effective interface potential $F(u)$ is $F(u)=u^{-8} - u^{-2} +1.$ It is based on a $6-12$ Lennard--Jones pair potential and it models disjoining/conjoining van der Waals interactions.
Numerical simulations in 1D have shown (see \cite{DORKP2019, Grun2006}) that discrepancies with respect to time scales of dewetting between physical experiment and deterministic numerical simulation can be overcome if appropriately scaled noise terms are considered.
The equations used for those models come along with a number of intrinsic difficulties. First, the degeneracy of the mobility $m(\cdot)$ in $u=0$, secondly the singular behaviour of the effective interface potential at $u=0$, and thirdly the fact that even in the deterministic case it is  still an open problem whether solutions are continuous or bounded if the spatial dimension is at least $d=2$. 
This is in sharp contrast to the one-dimensional case where Sobolev embedding results are the key to establish H\"older continuity  in space and time.

\revx{The scope of the present work is threefold. First, we wish to establish the existence of martingale solutions for a model problem with quadratic mobility $m(u)=u^2$. Secondly, we shall address the case that the stochastic integral in \eqref{eq:stfe} is to be understood in the sense of Stratonovich. 
Finally, as the third perspective of the techniques developed in the present paper, we recall that Cahn--Hilliard equations with degenerate mobility are intimately related to thin-film equations \cite{GarckeCahnHilliard, ZAAGruen}.  Hence, we expect only slight modifications to be required to obtain existence results for Cahn--Hilliard equations with Stratonovich conservative noise -- for results on stochastic degenerate Cahn--Hilliard equations with non-conservative noise, we refer to the recent paper \cite{Scarpa2021}.

With respect to the second goal, it turns out that the correction term necessary to rewrite the Stratonovich integral as an \Ito~integral can be included into the frame of the generic form \eqref{eq:stfe}  just as a modification of the singular potential $F\trkla{\cdot}$. In particular, this modification does not affect the structural conditions formulated on $F$ in the \Ito-case,  provided some natural hypotheses on the decay parameters of the corresponding Wiener processes are met.
Therefore, the arguments presented hereafter carry over to the case of Stratonovich noise via slight changes in the singular potential $F$.
}

Growing interest in stochastic thin-film equations with Stratonovich noise arose with the work of Gess and Gnann  \cite{GessGnann2020} who proved existence of solutions to stochastic versions of thin-film equations driven only by surface tension, i.e.~with $F(\cdot)\equiv 0.$ Using a novel approximation scheme, Dareiotis, Gess, Gnann, and the second author of this paper \cite{DareiotisGGG} extended this result to stochastic thin-film equations with conservative nonlinear multiplicative noise, covering in particular the case of a no-slip condition at the liquid-solid interface which corresponds to $m(u)=u^3.$ 
\revx{In contrast to \cite{DareiotisGGG} which covers the case $m(u)=u^n$ with $n\in [8/3, 4)$, \cite{GessGnann2020} allows for initial data with compact support, however, at the prize of reduced regularity in space. Another approach to construct nonnegative solutions to initial data with compact support is presented in \cite{GrunKlein2021}. Here, the 1D-predecessor of the current work (see \cite{FischerGrun2018}) is used as the starting point of the analysis. So called $\alpha$-entropy estimates are derived which guarantee almost surely that almost everywhere in time the solution is of class $C^1$ in space, this way providing more regularity than \cite{GessGnann2020}. In this spirit, the present paper may similarly serve as a starting point to investigate the case $F\trkla{\cdot}\equiv0$ and to obtain solutions to equations with compactly supported initial data in the 2D setting.
Two months after the current work had been submitted, Sauerbrey \cite{Sauerbrey2021} suggested another approach, combining $\alpha$-entropy estimates with the method of \cite{GessGnann2020} to establish existence of solutions in 2D for quadratic mobilities  in the case of Stratonovich noise with $F\equiv 0.$}

Note that Davidovitch et al.~\cite{DMS2005} derived stochastic thin-film equations for surface tension driven flow to study numerically the influence of thermal noise on the spreading behaviour. For mathematically rigorous results on the noise impact on free boundary propagation, see \cite{fischergruen2015, DirrGrillmeierGruen2021, Grillmeier2020, gess2013, barbuRoeckner2012} and the references therein which are preliminary studies devoted to stochastic second order degenerate parabolic equations of porous-media and of parabolic $p$-Laplace-type. However, it is worth mentioning that the techniques of \cite{fischergruen2015, Grillmeier2020} are expected to be universal in the sense that they are based on energy methods developed in \cite{WaitingTime, Grun2001b, DoublyNonlinear, ChipotSideris, UpperBoundsThinFilmWeakSlippage, SupportPropagationThinFilm} which work for general classes of degenerate parabolic equations of second and higher order.

In \cite{FischerGrun2018}, it has already been argued that space-time white noise is not compatible with the finiteness of the physical energies encountered in thin-film flow. Therefore, we consider $Q$-Wiener processes. To guarantee conservation of mass, we work on rectangular physical domains $\mathcal O:=\mathcal O^x\times \mathcal O^y:=(0,L_x)\times (0,L_y)$, and we prescribe periodic boundary conditions. 
We may consider an ON-basis  $(\g{kl})_{k,l\in\N}$ where the $\g{kl}$ are given as the product $\g{kl}(x,y):=g_k^x(x)g_l^y(y)$ of appropriately scaled eigen-functions of the one-dimensional Laplace operator on $\Ox$ and $\Oy$, respectively (cf.~Remark \ref{rem:basisfunctions}). 

We consider driving noise $\mathbf W$  given by 
\begin{equation}\label{eq:noisevec}
\begin{pmatrix}
\sum_{k,l\in\Z}\lambda^x_{kl}\g{kl}\beta_{kl}^x\\
\sum_{k,l\in\Z} \lambda_{kl}^y \g{kl}\beta_{kl}^y 
\end{pmatrix}
\end{equation} 
where
\begin{itemize}
\item the $\beta_{kl}^\alpha, $ $ \alpha\in\{x,y\}$, $k,l\in\Z$ constitute a family of i.i.d. Brownian motions, 
\item the $\lambda^\alpha_{kl}$, $\alpha\in\{x,y\}$, $k,l\in\Z$ are a family of nonnegative real numbers converging sufficiently fast to zero -- see Hypothesis \ref{item:stochbasis:bounds} for more details.
\end{itemize}
Therefore, we are interested in global existence of a.s. nonnegative martingale solutions to the stochastic thin-film equation
\begin{align}\label{eq:stfe2}
\text{d}u=-\div\tgkla{m\trkla{u}\nabla\trkla{\Delta u-F^\prime\trkla{u}}}\dt+\sum_{\alpha\in\tgkla{x,y}}\sum_{k,l\in\Z}\para{\alpha}\trkla{\sqrt{m\trkla{u}} \lalpha{kl}\g{kl}}\dbalpha{kl}
\end{align}
on $\Om\times[0,\infty)$ subject to periodic boundary conditions.

As the analysis of stochastic thin-film equations is influenced by the deterministic theory, we give a brief account on the literature, following here the exposition in \cite{DareiotisGGG}.

A theory of existence of weak solutions for the deterministic thin-film equation in space dimension $d=1$  has been developed in \cite{BF1990,BBDP1995,BP1996} and \cite{Otto1998,BGK2005,Mellet2015} for zero and nonzero contact angles at the intersection of the liquid-gas and liquid-solid interfaces, respectively, while the multi-dimensional version  with $F \equiv 0$ in $\mathcal{O}\times (0,T] $ and zero contact angles has been the subject of \cite{Dal1998,Grun2004}. For these solutions, a number of quantitative results has been obtained -- including optimal estimates on spreading rates of free boundaries, i.e.~the triple lines separating liquid, gas, and solid, see \cite{HulshofShishkov98, Bertsch1998, Gruen2002, SupportPropagationThinFilm}, optimal conditions on the occurrence of waiting time phenomena \cite{WaitingTime}, as well as scaling laws for the size of waiting times \cite{LowerBounds, UpperBoundsThinFilmWeakSlippage}. For the deterministic case with $F\neq 0$, we refer to \cite{Grun2003} for an existence result based on numerical analysis.

A corresponding theory of classical solutions, giving the existence and uniqueness for initial data close to generic solutions or short times, has been developed in \cite{GKO2008,GK2010,GGKO2014,Gnann2015,Gnann2016,GIM2019} for zero contact angles and in \cite{Knuepfer2011,KM2013,KM2015,Knuepfer2015,Esselborn2016} for nonzero contact angles in one space dimension, while the higher-dimensional version has been the subject of \cite{John2015,Seis2018,GP2018} and \cite{Degtyarev2017} for zero and nonzero contact angles, respectively.

\revy{It is worth recalling that the thin-film equation is one of the very few examples of (degenerate) parabolic fourth-order equations which allow for globally nonnegative solutions. To retain this property also in the stochastic case, neither additive noise nor multiplicative noise not degenerating for $u(x)=0$ seems to be appropriate. The special structure with the noise term in \eqref{eq:stfe} given as the square root of the mobility $m(\cdot)$ has been suggested by the derivation of stochastic thin-film equations, see \cite{DMS2005, Grun2006}.}

The outline of our paper is as follows.
Conceptually, our existence result is based on stochastic counterparts of integral estimates known from the deterministic setting which we combine with Jakubowski/Skorokhod-type methods to construct martingale solutions -- \revx{see  \cite{BrzezniakOndrejat, HofmanovaSeidler} for the basic ideas of this approach and \cite{HofmanovaRoegerRenesse, BreitFeireislHofmanova} for applications to other problems.}
More precisely, we will control the energy 
\begin{align}\label{eq:energy}
\mathcal{E}\trkla{u}:=\iO \tfrac12\tabs{\nabla u}^2\dx\dy +\iO F\trkla{u}\dx\dy
\end{align}
and the so called mathematical entropy
\begin{align}
\mathcal{S}\trkla{u}:=\iO G\trkla{u}\dx\dy
\end{align}
where 
\begin{align}
G\trkla{u}:=\int_1^u\int_1^s \tfrac{1}{m\trkla{r}}\text{d}r\,\text{d}s
\end{align}
is a second primitive of the reciprocal mobility. 

\revy{Considering the case $m\trkla{u}:=u^2$, we mimic the 1d strategy used in \cite{FischerGrun2018}, i.e.~we perform discretization in space and apply \Ito's formula to the resulting system of SDEs to derive those integral estimates.
Taking a step towards the derivation of tractable, fully discrete finite element schemes, we use a basis of finite element functions to perform the spatial discretization.
Although this choice introduces additional mathematical difficulties, we strongly believe that it will serve as a cornerstone for fully discrete schemes.}
Extending this 1D approach (in particular the treatment of the additional terms arising from \Ito's formula) to the two-dimensional setting requires the use of tensor product finite elements.\\
We adapt the stopping-time-approach of \cite{FischerGrun2018} to comply with the singularities in the effective interface potential $F(\cdot)$. In contrast to the spatially one-dimensional setting of \cite{FischerGrun2018}, boundedness of the physical energy $\mathcal E(u)$ does no longer imply strict positivity in our case. Therefore, in the discrete setting, $\mathcal E(u)$ is augmented by the square of the $L^2$-norm of the discrete Laplacian of discrete solutions $u_h$ weighted by a factor which vanishes in the limit $h\to 0$.
For the details, see \eqref{eq:discenergy} and \eqref{eq:defDiscLapl}.

Section~\ref{sec:notation} is devoted to notation and the large number of technical preliminaries which come along with our discretization and the energy regularization mentioned before. Moreover, the assumptions on initial data, growth behavior of the effective interface potential $F$ and on the driving noise $\mathbf W$ are specified in Section~\ref{sec:notation}, too.

Section~\ref{sec:scheme} is devoted to the discussion of the semi-discrete scheme which is formulated in such a way that it may be used for practical numerical simulations of stochastic thin-film equations as well. Section~\ref{sec:scheme} contains also the main existence result together with the applied solution concept. Moreover, we present a lemma which permits the control of the oscillation of discrete solutions on single finite elements.

In Section~\ref{sec:apriori}, we present the core result of the analysis in this paper -- a discrete combined energy-entropy estimate. It is based on a combination of \Ito's formula with inverse estimates for finite-element functions, and with error estimates for interpolation operators which are collected in a synopsis in Appendix \ref{sec:appendix}.
Moreover, Section~\ref{sec:apriori} contains results on compactness in time of discrete solutions which -- in combination with the aforementioned energy-entropy estimate -- are the key to apply the Skorokhod-Jakubowski-method (cf.~\cite{BrzezniakOndrejat,HofmanovaSeidler,Jakubowski1998}) to pass to the limit $h\to 0$ which is the topic of Section~\ref{sec:limit}.

It is worth mentioning that the passage to the limit in the deterministic terms poses new intricacies due to the lack of strict positivity results. We base our arguments on appropriate exhaustion arguments combined with generalizations of Egorov's theorem for Bochner-integrable functions.

\paragraph*{\underline{Notation:}} Throughout the paper, we use the standard notation for Sobolev spaces, i.e.~for a spatial domain $\Om\subset\R^2$, we denote the space of $k$-times weakly differentiable functions with weak derivatives in $L^p\trkla{\Om}$ by $W^{k,p}\trkla{\Om}$. 
For $p=2$, we denote the Hilbert spaces $W^{k,2}\trkla{\Om}$ by $H^k\trkla{\Om}$.
The corresponding subspaces of $\Om$-periodic functions will be denoted by the subscript `per'.
The subspace of $\Om$-periodic $H^1\trkla{\Om}$-functions with mean-value zero will be denoted by $H^1_*\trkla{\Om}$.
Furthermore, we denote the dual space of $H^1_\per\trkla{\Om}$ by $\trkla{H^1_\per\trkla{\Om}}^\prime$.
The space of continuous $\Om$-periodic functions is denoted by $C_\per\trkla{\overline{\Om}}$ and $C_\per^\gamma\trkla{\overline{\Om}}$ is the space of $\Om$-periodic, Hölder continuous functions with Hölder exponent $\gamma$.\\
For a time interval $I$ and a Banach space $X$, the space of $L^p$-integrable functions with values in $X$ is denoted by $L^p\trkla{I; X}$.
Similarly, we denote the space of $k$-times weakly differentiable functions from $I$ to $X$ with weak derivatives in $L^p\trkla{I; X}$ by $W^{k,p}\trkla{I;X}$ and the Hölder continuous functions from $I$ to $X$ with Hölder exponent $\gamma$ by $C^\gamma\trkla{I;X}$.

We shall also use some standard notation from stochastic analysis:
The notation $a\wedge b$ stands for the minimum of $a$ and $b$, and $L_2\trkla{X,Y}$ denotes the set of Hilbert-Schmidt operators from $X$ to $Y$.
For a stopping time $T$, we write $\chi_T$ to denote the ($\omega$-dependent) characteristic function of the time interval $\tekla{0,T}$.

Further notation related to the semi-discrete scheme is introduced in Section \ref{sec:notation}.
\section{Notation, technical preliminaries, and basic assumptions on the data}\label{sec:notation}
We consider the torus $\Om:=\Ox\times\Oy:=\rkla{0,\Lx}\times\rkla{0,\Ly}$. 
We introduce partitions $\Thx$ and $\Thy$ of $\Ox$ and $\Oy$ satisfying the following assumption:
\begin{itemize}
\labitem{\textbf{(S)}}{item:S} $\tgkla{\Thx}_{h>0}$ and $\tgkla{\Thy}_{h>0}$ are families of equidistant partitions of $\Ox$ and $\Oy$, respectively, into disjoint, open intervals such that
\begin{align*}
\overline{\Ox}\equiv\bigcup_{\Kx\in\Thx}\overline{\Kx}&&\text{and}&&\overline{\Oy}\equiv\bigcup_{\Ky\in\Thy}\overline{\Ky}\,.
\end{align*}
In particular, there exist positive constants $\hat{c}_1$, $\hat{C}_2$ such that
\begin{align*}
\hat{c}_1 h\leq \hx,\,\hy\leq\hat{C}_2h
\end{align*}
with $\hx:=\diam \Kx$ ($\Kx\in\Thx$), $\hy:=\diam\Ky$ ($\Ky\in\Thy$), and $h\in\trkla{0,1}$.
\end{itemize}
Combining $\tgkla{\Thx}\h$ and $\tgkla{\Thy}\h$, we obtain a family of partitions $\tgkla{\Qh}\h$ of $\Om$ which is defined via
\begin{align}
\Qh=\tgkla{\Q=\Kx\times\Ky\,:\,\Kx\in\Thx\text{~and~}\Ky\in\Thy}\,.
\end{align}
Based on these partitions, we introduce the following spaces of continuous, piecewise linear finite element functions.
\begin{subequations}\label{eq:def:fespace}
\begin{align}
\Uhx:=&\tgkla{v\in C_{\text{per}}\trkla{\overline{\Ox}}\,:\,\restr{v}{\Kx}\in\mathcal{P}_1\trkla{\Kx}\quad\forall\Kx\in\Thx}\,,\\
\Uhy:=&\tgkla{v\in C_{\text{per}}\trkla{\overline{\Oy}}\,:\,\restr{v}{\Ky}\in\mathcal{P}_1\trkla{\Ky}\quad\forall\Ky\in\Thy}\,,\\
\Uhxy:=&\Uhx\otimes\Uhy\label{eq:def:Uhxy}\,.
\end{align}
\end{subequations}
Imposing periodic boundary conditions, we denote the vertices of $\Thx$ by $\tgkla{x_i}_{i=1,\ldots\dim\Uhx+1}\!=\tgkla{\trkla{i-1}\hx}_{i=1,\ldots\dim\Uhx+1}$ and identify $x_{\dim\Uhx+1}=\Lx$ with $x_1=0$.
Furthermore, we denote the dual basis to these vertices by $\tgkla{\ex{i}}_{i=1,\ldots,\dim\Uhx}$.
Similarly,  we denote the vertices of $\Thy$ by $\tgkla{y_j}_{j=1,\ldots,\dim\Uhy+1}$, identify $y_{\dim\Uhy+1}=\Ly$ with $y_1=0$, and consider their dual basis $\tgkla{\ey{j}}_{j=1,\ldots,\dim\Uhy}$.
Throughout this work, we will also identify $x_0$ with $x_{\dim\Uhx}$ and $y_0$ with $y_{\dim\Uhy}$.
In the same spirit, we shall identify $\ex{\dim\Uhx+1}$ with $\ex{1}$, $\ex{0}$ with $\ex{\dim\Uhx}$, $\ey{\dim\Uhy+1}$ with $\ey{1}$, and $\ey{0}$ with $\ey{\dim\Uhy}$.
For the spaces introduced  in \eqref{eq:def:fespace}, we define the interpolation operators
\begin{align}
\Ihxop\,&:\,C_{\text{per}}\trkla{\overline{\Ox}}\rightarrow\Uhx&a&\mapsto \sum_{i=1}^{\dim\Uhx}a\trkla{x_i}\ex{i}\,,\\
\Ihyop\,&:\,C_{\text{per}}\trkla{\overline{\Oy}}\rightarrow\Uhy&a&\mapsto \sum_{j=1}^{\dim\Uhy}a\trkla{y_j}\ey{j}\,,\\
\Ihxyop\,&:\, C_{\text{per}}\trkla{\overline{\Om}}\rightarrow\Uhxy&a&\mapsto\Ihx{\Ihy{a}}=\Ihy{\Ihx{a}}\,.
\end{align}
For future reference, we state the following norm equivalence for $p\in[1,\infty)$ and $u\h\in\Uhxy$:
\begin{align}\label{eq:normequivalence}
c\rkla{\iO\trkla{u\h}^p\dx\dy}^{1/p}\leq \rkla{\iO\Ihxy{\trkla{u\h}^p}\dx\dy}^{1/p}\leq C\rkla{\iO\trkla{u\h}^p\dx\dy}^{1/p}
\end{align}
with $c,C>0$ independent of $h$.
Similar (lower dimensional) results also hold true for $\Ihxop$ on $\Uhx$ and $\Ihyop$ on $\Uhy$.
These nodal interpolation operators satisfy error estimates similar to the ones established in \cite{Metzger2020} for simplicial elements. 
For the reader's convenience, we collect these estimates in Lemma \ref{lem:Ih:error} in the appendix.

With these interpolation operators, we define the discrete Laplacian \revy{$\Delta\h\,:\,\Uhxy\rightarrow\Uhxy\cap H^1_*\trkla{\Om}$} as follows:
\begin{align}\label{eq:defDiscLapl}
\begin{split}
\iO\Ihxy{-\Delta_h u\h\psi\h}\dx\dy:&=\iO\Ihy{\parx u\h\parx\psi\h}\dx\dy+\iO\Ihx{\pary u\h\pary \psi\h}\dx\dy\\
&=:-\iO\Ihxy{\Delta\h^x u\h \psi\h}\dx\dy-\iO\Ihxy{\Delta\h^y u\h\psi\h}\dx\dy\,
\end{split}
\end{align}
for all $\psi\in \Uhxy$.
\revy{Thereby, the operator $\Delta\h^x$ can be interpreted pointwise in $y$ as the one-dimensional discrete Laplacian w.r.t.~$x$ mapping $\Uhx$ onto $\Uhx\cap H^1_*\trkla{\Ox}$  and $\Delta\h^y$ can be interpreted pointwise in $x$ as the one-dimensional discrete Laplacian w.r.t.~$y$ mapping $\Uhy$ onto $\Uhy\cap H^1_*\trkla{\Oy}$.}\\ \noindent
We denote the forward and backward difference quotients w.r.t.~the spatial coordinates $x$ and $y$ by $\dxp$, $\dxm$, $\dyp$, and $\dym$, i.e.
\begin{subequations}\label{eq:def:dq}
\begin{align}
\dxp f\trkla{x,y}&:=\rkla{f\trkla{x+\hx,y}-f\trkla{x,y}}/\hx\,,\label{eq:def:dq:xp}\\
\dxm f\trkla{x,y}&:=\rkla{f\trkla{x,y}-f\trkla{x-\hx,y}}/\hx\,,\label{eq:def:dq:xm}\\
\dyp f\trkla{x,y}&:=\rkla{f\trkla{x,y+\hy}-f\trkla{x,y}}/\hy\,,\label{eq:def:dq:yp}\\
\dym f\trkla{x,y}&:=\rkla{f\trkla{x,y}-f\trkla{x,y-\hy}}/\hy\,\label{eq:def:dq:ym}
\end{align}
\end{subequations}
(with $f$ extended outside of $\Om$ by periodicity).
Assuming equidistant partitions w.r.t.~$x$ and $y$, the identities
\begin{subequations}\label{eq:dqLap}
\begin{align}
\Delta\h^x v\h&=\dxp\trkla{\dxm v\h}=\dxm\trkla{\dxp v\h}\,,\\
\Delta\h^y v\h&=\dyp\trkla{\dym v\h}=\dym\trkla{\dyp v\h}\,
\end{align}
hold true for $v\h\in\Uhxy$.
\end{subequations}

In addition, we introduce similar local interpolation operators as follows.
We consider the spaces
\begin{align}
C_{\text{per},\Thx}&:=\tgkla{v\in L_{\text{per}}^\infty\trkla{\Ox}\,:\,\restr{v}{\Kx}\in C\trkla{\Kx}\quad\forall\Kx\in\Thx}\,,\\
C_{\text{per},\Thy}&:=\tgkla{v\in L_{\text{per}}^\infty\trkla{\Oy}\,:\,\restr{v}{\Ky}\in C\trkla{\Ky}\quad\forall\Ky\in\Thy}\,
\end{align}
of bounded, piecewise continuous, periodic functions.
As we can extend a continuous function on an open interval to a continuous function on the closure of this interval, we may apply $\Ihxop$ and $\Ihyop$ locally on each element to obtain
\begin{subequations}
\begin{align}
\Ihxlocop\,&:\,C_{\text{per},\Thx}\rightarrow \tgkla{v\in C_{\text{per},\Thx}\,:\,\restr{v}{\Kx}\in\mathcal{P}_1\trkla{\Kx}\quad\forall\Kx\in\Thx}\,,\\
\Ihylocop\,&:\,C_{\text{per},\Thy}\rightarrow \tgkla{v\in C_{\text{per},\Thy}\,:\,\restr{v}{\Ky}\in\mathcal{P}_1\trkla{\Ky}\quad\forall\Ky\in\Thy}\,.
\end{align}
\end{subequations}
Obviously, these local interpolation operators satisfy the identities
\begin{subequations}\label{eq:Ihlocprop}
\begin{align}
\Ihxloc{\parx a^x\h v}&=\parx a^x\h\Ihxloc{v}\,,&\text{and}&&\Ihyloc{\pary a^y\h \hat{v}}&=\pary a^y\h\Ihyloc{\hat{v}}\label{eq:Ihlocpar}\,,\\
\Ihx{\tilde{v}}&=\Ihxloc{\tilde{v}}\,,&\text{and}&&\Ihy{\bar{v}}&=\Ihyloc{\bar{v}}\label{eq:Ih=Ihloc}
\end{align}
\end{subequations}
for $v\in C_{\text{per},\Thx}$, $\hat{v}\in C_{\text{per},\Thy}$, $\tilde{v}\in C_\per\trkla{\overline{\Ox}}$, and $\bar{v}\in C_\per\trkla{\overline{\Oy}}$.
Here, the first identity follows directly from the definition of $\Ihxlocop$, as $\parx a\h$ is constant w.r.t.~$x$ on every element.

In order to allow for a discrete version of the chain rule, we introduce for a continuous function $f\,:\,\R\rightarrow\R$ and $u\in C_\per\trkla{\overline{\Om}}$ for every $y\in\Oy$ and $x\in\trkla{ih,\trkla{i+1}h}=:\Kx\in\Thx$ the function
\begin{align}
\ekla{f\trkla{u}}_x\trkla{x,y}:=\fint_{u\trkla{ih,y}}^{u\trkla{\trkla{i+1}h,y}} f\trkla{s}\ds\,.
\end{align}
Similarly, we define for $x\in\Ox$ and $y\in\trkla{jh,\trkla{j+1}h}=:\Ky\in\Thy$
\begin{align}
\ekla{f\trkla{u}}_y\trkla{x,y}:=\fint_{u\trkla{x,jh}}^{u\trkla{x,\trkla{j+1}h}} f\trkla{s}\ds\,.
\end{align}
Obviously, these definitions provide for $u\h\in\Uhxy$
\begin{align}\label{eq:chainrule}
\parx\Ihxy{f\trkla{u\h}}&=\Ihy{\ekla{f^\prime\trkla{u\h}}_x\parx u\h}\,,&\pary\Ihxy{f\trkla{u\h}}&=\Ihx{\ekla{f^\prime\trkla{u\h}}_y\pary u\h}\,.
\end{align}
On the periodic domain $\Om$, we define the Ritz projection operator $\Ritzop\,:\,H^1_\per\trkla{\Om}\rightarrow \Uhxy$ via
\begin{align}\label{eq:def:Ritz}
\iO\nabla\Ritz{u}\cdot\nabla v\h\dx\dy=\iO\nabla u\cdot\nabla v\dx\dy &&\text{for all~}v\h\in\Uhxy
\end{align}
with the additional constraint $\iO\Ritz{u}\dx\dy=\iO u\dx\dy$.
\newline
\revy{In this publication, we consider the case of a quadratic mobility, i.e.~$m\trkla{u}:=u^2$.}
Let us specify our assumptions on initial data, effective interface potential, and the noise.
\begin{itemize}
\labitem{\textbf{(I)}}{item:initial} Let $\Lambda$ be a probability measure on $H^2_\per\trkla{\Om}$ equipped with the Borel $\sigma$-algebra which is supported on the subset of strictly positive functions such that there is a positive constant $C$ with the property
\begin{align}
\operatorname{esssup}_{v\in\operatorname{supp}\Lambda}\rkla{\mathcal{E}\h\trkla{\Ihxy{v}}+\iO \Ihxy{v}\dx\dy+\rkla{\iO \Ihxy{v}\dx\dy}^{-1}}\leq C
\end{align}
for any $h>0$ with $\mathcal{E}\h$ being a discrete version of the energy \eqref{eq:energy}, which we will define in \eqref{eq:discenergy}.
\end{itemize}
\begin{itemize}
\labitem{\textbf{(P)}}{item:potential}The effective interface potential $F$ has continuous second-order derivatives on $\R^+$ and satisfies for some $p>2$ and $u>0$ the following estimates with appropriate positive constants:
\begin{align*}
F\trkla{u}\geq& c_1 u^{-p}\,,\\
\, \tabs{F^\prime\trkla{u}}\leq&\,\hat{C} u^{-p-1}+\hat{C}\,,\\
\tilde{c}_1 u^{-p-2}-\tilde{c}_2\leq F^{\prime\prime}\trkla{u}\leq&\, \tilde{C} u^{-p-2}+\tilde{C}\,.
\end{align*}
For nonpositive $u$, we define $F(u):=+\infty.$
\end{itemize}
\begin{itemize}
\labitem{\textbf{(B)}}{item:stoch}Let $\trkla{\Omega,\,\mathcal{F},\,\trkla{\mathcal{F}_t}_{t\geq0},\,\Prob}$ be a stochastic basis with a complete, right-continuous filtration such that
\begin{itemize}
\labitem{\textbf{(B1)}}{item:stochbasis:Q} $\bs{W}$ is a $Q$-Wiener process on $\Omega$ adapted to $\trkla{\mathcal{F}_t}_{t\geq0}$ which admits a decomposition of the form $\bs{W}=\sum_{\alpha\in\tgkla{x,y}}\sum_{k,l\in\Z}\lalpha{kl}\g{kl}\kbalpha\balpha{kl}$ for independent sequences of i.i.d.~Brownian motions $\balpha{kl}$ ($\alpha\in\tgkla{x,y}$) and a sequence of sufficiently smooth basis functions $\g{kl}$. 
Here, $\bs{b}_x$ and $\bs{b}_y$ denote the standard Cartesian basis vectors in $\R^2$.
Furthermore, we will denote its components by $W_\alpha:=\sum_{k,l\in\Z}\lalpha{kl}\g{kl}\balpha{kl}$ ($\alpha\in\tgkla{x,y}$). The corresponding components of $Q$ will be denoted by $Q_x$ and $Q_y$.
\labitem{\textbf{(B2)}}{item:stochbasis:initial} there exists a $\mathcal{F}_0$-measurable random variable $u^0$ such that $\Lambda=\Prob\circ \trkla{u^0}^{-1}$.
\labitem{\textbf{(B3)}}{item:stochbasis:bounds} the noise $\bs{W}$ is colored in the sense that $\sum_{k,l=1}^\infty \rkla{\lx{kl}^2+\ly{kl}^2}\norm{\g{kl}}_{W^{2,\infty}\trkla{\Om}}^2\leq C$ for a positive constant $C$.
\end{itemize}
\end{itemize}
\begin{remark}\label{rem:Strato}
  (1) Under natural assumptions on the decay parameters $\lalpha{kl}$ and the basis functions $\g{kl}$, $k,l\in\Z$, Hypothesis~\ref{item:potential} covers in fact also the case that the stochastic integral is to be understood in the sense of Stratonovich.
  Assuming
  \begin{itemize}
  \item the basis functions to be given by $\g{kl}(x,y)=g_k^x(x)g_l^y(y)$ with $g_k^x(\cdot)$ and $g_l^y(\cdot)$   as in \eqref{eq:basisfunctions},
  \item decay parameters $\lalpha{kl}$, $k,l\in\Z$, $\alpha\in\{x,y\}$ to satisfy
  \begin{align}
  \lx{kl}&=\ly{kl}\,, & \lalpha{(-k)l}&=\lalpha{kl}\,, &  \lalpha{k(-l)}&=\lalpha{kl}
  \end{align}
  for all $k,l\in\Z$ and $\alpha\in\tgkla{x,y}$,
    \end{itemize}
    the \Ito~correction of the Stratonovich term
    $$ \sum_{\alpha\in\{x,y\}}\sum_{k,l\in\Z} \partial_\alpha\left(u\lalpha{kl}\g{kl}\right)\circ \dbalpha{kl}$$
    becomes
    \begin{equation}
      \label{eq:itocorrect}
      C_{Strat}\Delta u \dt + \sum_{\alpha\in\{x,y\}}\sum_{k,l\in\Z} \partial_\alpha\left(u\lalpha{kl}\g{kl}\right) \dbalpha{kl}.
    \end{equation}
    Here, the positive constant $C_{Strat} $ is given by
    \begin{equation}
      \label{eq:Cstrat}
      C_{Strat}:=\tfrac{1}{L_xL_y}\left(\lambda_{00}^2+4\sum_{k,l\in\Z\setminus\{0\}}\lambda_{kl}^2+
      2 \sum_{k\in\Z\setminus\{0\}} (\lambda_{k0}^2+\lambda_{0k}^2)\right),
      \end{equation}
    where we omitted the superscript $\alpha$ as the decay parameters were chosen to be independent of $\alpha.$ This allows to write the stochastic thin-film equation with Stratonovich noise in the form
    \begin{equation}
      \label{eq:Stratono}
      du=-\div\left(u^2\nabla\left(\Delta u - F_{Strat}'(u)\right)\right)\dt + \sum_{\alpha\in\{x,y\}}\sum_{k,l\in\Z}\partial_\alpha\left(u\lambda_{kl}\g{kl}\right)\dbalpha{kl}
    \end{equation}
    with the energy $F_{Strat}$ given by
    \begin{equation}
      \label{eq:Stratenergy}
      F_{Strat}(u):= \begin{cases} F(u)+C_{Strat}\left(u-\log u\right) +const. & \text{ if } u>0\\+\infty & \text{ if }u\leq 0\end{cases}
    \end{equation}
    where the constant can be chosen in such a way that $F_{Strat}$ satisfies \ref{item:potential} if \ref{item:potential} is satisfied by $F$ itself. Hence, the analysis presented in this paper applies to the Stratonovich interpretation, too.
    \newline\newline
    \noindent
    (2) The approximation of initial data is based on the nodal interpolation operator to cope with the requirement of strictly positive discrete initial data. Therefore, we need the space of initial data to be continuously embedded in $C(\overline{\mathcal O})$. The specification in \ref{item:initial} that initial data should have $H^2$-regularity is presumably not the optimal one. It is, however, consistent with our regularization procedure -- see \eqref{eq:semidiscreteSTFE:2} -- of augmenting the pressure $p_h$ by a discrete Bi-Laplacian. As the energy estimate formulated for \eqref{eq:stfe2} (see \eqref{eq:gg100}) does not require more than $H^1$-regularity for initial data, it should have been possible to focus on $H^1$- initial data and to apply nodal interpolation operators to appropriate $H^2$-regularizations, e.g. by convolution. For the ease of presentation, we prefer to avoid those technicalities.
\end{remark}

\section{The semi-discrete scheme}\label{sec:scheme}
In order to control the oscillation of the discrete solution $u\h$ on each element, we  regularize the  energy under consideration.
Introducing a regularization parameter $2>\varepsilon>0$, we define the regularized discrete energy and the discrete entropy as
\begin{align}\label{eq:discenergy}
\begin{split}
\mathcal{E}\h\trkla{u\h}:=&\tfrac12\iO \Ihy{\tabs{\parx u\h}^2}+\Ihx{\tabs{\pary u\h}^2}\dx\dy +\iO\Ihxy{F\trkla{u\h}}\dx\dy\\
&+\tfrac12 h^\varepsilon\iO\Ihxy{\tabs{\Delta\h u\h}^2}\dx\dy\,,\end{split}\\
\mathcal{S}\h\trkla{u\h}:=&\iO\Ihxy{G\trkla{u\h}}\dx\dy\qquad\text{with}\qquad G\trkla{s}:=\int_1^s\int_1^r\tfrac{1}{m\trkla{\tau}}\text{d}\tau\text{d}r\,.
\end{align}
As it will be shown in Lemma \ref{lem:oscillation}, we  assume that 
\begin{itemize}
\labitem{\textbf{(R)}}{item:regularization} the regularization parameter $\varepsilon$ is small enough such that there exists a constant $\rho>0$ such that
\begin{align}
1>\frac2p+\frac\varepsilon2+\frac\rho{2p}\,,
\end{align}
where $p$ is the exponent associated with the growth of $F$ (cf.~Assumption \ref{item:potential}).
\end{itemize}
\begin{remark}
For every exponent $p>2$ in the effective interface potential, positive parameters $\varepsilon$ and $\rho$ exist, such that \ref{item:regularization} holds true.
\end{remark}
Given a positive time $\Tmax$, we introduce a threshold energy $\Emax:=\hat{C} h^{-\rho/(2+p)}$ for given $\hat{C}>0$ and $0<\rho<\!\!\!<1$ satisfying \ref{item:regularization}. 
Similarly as in \cite{FischerGrun2018}, we consider associated stopping times $T\h:=\Tmax\wedge\inf\tgkla{t\geq0\,:\,\mathcal{E}\trkla{u\h}\geq \Emax}$.
We approximate the infinite dimensional Wiener process by a finite dimensional noise term.
In particular, we introduce the sets $\Zx\subset\Z$ and $\Zy\subset\Z$ satisfying
\begin{itemize}
\labitem{\textbf{(B3$\bs{^*}$)}}{item:stochbasis:boundthirdderivative} $\sum_{k\in\Zx}\sum_{l\in\Zy}\rkla{\lx{kl}^2+\ly{kl}^2}h^\varepsilon\norm{\g{kl}}_{W^{3,\infty}\trkla{\Om}}^2\leq C$ for $h\searrow0$,
\labitem{\textbf{(B4)}}{item:stochbasis:convergence}$\Zx\subseteq I^x_{\hat{h}}$, $\Zy\subseteq I^y_{\hat{h}}$ for $h\geq\hat{h}$ and $\bigcup_{h>0}\Zx=\Z$ and $\bigcup_{h>0}\Zy=\Z$.
\end{itemize} 
\begin{remark}\label{rem:basisfunctions}
Often the basis functions $\g{kl}$ are assumed to be eigenfunctions of the negative Laplacian on $\Om$ under periodic boundary conditions. In particular, the functions $\g{kl}$ are assumed to be the product of eigenfunctions $g_k^x$ and $g_l^y$ of the one-dimensional Laplacian on $\Ox$ and $\Oy$, respectively, i.e.
\begin{align}\label{eq:basisfunctions}\begin{split}
g_k^x\trkla{x}:=\sqrt{\frac{2}{\Lx}}\left\{\begin{array}{cl}
\cos\rkla{\tfrac{2\pi k x}{\Lx}} & \text{for~} k\geq1\,,\\
\tfrac{1}{\sqrt{2}}&\text{for~}k=0\,,\\
\sin\rkla{\tfrac{2\pi kx}{\Lx}}&\text{for~}k\leq-1\,,
\end{array}\right.
g_l^y\trkla{y}:=\sqrt{\frac{2}{\Ly}}\left\{\begin{array}{cl}
\cos\rkla{\tfrac{2\pi ly}{\Ly}} & \text{for~} l\geq1\,,\\
\tfrac{1}{\sqrt{2}}&\text{for~}l=0\,,\\
\sin\rkla{\tfrac{2\pi ly}{\Ly}}&\text{for~}l\leq-1\,.
\end{array}\right.
\end{split}
\end{align}
In this case, we have $\norm{\g{kl}}_{W^{2,\infty}\trkla{\Om}}\sim \trkla{k^2+l^2}$ and $\norm{\g{kl}}_{W^{3,\infty}\trkla{\Om}}\leq C\trkla{k^3+l^3}$.
Therefore, one may choose $\Zx=\Zy=\tgkla{z\in\Z\,:\, \tabs{z}\leq \hat{C} h^{-\varepsilon/2}}$ for given $\hat{C}>0$ to satisfy Assumptions \ref{item:stochbasis:boundthirdderivative} and \ref{item:stochbasis:convergence}, i.e.~the additional restrictions on the noise term imposed in \ref{item:stochbasis:boundthirdderivative} vanish when passing to the limit $h\searrow0$.
\end{remark}
With the perspective to simplify the implementation in a practical numerical scheme, we approximate the basis functions $\g{kl}$ by $\gt{kl}:=\Ihxy{\g{kl}}$.

In this work, we consider solutions
\begin{subequations}
\begin{align}
u\h&\,\in L^2\trkla{\Omega;C\trkla{\tekla{0,\Tmax};\Uhxy}}\,,\\
p\h&\,\in L^2\trkla{\Omega;L^\infty\trkla{0,\Tmax;\Uhxy}}
\end{align}
\end{subequations} 
to the following regularized, semi-discrete version of \eqref{eq:stfe2}:
\begin{subequations}\label{eq:semidiscreteSTFE}
\begin{align}\label{eq:semidiscreteSTFE:1}
\begin{split}
\iO&\,\Ihxy{u\h\trkla{T}\psi\h}\dx\dy -\iO\Ihxy{u\h\trkla{0}\psi\h}\dx\dy\\
&\qquad+\iOOTTh\Ihy{\meanGinvX{u\h}\parx p\h\parx\psi\h}\dx\dy\dt\\
&\qquad +\iOOTTh\Ihx{\meanGinvY{u\h}\pary p\h\pary\psi\h}\dx\dy\dt\\
=&\,\sum_{k\in \Zx}\sum_{k\in\Zy}\lx{kl}\iOOTTh\Ihy{\Ihxloc{\parx\trkla{u\h\gt{kl}}\psi\h}}\dx\dy\dbx{kl}\\
&\qquad+\sum_{k\in \Zx}\sum_{k\in\Zy}\ly{kl}\iOOTTh\Ihx{\Ihyloc{\pary\trkla{u\h\gt{kl}}\psi\h}}\dx\dy\dby{kl}
\end{split}
\end{align}
\begin{align}\label{eq:semidiscreteSTFE:2}
\begin{split}
\iO\Ihxy{p\h\psi\h}\dx\dy=\chiTh\iO\Ihy{\parx u\h\parx\psi\h}\dx\dy +\chiTh\iO\Ihx{\pary u\h\pary\psi\h}\dx\dy\\
+\chiTh\iO\Ihxy{F^\prime\trkla{u\h}\psi\h}\dx\dy+\chiTh h^\varepsilon\iO\Ihxy{\Delta\h u\h\Delta\h\psi\h}\dx\dy\,.
\end{split}
\end{align}
\end{subequations}
\begin{remark}\label{rem:initialdata}
Note that for discrete solutions of \eqref{eq:semidiscreteSTFE:1} the mass of discrete solutions, i.e. $\int_{\mathcal O} u_h(t)\dx\dy = \int_{\mathcal O} \Ihxy{u_h(t)}\dx\dy, $ is constant in time. Of course, it is natural to choose $\psi_h\equiv 1$ as the test function in \eqref{eq:semidiscreteSTFE:1}. Obviously, the contribution by the elliptic terms vanishes. So
 let us briefly prove that also the stochastic terms become zero. Using \eqref{eq:Ihlocpar} and both $u_h$ and $\gt{kl}$ to be contained in $U_h$, we find
\begin{align*}
&\int_{\mathcal O}\Ihy{\Ihxloc{\partial_x\left(u_h\gt{kl}\right)}}\dx\dy \\
&\qquad =\int_{\mathcal O}\Ihy{\partial_x u_h\Ihxloc{\gt{kl}}+\partial_x\gt{kl}\Ihxloc{u_h}}\dx \dy\\
& \qquad =\int_{\mathcal O}\Ihy{\partial_x\left(u_h\gt{kl}\right)}\dx\dy = 0
\end{align*}
due to integration by parts.
\end{remark}
\begin{definition}\label{def:wms}
Let $\Lambda$ be a probability measure on $H^2_\per\trkla{\Om}$ satisfying \ref{item:initial}.
A triple $\trkla{\trkla{\tilde{\Omega},\tilde{\mathcal{F}},\trkla{\tilde{\mathcal{F}}_t}_{t\geq0},\tilde{\Prob}},\tilde{u},\tilde{\bs{W}}}$ is called a weak martingale solution to the stochastic thin-film equation \eqref{eq:stfe2} with initial data $\Lambda$ on the time interval $\tekla{0,\Tmax}$ provided
\begin{enumerate}
\item $\trkla{\tilde{\Omega},\tilde{\mathcal{F}},\trkla{\tilde{\mathcal{F}}_t}_{t\geq0},\tilde{\Prob}}$ is a stochastic basis with a complete, right-continuous filtration,
\item $\tilde{\bs{W}}$ satisfies Assumption \ref{item:stoch} with respect to $\trkla{\tilde{\Omega},\tilde{\mathcal{F}},\trkla{\tilde{\mathcal{F}}_t}_{t\geq0},\tilde{\Prob}}$,
\item the solution $\tilde{u}$ is element of
\begin{align}
\begin{split}
 L^q\trkla{\tilde{\Omega};L^\infty\trkla{0,\Tmax;H^1_\per\trkla{\Om}}}\cap&\, L^2\trkla{\tilde{\Omega};L^2\trkla{0,\Tmax;H^2_\per\trkla{\Om}}}\\
 &\cap L^\sigma\trkla{\tilde{\Omega};C^{1/4}\trkla{\tekla{0,\Tmax};\trkla{H^1_\per\trkla{\Om}}^\prime}}
 \end{split}
 \end{align} 
 for all $q<\infty$ and $\sigma<8/5$ such that $\sqrt{m\trkla{\tilde{u}}}\nabla\trkla{\Delta\tilde{u}-F^\prime\trkla{\tilde{u}}}\in L^2\trkla{\tekla{\tilde{u}>0}}$,
\item there exists an $\tilde{F}_0$-measurable $H^2_\per\trkla{\Om; \R^+}$-valued random variable $\tilde{u}^0$ such that $\Lambda=\tilde{\Prob}\circ\trkla{\tilde{u}^0}^{-1}$, and the equation
\begin{align}
\begin{split}
\iO \tilde{u}\trkla{t}\phi\dx\dy=&\,\iO\tilde{u}^0\phi\dx\dy+ \int\!\!\int_{\tekla{\tilde{u}>0}} m\trkla{\tilde{u}}\nabla\trkla{\Delta\tilde{u} -F^\prime\trkla{\tilde{u}}}\cdot\nabla\phi\dx\dy\ds\\
&-\sum_{\alpha\in\tgkla{x,y}}\sum_{k,l\in\Z}\lalpha{kl}\int_0^t\iO\sqrt{m\trkla{\tilde{u}}}\g{kl}\para{\alpha}\phi\dx\dy\dbtalpha{kl}
\end{split}
\end{align}
holds true $\tilde{\Prob}$-almost surely for all $t\in\tekla{0,\Tmax}$ and all $\phi\in W^{1,q^*}_\per\trkla{\Om}$ with $q^*>2$.
\end{enumerate}
\end{definition}
The aim of this work is to establish the existence of weak martingale solutions starting from semi-discrete solutions to \eqref{eq:semidiscreteSTFE}.
In particular, we shall prove the following theorem.
\begin{theorem}\label{th:existence}
Let Assumptions \ref{item:S}, \ref{item:initial}, \ref{item:potential}, \ref{item:stoch}, \ref{item:regularization}, \ref{item:stochbasis:boundthirdderivative}, and \ref{item:stochbasis:convergence} be satisfied and let $\Tmax>0$ be given.
Furthermore, let $\trkla{u\h,p\h}_{h\searrow0}$ be a sequence of solutions to the regularized Faedo-Galerkin scheme \eqref{eq:semidiscreteSTFE} for the stochastic thin-film equation \eqref{eq:stfe2} with $\Emax=\hat{C} h^{-\rho/(2+p)}$ for some given $\hat{C}>0$.

Then there exist a stochastic basis $\rkla{\tilde{\Omega},\tilde{\mathcal{F}},\trkla{\tilde{\mathcal{F}}_t}_{t\geq0},\tilde{\Prob}}$ as well as processes $\tilde{u}\h$, $\tilde{J}^x\h$, $\tilde{J}^y\h$, and $\tilde{u}$ such that the following holds:
The processes $\tilde{u}\h$, $\tilde{J}^x\h$, and $\tilde{J}^y\h$ have the same law as the processes $u\h$, $J^x\h:=\Ihy{\sqrt{\meanGinvX{u\h}}\parx p\h}$, and $J^y\h:=\Ihx{\sqrt{\meanGinvY{u\h}}\pary p\h}$ and for a subsequence we $\tilde{\Prob}$-almost surely have the convergences $\tilde{u}\h\rightarrow\tilde{u}$ strongly in $C\trkla{\tekla{0,\Tmax};L^q\trkla{\Om}}\cap L^2\trkla{0,\Tmax;W_\per^{1,q}\trkla{\Om}}$ ($q<\infty$), $\tilde{J}\h^x\weak \tilde{J}^x$ weakly in $L^2\trkla{0,\Tmax;L^2\trkla{\Om}}$, which can be identified with $-\tilde{u}\parx\trkla{\Delta \tilde{u}-F^\prime\trkla{\tilde{u}}}$ on $\tekla{\tilde{u}>0}$, and $\tilde{J}\h^y\weak \tilde{J}^y$ weakly in $L^2\trkla{0,\Tmax;L^2\trkla{\Om}}$, which can be identified with $-\tilde{u}\pary\trkla{\Delta \tilde{u}-F^\prime\trkla{\tilde{u}}}$ on $\tekla{\tilde{u}>0}$.
Furthermore, $\tilde{u}$ is a weak martingale solution to the stochastic thin-film equation in the sense of Definition \ref{def:wms} satisfying the additional bound
\begin{align}
\expectedt{\sup_{t\in\tekla{0,\Tmax}}\trkla{\mathcal{E}\trkla{\tilde{u}}}^\p} +\expectedt{\int\!\!\int_{\tekla{\tilde{u}>0}} m\trkla{\tilde{u}}\tabs{\nabla\trkla{\Delta\tilde{u}-F^\prime\trkla{\tilde{u}}}}^2\dx\dy\dt}\leq C\trkla{u^0,\,\p,\,\Tmax}\label{eq:gg100}
\end{align}
with $\p<\infty$.
In particular, $\tilde \Prob$-almost surely, $\tilde u(\cdot,t)$ is strictly positive for almost all $t\in[0,\Tmax]$.

\end{theorem}

\begin{lemma}\label{lem:oscillation}
Let $u\h\in\Uhxy$ be strictly positive and let $1>\gamma\geq\tfrac2p+\tfrac\varepsilon2+\tfrac\rho{2p}$ and let 
\begin{align*}
\mathcal{E}\h\trkla{u\h} \leq C h^{-\rho/\trkla{2+p}}\,.
\end{align*}
Then, there exists an $h$-independent constant $C_{\text{osc}}>0$ such that the estimate
\begin{align}\label{eq:oscillation}
\frac{u\h\trkla{x_i,y_j}}{u\h\trkla{x_{\hat{i}},y_{\hat{j}}}}\leq C_{\text{osc}}
\end{align}
holds true for all $i\in\tgkla{1,\ldots,\dim\Uhx}$, $j\in\tgkla{1,\ldots,\dim\Uhy}$, $\hat{i}\in\tgkla{i-1,i,i+1}$, and $\hat{j}\in\tgkla{j-1,j,j+1}$.
\end{lemma}

\begin{proof}
Using the standard embedding theorems for Hölder continuous functions and the discrete embedding proven in Corollary \ref{cor:embedding}, we obtain
\begin{align*}
\norm{u\h}_{C^{\gamma}\trkla{\overline{\Om}}}\leq C\norm{\nabla u\h}_{L^q\trkla{\Om}}\leq C \rkla{\norm{u\h}_{H^1\trkla{\Om}} + \norm{\Delta\h u\h}_{L^2\trkla{\Om}}}\leq C \sqrt{h^{-\varepsilon}\mathcal{E}\h\trkla{u\h}}
\end{align*}
with $q<\infty$ large enough.
Furthermore, we have
\begin{align}
\sup_{\trkla{x,y}\in\Om} u\h^{-1}= \trkla{\sup_{\trkla{x,y}\in\Om} u\h^{-p}}^{1/p}\leq C\rkla{h^{-2}\iO\Ihxy{F\trkla{u\h}}\dx\dy}^{1/p}\leq C h^{-2/p}\mathcal{E}\h\trkla{u\h}^{1/p}\,.
\end{align}
Since there exists an element $\Q\in\Qh$ including the vertices $\trkla{x_i,y_j}$ and $\trkla{x_{\hat{i}},y_{\hat{j}}}$ by assumption, we combine the  estimates above and obtain
\begin{align}
\begin{split}
\abs{\frac{u\h\trkla{x_i,y_j}}{u\h\trkla{x_{\hat{i}},y_{\hat{j}}}}-1}&=\abs{\frac{u\trkla{x_i,y_j}-u\h\trkla{x_{\hat{i}},y_{\hat{j}}}}{u\h\trkla{x_{\hat{i}},y_{\hat{j}}}}}\leq C\sup_{\trkla{x,y}\in\Om} u\h^{-1}~ h^\gamma\norm{u\h}_{C^{\gamma}\trkla{\overline{\Om}}}\\
&\leq C h^{-2/p}\mathcal{E}\h\trkla{u\h}^{1/p} h^\gamma h^{-\varepsilon/2}\mathcal{E}\h\trkla{u\h}^{1/2}\\
&=C h^{-2/p-\varepsilon/2 +\gamma} \mathcal{E}\h\trkla{u\h}^{1/p+1/2}\leq C=:C_{\text{osc}}\,,
\end{split}
\end{align}
which completes the proof.
\end{proof}
We will start analyzing scheme \eqref{eq:semidiscreteSTFE} by showing that it admits a solution.
\begin{lemma}\label{lem:existence}
Let $\Tmax$ be a positive real number and $\Emax=\hat{C}h^{-\rho/\trkla{2+p}}$.
Then there exist  stochastic processes $u\h\in L^2\trkla{\Omega;C\trkla{\tekla{0,\Tmax};\Uhxy}}$ and $p\h\in L^2\trkla{\Omega;L^\infty\trkla{0,\Tmax;\Uhxy}}$ as well as associated stopping times $T_h$ such that:
\begin{itemize}
\item Almost surely, we have $T\h=\Tmax\wedge \inf\tgkla{t\in[0,\infty)\,:\,\mathcal{E}\h\trkla{u\h\trkla{\cdot,t}}\geq \Emax}$.
\item Almost surely, the process $p\h$ solves \eqref{eq:semidiscreteSTFE:2} for $t\leq \Tmax$ and is contained in $C\trkla{\tekla{0,T\h};\Uhxy}$.
\item Almost surely, the process $u\h$ solves \eqref{eq:semidiscreteSTFE:1} for $t\leq\Tmax$ and is constant for $t\in\tekla{T\h,\Tmax}$.
\end{itemize}
\end{lemma}
\begin{proof}
As the additional regularization term changes neither the Lipschitz continuity of $\mathcal{E}\h\trkla{u\h}$ w.r.t.~$u\h$ when $\mathcal{E}\h\trkla{u}\leq 2\Emax$ nor the Lipschitz continuous dependence of $p\h$ on $u\h$ when $\mathcal{E}\h\trkla{u\h}\leq 3\Emax$, the result follows along the lines of proof of Lemma 4.2 in \cite{FischerGrun2018}.
\end{proof}

As the solutions $u\h$ are continuous in space and time for $h>0$, the positivity of the initial data immediately provides the positivity of the semidiscrete solutions.
\begin{corollary}
The solutions constructed in Lemma \ref{lem:existence} are strictly positive for all $h>0$.
\end{corollary}

\section{A priori estimates}\label{sec:apriori}
In this section, we shall establish uniform a priori estimates for the semi-discrete solution established in Lemma \ref{lem:existence}.
These results will be used in the next section to pass to the limit $h\searrow0$ and this way to  prove Theorem \ref{th:existence}.
\subsection{The combined energy-entropy estimate}
We start this section by demonstrating that our spatial semidiscretization \eqref{eq:semidiscreteSTFE} satisfies a combined energy-entropy estimate as long as the energy remains below the critical threshold energy $\Emax$ which becomes infinite for $h\searrow0$.
Due to the cut-off mechanism implemented in \eqref{eq:semidiscreteSTFE}, it is possible to extend the results to $\tekla{0,\Tmax}$.
\\
Writing $u\h\trkla{x,y,t}$ as $\sum_{i=1}^{\dim\Uhx}\sum_{j=1}^{\dim\Uhy} u_{ij}\trkla{t}\ex{i}\trkla{x}\ey{j}\trkla{y}$  and choosing $\psi\h\trkla{x,y}=\ex{i}\trkla{x}\ey{j}\trkla{y}$ in \eqref{eq:semidiscreteSTFE:1} provides
\begin{align}\label{eq:semi2}
\begin{split}
\text{d}u_{ij} &+ \chiTh M_{ij}^{-1}\iO\Ihy{\meanGinvX{u\h}\parx p\h\parx\trkla{\ex{i}\trkla{x}\ey{j}\trkla{y}}}\dx\dy\dt\\
&+ \chiTh M_{ij}^{-1}\iO\Ihx{\meanGinvY{u\h}\pary p\h\pary\trkla{\ex{i}\trkla{x}\ey{j}\trkla{y}}}\dx\dy\dt\\
&-\chiTh M_{ij}^{-1}\sum_{k\in\Zx,\,l\in\Zy}\lx{kl}\iO\Ihy{\Ihxloc{\parx\trkla{u\h\gt{kl}}\ex{i}\trkla{x}\ey{j}\trkla{y}}}\dx\dy\dbx{kl}\\
&-\chiTh M_{ij}^{-1}\sum_{k\in\Zx,\,l\in\Zy}\ly{kl}\iO\Ihx{\Ihyloc{\pary\trkla{u\h\gt{kl}}\ex{i}\trkla{x}\ey{j}\trkla{y}}}\dx\dy\dby{kl}\,\revy{=0}
\end{split}
\end{align}
with $M_{ij}=\iO\ex{i}\trkla{x}\ey{j}\trkla{y}\dx\dy$.
As we assume the subdivision to be equidistant, we have $M_{ij}=\hx\hy$ for all $i\in\tgkla{1,\ldots,\dim\Uhx}$ and $j\in\tgkla{1,\ldots,\dim\Uhy}$.

Furthermore, we define for $i\in\tgkla{1,\ldots,\dim\Uhx}$ and $j\in\tgkla{1,\ldots,\dim\Uhy}$
\begin{align}
\begin{split}\label{eq:def:L}
L_{ij}\trkla{t}:=&-\chiTh M_{ij}^{-1}\iO\Ihy{\meanGinvX{u\h}\parx p\h\parx\trkla{\ex{i}\trkla{x}\ey{j}\trkla{y}}}\dx\dy\\
& -\chiTh M_{ij}^{-1}\iO\Ihx{\meanGinvY{u\h}\pary p\h\pary\trkla{\ex{i}\trkla{x}\ey{j}\trkla{y}}}\dx\dy \,,
\end{split}\\
Z_{ij}^x\trkla{\omega}:=&\chiTh M_{ij}^{-1}\sum_{k,l\in\Z} \iO\Ihy{\Ihxloc{\parx\trkla{\skla{\g{kl},\omega}_{L^2} u\h\gt{kl}}\ex{i}\trkla{x}\ey{j}\trkla{y}}}\dx\dy\,,\label{eq:def:Zx}\\
Z_{ij}^y\trkla{\omega}:=&\chiTh M_{ij}^{-1}\sum_{k,l\in\Z} \iO\Ihx{\Ihyloc{\pary\trkla{\skla{\g{kl},\omega}_{L^2} u\h\gt{kl}}\ex{i}\trkla{x}\ey{j}\trkla{y}}}\dx\dy\,.\label{eq:def:Zy}
\end{align}
Here, $\skla{\cdot,\cdot}_{L^2}$ denotes the standard $L^2\trkla{\Om}$ inner product.
With this notation we may rewrite \eqref{eq:semi2} as
\begin{align}
\text{d}u_{ij}=L_{ij}\trkla{t}\dt+\sum_{k\in\Zx,\,l\in\Zy}\rkla{Z_{ij}^x\trkla{\lx{kl}\g{kl}}\dbx{kl} +Z_{ij}^y\trkla{\ly{kl}\g{kl}}\dby{kl}}\,.
\end{align}

For given positive parameters $\alpha$ and $\kappa$, we consider the integral quantity
\revy{\begin{align}
R_{\alpha,\kappa,h}\trkla{u\h\trkla{t}}:=\alpha+\mathcal{E}\h\trkla{u\h\trkla{t}}+\kappa\mathcal{S}\h\trkla{u\h\trkla{t}}\,.
\end{align}
For the ease of presentation, we will often drop the explicit dependence on $u\h$ and abbreviate $R\trkla{t}:=R\trkla{u\h\trkla{t}}:=R_{\alpha,\kappa,h}\trkla{u\h\trkla{t}}$.}
\begin{lemma}\label{lem:ableitungen}
Let $\p\geq1$ be given. The first and second variations of $R\trkla{s}^\p$ \revy{w.r.t.~$u\h$} are given by
\begin{subequations}
\begin{align}
\revy{D\trkla{R\trkla{u\h\trkla{s}}^\p}}=\p R\trkla{s}^{\p-1}\trkla{D\mathcal{E}\h\trkla{u\h\trkla{s}}+\kappa D\mathcal{S}\h\trkla{u\h\trkla{s}}}
\end{align}
and 
\begin{multline}
\revy{D^2\trkla{R\trkla{u\h\trkla{s}}^\p}}=\,\p R\trkla{s}^{\p-1}\trkla{D^2\mathcal{E}\h\trkla{u\h\trkla{s}}+\kappa D^2\mathcal{S}\h\trkla{u\h\trkla{s}}}\\
+\p\trkla{\p-1}R\trkla{s}^{\p-2}\trkla{D\mathcal{E}\h\trkla{u\h\trkla{s}}+\kappa D\mathcal{S}\h\trkla{u\h\trkla{s}}}\otimes\trkla{D\mathcal{E}\h\trkla{u\h\trkla{s}}+\kappa D\mathcal{S}\h\trkla{u\h\trkla{s}}}
\end{multline}
\end{subequations}
with
\begin{align}
D\mathcal{E}\h\trkla{u\h\trkla{s}}\psi\h=&\,\iO\Ihy{\parx u\h\parx\psi\h} +\Ihx{\pary u\h\pary\psi\h}\dx\dy\\& +\iO\Ihxy{F^\prime\trkla{u\h}\psi\h}\dx\dy+h^\varepsilon \iO\Ihxy{\Delta\h u\h\Delta\h\psi\h}\dx\dy\,,\\
\begin{split}
D^2\mathcal{E}\h\trkla{u\h\trkla{s}}\trkla{\phi\h,\psi\h}=&\,\iO \Ihy{\parx\phi\h\parx\psi\h} +\Ihx{\pary\phi\h\pary\psi\h}\dx\dy\\
& +\iO \Ihxy{F^{\prime\prime}\trkla{u\h}\phi\h\psi\h}\dx\dy+h^\varepsilon \iO\Ihxy{\Delta\h\phi\h\Delta\h\psi\h}\dx\dy\,,
\end{split}\\
D\mathcal{S}\h\trkla{u\h\trkla{s}}\psi\h=&\,\iO\Ihxy{G^\prime\trkla{u\h}\psi\h}\dx\dy\,,\\
D^2\mathcal{S}\h\trkla{u\h\trkla{s}}\trkla{\phi\h,\psi\h}=&\,\iO \Ihxy{G^{\prime\prime}\trkla{u\h}\phi\h\psi\h}\dx\dy\,.
\end{align}
\end{lemma}
Applying \Ito's formula, we are able to show the following combined energy-entropy estimate.
\begin{proposition}\label{prop:energyentropyestimate}
Let $\p\geq 1$ be arbitrary and let $\trkla{u\h,p\h}$ be a solution to \eqref{eq:semidiscreteSTFE} for a parameter $h\in\trkla{0,1}$. 
Furthermore, let the Assumptions 
\ref{item:stoch},  \ref{item:stochbasis:boundthirdderivative},  \ref{item:initial}, \ref{item:potential}, \ref{item:regularization}, and  \ref{item:S} hold true.
Then, for sufficiently large $\alpha$ and $\kappa$ depending only on $\trkla{\lx{kl}}_{kl}$, $\trkla{\ly{kl}}_{kl}$, $\p$, and $\Tmax$, there exists a positive, $h$-independent constant $C$ such that
\begin{align}
\begin{split}
&\expected{\sup_{t\in\tekla{0,\Tmax}} R\trkla{t}^\p} + \expected{\int_0^{T\h} R\trkla{s}^{\p-1}\iO\Ihy{\meanGinvX{u\h}\tabs{\parx p\h}^2}\dx\dy\ds}\\
&\quad+\expected{\int_0^{T\h} R\trkla{s}^{\p-1}\iO\Ihx{\meanGinvY{u\h}\revy{\tabs{\pary p\h}^2}}\dx\dy\ds}\\
&\quad +\expected{\int_0^{T\h} R\trkla{s}^{\p-1}\norm{\Delta\h u\h}\h^2\ds} +\expected{\int_0^{T\h} R\trkla{s}^{\p-1}h^\varepsilon\iO\Ihy{\tabs{\parx\Delta\h u\h}^2}\dx\dy\ds}\\
&\quad+\expected{\int_0^{T\h} R\trkla{s}^{\p-1}h^\varepsilon\iO\Ihx{\tabs{\pary\Delta\h u\h}^2}\dx\dy\ds} \\
&\quad+ \expected{\int_0^{T\h} R\trkla{s}^{\p-1}\iO\Ihy{\meanspX{u\h}\tabs{\parx u\h}^2}\dx\dy\ds}\\
&\quad+ \expected{\int_0^{T\h} R\trkla{s}^{\p-1}\iO\Ihx{\meanspY{u\h}\tabs{\pary u\h}^2}\dx\dy\ds}
 \leq C\,.
 \end{split}
\end{align}
\end{proposition}
\begin{proof}
Using the notation
\begin{align}
\varphi\h\trkla{t}:=\varphi\h\trkla{x,y,t}:=\sum_{i=1}^{\dim\Uhx}\sum_{j=1}^{\dim\Uhy} L_{ij}\trkla{t}\ex{i}\trkla{x}\ey{j}\trkla{y}
\end{align}
and
\begin{align}
\Phi\h\trkla{t}\trkla{\trkla{\omega_x,\,\omega_y}^T}:=\Phi\h\trkla{x,y,t}\trkla{\trkla{\omega_x,\,\omega_y}^T}:=\sum_{i=1}^{\dim\Uhx}\sum_{j=1}^{\dim\Uhy}\trkla{Z_{ij}^x\trkla{\omega_x}+Z_{ij}^y\trkla{\omega_y}}\ex{i}\trkla{x}\ey{j}\trkla{y}\,,
\end{align}
we may rewrite \eqref{eq:semidiscreteSTFE} as
\begin{align}
\text{d}u\h=\varphi\h\trkla{t}\dt+\Phi\h\trkla{t}\trkla{\text{d}\bs{W}_{Q,h}}
\end{align}
with \begin{align}
\bs{W}_{Q,h}:=\sum_{\alpha\in\tgkla{x,\,y}}\sum_{k\in\Zx,\,l\in\Zy}\lalpha{kl}\g{kl}\balpha{kl}\kbalpha\,,
\end{align}
where $\bs{b}_x$, $\bs{b}_y$ denote the standard Cartesian basis vectors of $\R^2$. 
Applying \Ito's formula, we compute
\begin{align}
&R\trkla{t\wedge T\h}^\p=R\trkla{0}^\p+\itTh\p R\trkla{s}^{\p-1}\trkla{D\mathcal{E}\h+\kappa D\mathcal{S}\h}\varphi\h\trkla{s}\ds\nonumber\\
&\quad+\itTh\p R\trkla{s}^{\p-1}\trkla{D\mathcal{E}\h+\kappa D\mathcal{S}\h}\Phi\h\trkla{s}\text{d}\boldsymbol{W}_{Q,h}\nonumber\\
&\quad+\tfrac12 \sum_{k\in\Zx,\,l\in\Zy}\itTh\p R\trkla{s}^{\p-1}\trkla{D^2\mathcal{E}\h+\kappa D^2\mathcal{S}\h}\trkla{\Phi_{h,kl}^x,\Phi_{h,kl}^x}\ds\nonumber\\
&\quad+\tfrac12 \sum_{k\in\Zx,\,l\in\Zy}\itTh\p R\trkla{s}^{\p-1}\trkla{D^2\mathcal{E}\h+\kappa D^2\mathcal{S}\h}\trkla{\Phi_{h,kl}^y,\Phi_{h,kl}^y}\ds\nonumber\\
&\quad+\tfrac12\sum_{k\in\Zx,\,l\in\Zy}\itTh\p\trkla{\p-1}R\trkla{s}^{\p-2}\trkla{D\mathcal{E}\h+\kappa D\mathcal{S}\h}\otimes\trkla{D\mathcal{E}\h+\kappa D\mathcal{S}\h}\trkla{\Phi_{h,kl}^x,\Phi_{h,kl}^x}\ds\nonumber\\
&\quad+\tfrac12\sum_{k\in\Zx,\,l\in\Zy}\itTh\p\trkla{\p-1}R\trkla{s}^{\p-2}\trkla{D\mathcal{E}\h+\kappa D\mathcal{S}\h}\otimes\trkla{D\mathcal{E}\h+\kappa D\mathcal{S}\h}\trkla{\Phi_{h,kl}^y,\Phi_{h,kl}^y}\ds\nonumber\\
&=:\, R\trkla{0}^\p+ I +II +III +IV+V+VI\,,
\end{align}
where we used the abbreviations 
\begin{subequations}
\begin{align}
&&\Phi_{h,kl}^x\trkla{s}&:=\Phi\h\trkla{s}\trkla{\trkla{\lx{kl}\g{kl},0}^T}=\sum_{i=1}^{\dim\Uhx}\sum_{j=1}^{\dim\Uhy} Z_{ij}^x\trkla{\lx{kl}\g{kl}}\ex{i}\trkla{x}\ey{j}\trkla{y}\\
\text{and}&&\Phi_{h,kl}^y\trkla{s}&:=\Phi\h\trkla{s}\trkla{\trkla{0,\ly{kl}\g{kl}}^T}=\sum_{i=1}^{\dim\Uhx}\sum_{j=1}^{\dim\Uhy} Z_{ij}^y\trkla{\ly{kl}\g{kl}}\ex{i}\trkla{x}\ey{j}\trkla{y}\,.
\end{align}
\end{subequations}
Using Lemma \ref{lem:ableitungen}, we compute
\begin{align}
I=&\,\itTh\p R\trkla{s}^{\p-1}\sum_{i=1}^{\dim\Uhx}\sum_{j=1}^{\dim\Uhy}\iO\Ihy{\parx u\h\parx\ex{i}\trkla{x}\ey{j}\trkla{y}}\dx\dy L_{ij}\trkla{s}\ds\nonumber\\
&+\itTh\p R\trkla{s}^{\p-1}\sum_{i=1}^{\dim\Uhx}\sum_{j=1}^{\dim\Uhy}\iO\Ihx{\pary u\h\pary\ey{j}\trkla{y}\ex{i}\trkla{x}}\dx\dy L_{ij}\trkla{s}\ds\nonumber\\
&+\itTh\p R\trkla{s}^{\p-1} \sum_{i=1}^{\dim\Uhx}\sum_{j=1}^{\dim\Uhy}\iO\Ihxy{F^\prime\trkla{u\h}\ex{i}\trkla{x}\ey{j}\trkla{y}}\dx\dy L_{ij}\trkla{s}\ds\nonumber\\
&+h^\varepsilon\itTh \p R\trkla{s}^{\p-1}\sum_{i=1}^{\dim\Uhx}\sum_{j=1}^{\dim\Uhy}\iO\Ihxy{\Delta\h u\h \Delta\h \trkla{\ex{i}\trkla{x}\ey{j}\trkla{y}}}\dx\dy L_{ij}\trkla{s}\ds\nonumber\\
&+\kappa\itTh\p R\trkla{s}^{\p-1}\sum_{i=1}^{\dim\Uhx}\sum_{j=1}^{\dim\Uhy}\iO\Ihxy{G^\prime\trkla{u\h}\ex{i}\trkla{x}\ey{j}\trkla{y}}\dx\dy L_{ij}\trkla{s}\ds\nonumber\\
=:&\, I_a+I_b+I_c+I_d + I_e\label{eq:energyentropy:decomposition}\,.
\end{align}
From \eqref{eq:def:L} we obtain by straightforward computations
\begin{multline}\label{eq:tmp:Lij}
\iO\Ihxy{w\sum_{i=1}^{\dim\Uhx}\sum_{j=1}^{\dim\Uhy} L_{ij}\trkla{s}\ex{i}\trkla{x}\ey{j}\trkla{y}}\dx\dy\\
=\sum_{i=1}^{\dim\Uhx}\sum_{j=1}^{\dim\Uhy}M_{ij} L_{ij} w\trkla{\trkla{i-1}\hx,\trkla{j-1}\hy}\\
=-\chiTh \iO\Ihy{\meanGinvX{u\h}\parx p\h\parx\Ihxy{w}}+\Ihx{\meanGinvY{u\h}\pary p\h\pary\Ihxy{w}}\dx\dy
\end{multline}
for $w\in C_\per\trkla{\overline{\Om}}$ and therefore after integration by parts and using \eqref{eq:semidiscreteSTFE:2}
\begin{align}
\begin{split}
&I_a+I_b+I_c+I_d \\
&\quad=\itTh \p R\trkla{s}^{\p-1}\sum_{i=1}^{\dim\Uhx}\sum_{j=1}^{\dim\Uhy}\iO\Ihxyop\{\rkla{-\Delta\h u\h + F^\prime\trkla{u\h} +h^\varepsilon\Delta\h\Delta\h u\h}\\
&\qquad\qquad\qquad\qquad\qquad\qquad\qquad\qquad\qquad\qquad\qquad\times L_{ij}\trkla{s}\ex{i}\trkla{x}\ey{j}\trkla{y}\}\dx\dy\ds\\
&\quad=\itTh \p R\trkla{s}^{\p-1}\sum_{i=1}^{\dim\Uhx}\sum_{j=1}^{\dim\Uhy}\iO\Ihxy{p\h L_{ij}\trkla{s}\ex{i}\trkla{x}\ey{j}\trkla{y}}\dx\dy\ds\\
&\quad=-\p\itTh\!\! R\trkla{s}^{\p-1}\!\iO\!\Ihy{\meanGinvX{u\h}\tabs{\parx p\h}^2}\!+\!\Ihx{\meanGinvY{u\h}\tabs{\pary p\h}^2}\dx\dy\ds\,.
\end{split}
\end{align}
Similarly, we use \eqref{eq:tmp:Lij} with $w=G^\prime\trkla{u\h}$ and \eqref{eq:chainrule} to compute
\begin{align}
\begin{split}
\iO&\Ihxy{G^\prime\trkla{u\h}\sum_{i=1}^{\dim\Uhx}\sum_{j=1}^{\dim\Uhy} L_{ij}\trkla{s}\ex{i}\trkla{x}\ey{j}\trkla{y}}\dx\dy\\
&=-\chiTh\iO\Ihy{\meanGinvX{u\h}\parx p\h\parx \Ihxy{G^\prime\trkla{u\h}}} \\
&\quad+\Ihx{\meanGinvY{u\h}\pary u\h\pary \Ihxy{G^\prime\trkla{u\h}}}\dx\dy\\
&=-\chiTh \iO\Ihy{\parx u\h\parx p\h}+\Ihx{\pary u\h\pary p\h}\dx\dy\\
&=\chiTh\iO\Ihxy{\Delta_h u\h p\h}\dx\dy\\
&=-\chiTh\norm{\Delta\h u\h}^2\h-\chiTh\iO\Ihy{\parx u\h\parx \Ihxy{F^\prime\trkla{u\h}}}\dx\dy\\
&\quad+\chiTh\iO\Ihx{\pary u\h\pary \Ihxy{F^\prime\trkla{u\h}}}\dx\dy+\chiTh h^\varepsilon \iO\Ihxy{\Delta\h u\h\Delta\h\Delta\h u\h}\dx\dy\\
&=-\chiTh \norm{\Delta\h u\h}^2\h -\chiTh\iO\Ihy{\meanFX{u\h}\tabs{\parx u\h}^2}+\Ihx{\meanFY{u\h}\tabs{\pary u\h}^2}\dx\dy\\
&\quad-\chiTh h^\varepsilon\iO\Ihy{\tabs{\parx\Delta\h u\h}^2} +\Ihx{\tabs{\pary\Delta\h u\h}^2}\dx\dy\,.
\end{split}
\end{align}
From Assumption \ref{item:potential}, we obtain the estimate $\meanFX{u\h}\geq \tilde{c}_1 \meanspX{u\h} -\tilde{c}_2$ and \linebreak$\meanFY{u\h}\geq \tilde{c}_1 \meanspY{u\h} -\tilde{c}_2$.
Therefore, combining the above estimates we conclude
\begin{align}
\begin{split}
I_e\leq&\, -\kappa\p\itTh R\trkla{s}^{\p-1}\norm{\Delta\h u\h}\h^2\ds \\
&-\kappa\p h^\varepsilon\itTh R\trkla{s}^{\p-1} \iO\Ihy{\tabs{\parx\Delta\h u\h}^2} + \Ihx{\tabs{\pary\Delta\h u\h}^2}\dx\dy\\
&-\tilde{c}_1\kappa\p\itTh R\trkla{s}^{\p-1}\iO\Ihy{\meanspX{u\h}\tabs{\parx u\h}^2}+\Ihx{\meanspY{u\h}\tabs{\pary u\h}^2}\dx\dy\ds\\
&+\tilde{c}_2\kappa\p\itTh R\trkla{s}^\p\ds\,.
\end{split}
\end{align}
Noting
\begin{align}
\begin{split}
D^2\mathcal{E}\h\trkla{\Phi^x_{h,kl}, \Phi^x_{h,kl}}=&\,\iO\Ihy{\abs{\sum_{i=1}^{\dim\Uhx}\sum_{j=1}^{\dim\Uhy} Z_{ij}^x\trkla{\lx{kl}\g{kl}}\parx\ex{i}\trkla{x}\ey{j}\trkla{y}}^2}\dx\dy\\
&+\iO\Ihx{\abs{\sum_{i=1}^{\dim\Uhx}\sum_{j=1}^{\dim\Uhy} Z_{ij}^x\trkla{\lx{kl}\g{kl}}\ex{i}\trkla{x}\pary\ey{j}\trkla{y}}^2}\dx\dy\\
&+\iO\Ihxy{F^{\prime\prime}\trkla{u\h}\sum_{i=1}^{\dim\Uhx}\sum_{j=1}^{\dim\Uhy}\tabs{Z_{ij}^x\trkla{\lx{kl}\g{kl}}}^2\ex{i}\trkla{x}\ey{j}\trkla{y}}\dx\dy\\
&+h^\varepsilon\iO\Ihxy{\abs{\Delta\h\sum_{i=1}^{\dim\Uhx}\sum_{j=1}^{\dim\Uhy} Z_{ij}^x\trkla{\lx{kl}\g{kl}}\ex{i}\trkla{x}\ey{j}\trkla{y}}^2}\dx\dy\,,
\end{split}\\
D^2\mathcal{S}\h\trkla{\Phi^x_{h,kl},\Phi^x_{h,kl}}=&\,\iO\Ihxy{G^{\prime\prime}\trkla{u\h}\sum_{i=1}^{\dim\Uhx}\sum_{j=1}^{\dim\Uhy}\tabs{Z^x_{ij}\trkla{\lx{kl}\g{kl}}}^2\ex{i}\trkla{x}\ey{j}\trkla{y}}\dx\dy\,,
\end{align}
and similar identities for $D^2\mathcal{E}\h\trkla{\Phi^y_{h,kl}, \Phi^y_{h,kl}}$ and $D^2\mathcal{S}\h\trkla{\Phi^y_{h,kl},\Phi^y_{h,kl}}$, we combine $III$ and $IV$ and obtain after reordering
\begin{align}
&\tfrac12\!\!\!\sum_{k\in\Zx,\,l\in\Zy}\itTh\!\!\p R\trkla{s}^{\p-1}\iO\Ihy{\abs{\sum_{i=1}^{\dim\Uhx}\sum_{j=1}^{\dim\Uhy} Z_{ij}^x\trkla{\lx{kl}\g{kl}}\parx\ex{i}\trkla{x}\ey{j}\trkla{y}}^2}\dx\dy\ds\nonumber\\
&+\tfrac12\!\!\!\sum_{k\in\Zx,\,l\in\Zy}\itTh\!\!\p R\trkla{s}^{\p-1}\iO\Ihx{\abs{\sum_{i=1}^{\dim\Uhx}\sum_{j=1}^{\dim\Uhy} Z_{ij}^x\trkla{\lx{kl}\g{kl}}\pary\ey{j}\trkla{y}\ex{i}\trkla{x}}^2}\dx\dy\ds\nonumber\\
&+\tfrac12\!\!\!\sum_{k\in\Zx,\,l\in\Zy}\itTh\!\!\p R\trkla{s}^{\p-1}\iO\Ihy{\abs{\sum_{i=1}^{\dim\Uhx}\sum_{j=1}^{\dim\Uhy} Z_{ij}^y\trkla{\ly{kl}\g{kl}}\parx\ex{i}\trkla{x}\ey{j}\trkla{y}}^2}\dx\dy\ds\nonumber\\
&+\tfrac12\!\!\!\sum_{k\in\Zx,\,l\in\Zy}\itTh\!\!\p R\trkla{s}^{\p-1}\iO\Ihx{\abs{\sum_{i=1}^{\dim\Uhx}\sum_{j=1}^{\dim\Uhy} Z_{ij}^y\trkla{\ly{kl}\g{kl}}\pary\ey{j}\trkla{y}\ex{i}\trkla{x}}^2}\dx\dy\ds\nonumber\\
&+\tfrac12\!\!\!\sum_{k\in\Zx,\,l\in\Zy}\itTh\!\!\p R\trkla{s}^{\p-1}\iO\Ihxy{F^{\prime\prime}\trkla{u\h}\sum_{i=1}^{\dim\Uhx}\sum_{j=1}^{\dim\Uhy}\abs{Z_{ij}^x\trkla{\lx{kl}\g{kl}}}^2\ex{i}\trkla{x}\ey{j}\trkla{y}}\dx\dy\ds\nonumber\\
&+\tfrac12\!\!\!\sum_{k\in\Zx,\,l\in\Zy}\itTh\!\!\p R\trkla{s}^{\p-1}\iO\Ihxy{F^{\prime\prime}\trkla{u\h}\sum_{i=1}^{\dim\Uhx}\sum_{j=1}^{\dim\Uhy}\abs{Z_{ij}^y\trkla{\ly{kl}\g{kl}}}^2\ex{i}\trkla{x}\ey{j}\trkla{y}}\dx\dy\ds\nonumber\\
&+\tfrac12h^\varepsilon\!\!\!\sum_{k\in\Zx,\,l\in\Zy}\itTh\!\!\p R\trkla{s}^{\p-1}\iO\Ihxy{\abs{\Delta\h\sum_{i=1}^{\dim\Uhx}\sum_{j=1}^{\dim\Uhy} Z_{ij}^x\trkla{\lx{kl}\g{kl}}\ex{i}\trkla{x}\ey{j}\trkla{y}}^2}\dx\dy\ds\nonumber\\
&+\tfrac12h^\varepsilon\!\!\!\sum_{k\in\Zx,\,l\in\Zy}\itTh\!\!\p R\trkla{s}^{\p-1}\iO\Ihxy{\abs{\Delta\h\sum_{i=1}^{\dim\Uhx}\sum_{j=1}^{\dim\Uhy} Z_{ij}^y\trkla{\ly{kl}\g{kl}}\ex{i}\trkla{x}\ey{j}\trkla{y}}^2}\dx\dy\ds\nonumber\\
&+\tfrac12\kappa\!\!\!\sum_{k\in\Zx,\,l\in\Zy}\itTh\!\!\p R\trkla{s}^{\p-1}\iO\Ihxy{G^{\prime\prime}\trkla{u\h}\sum_{i=1}^{\dim\Uhx}\sum_{j=1}^{\dim\Uhy}\abs{Z_{ij}^x\trkla{\lx{kl}\g{kl}}}^2\ex{i}\trkla{x}\ey{j}\trkla{y}}\dx\dy\ds\nonumber\\
&+\tfrac12\kappa\!\!\!\sum_{k\in\Zx,\,l\in\Zy}\itTh\!\!\p R\trkla{s}^{\p-1}\iO\Ihxy{G^{\prime\prime}\trkla{u\h}\sum_{i=1}^{\dim\Uhx}\sum_{j=1}^{\dim\Uhy}\abs{Z_{ij}^y\trkla{\ly{kl}\g{kl}}}^2\ex{i}\trkla{x}\ey{j}\trkla{y}}\dx\dy\ds\nonumber\\
&= III_a+III_b+III_c+III_d+III_e+III_f+III_g+III_h+III_i+III_j\,.
\end{align}
To derive an estimate for $III_a$, we adapt the ideas of \cite{FischerGrun2018}.
Using the periodicity, the specific form of the one dimensional stiffness matrix on equidistant meshes and $M_{ij}=\hx\hy$ for all $i=1,\ldots,\dim\Uhx$ and $j=1,\ldots,\dim\Uhy$, we compute using \eqref{eq:Ihlocpar} and \eqref{eq:Ih=Ihloc}
\begin{align}
\iO&\Ihy{\abs{\sum_{i=1}^{\dim\Uhx}\sum_{j=1}^{\dim\Uhy} Z_{ij}^x\trkla{\lx{kl}\g{kl}}\parx\ex{i}\trkla{x}\ey{j}\trkla{y}}^2}\dx\dy\nonumber\\
&=\int_{\Oy}\sum_{j=1}^{\dim\Uhy}\int_{\Ox}\abs{\sum_{i=1}^{\dim\Uhx} Z_{ij}^x\trkla{\lx{kl}\g{kl}}\parx\ex{i}\trkla{x}}^2\dx\,\ey{j}\trkla{y}\dy\nonumber\\
&=\sum_{i=1}^{\dim\Uhx}\sum_{j=1}^{\dim\Uhy}\int_{\Oy}\trkla{2\trkla{Z_{ij}^x}^2-Z_{i-1,j}^xZ_{ij}^x-Z_{ij}^xZ_{i+1,j}^x}\trkla{\lx{kl}\g{kl}}\ey{j}\trkla{y}\dy \hx^{-1}\nonumber\\
&=\sum_{i=1}^{\dim\Uhx}\sum_{j=1}^{\dim\Uhy}\int_{\Oy}\trkla{Z_{i+1,j}^x-Z_{ij}^x}^2\trkla{\lx{kl}\g{kl}} \ey{j}\trkla{y}\dy\,\hx^{-1}\nonumber\\
&=\chiTh\sum_{i=1}^{\dim\Uhx}\sum_{j=1}^{\dim\Uhy} \rkla{\iO\Ihy{\Ihxloc{\parx\trkla{\lx{kl} u\h\gt{kl}} \tfrac{\ex{i+1}\trkla{x}-\ex{i}\trkla{x}}{\hx} }\ey{j}\trkla{y}}\dx\dy}^2 \hx^{-1}\hy^{-1}\nonumber\\
&\leq2\chiTh\sum_{i=1}^{\dim\Uhx}\sum_{j=1}^{\dim\Uhy}\rkla{\iO\Ihy{\lx{kl}\parx u\h\Ihx{\gt{kl}\tfrac{\ex{i+1}\trkla{x}-\ex{i}\trkla{x}}{\hx}}\ey{j}\trkla{y}}\dx\dy}^2\hx^{-1}\hy^{-1}\nonumber\\
&\quad+2\chiTh\sum_{i=1}^{\dim\Uhx}\sum_{j=1}^{\dim\Uhy}\rkla{\iO\Ihy{\lx{kl}\parx \gt{kl}\Ihx{u\h\tfrac{\ex{i+1}\trkla{x}-\ex{i}\trkla{x}}{\hx}}\ey{j}\trkla{y}}\dx\dy}^2\hx^{-1}\hy^{-1}\nonumber\\
&=\trkla{*}\,.
\end{align}
Recalling $\ex{i+1}\trkla{x}=\ex{i}\trkla{x-\hx}$ and performing a discrete integration by parts (cf.~Lemma \ref{lem:disc:pi} in the appendix), we continue with the estimate
\begin{align}
\begin{split}
\trkla{*}&\leq\chiTh C\sum_{i=1}^{\dim\Uhx}\sum_{j=1}^{\dim\Uhy}\rkla{\lx{kl}\iO\Ihy{\parx u\h\Ihx{\dxm \gt{kl}\ex{i+1}\trkla{x}}\ey{j}\trkla{y}}\dx\dy}^2\hx^{-1}\hy^{-1}\\
&+\chiTh C\sum_{i=1}^{\dim\Uhx}\sum_{j=1}^{\dim\Uhy}\rkla{\lx{kl}\iO\Ihy{\dxp \parx u\h\Ihx{\gt{kl}\ex{i}\trkla{x}}\ey{j}\trkla{y}}\dx\dy}^2\hx^{-1}\hy^{-1}\\
&+\chiTh C\sum_{i=1}^{\dim\Uhx}\sum_{j=1}^{\dim\Uhy}\rkla{\lx{kl}\iO\Ihy{\parx \gt{kl}\Ihx{\dxm u\h\ex{i+1}\trkla{x}}\ey{j}\trkla{y}}\dx\dy}^2\hx^{-1}\hy^{-1}\\
&+\chiTh C\sum_{i=1}^{\dim\Uhx}\sum_{j=1}^{\dim\Uhy}\rkla{\lx{kl}\iO\Ihy{\dxp \parx \gt{kl}\Ihx{u\h\ex{i}\trkla{x}}\ey{j}\trkla{y}}\dx\dy}^2\hx^{-1}\hy^{-1}\\
\leq&\,\chiTh C\lx{kl}^2\norm{\parx\g{kl}}_{L^\infty\trkla{\Om}}^2 \iO\Ihy{\tabs{\parx u\h}^2}\dx\dy\\
&+\chiTh C\lx{kl}^2\norm{\g{kl}}_{L^\infty\trkla{\Om}}^2 \iO\Ihy{\tabs{\Delta\h^x u\h}^2}\dx\dy\\
&+\chiTh C\lx{kl}^2\norm{\parx\g{kl}}_{L^\infty\trkla{\Om}}^2 \iO\Ihxy{\tabs{\dxm u\h}^2}\dx\dy\\
&+\chiTh C\lx{kl}^2\norm{\dxp \parx\gt{kl}}_{L^\infty\trkla{\Om}}^2 \iO\Ihxy{u\h^2}\dx\dy\\
\leq&\, \chiTh C \lx{kl}^2\norm{\parx\parx\g{kl}}_{L^\infty\trkla{\Om}}^2 R\trkla{s} +\chiTh C\lx{kl}^2\norm{\g{kl}}_{L^\infty\trkla{\Om}}^2 \iO\Ihxy{\tabs{\Delta\h^x u\h}^2}\dx\dy\,.
\end{split}
\end{align}
In the last step, we used Poincar\'e's inequality and the pathwise conservation of $\iO u\h\dx\dy$ (see Remark~\ref{rem:initialdata},  Assumption \ref{item:initial}, and the norm equivalence \eqref{eq:normequivalence}), as well as
\begin{multline}\label{eq:estimate:g:W2}
\norm{\dxp\parx\Ihx{\g{kl}}}_{L^\infty\trkla{\Om}}=\norm{\parx\Ihx{\dxp\g{kl}}}_{L^\infty\trkla{\Om}}\leq \norm{\parx\dxp\g{kl}}_{L^\infty\trkla{\Om}}\\
\leq \norm{\parx\parx\g{kl}}_{L^\infty\trkla{\Om}}\leq\norm{\g{kl}}_{W^{2,\infty}\trkla{\Om}}\,,
\end{multline}
which follows from \ref{item:S} and the stability of the nodal interpolation operator.
This provides
\begin{align}
\begin{split}
III_a\leq&\, C\p\sum_{k\in\Zx,\,l\in\Zy}\lx{kl}^2\norm{\g{kl}}_{W^{2,\infty}\trkla{\Om}}^2\itTh R\trkla{s}^\p \ds\\
&+C\p\sum_{k\in\Zx,\,l\in\Zy}\lx{kl}^2\norm{\g{kl}}_{L^\infty\trkla{\Om}}^2\itTh R\trkla{s}^{\p-1} \iO\Ihxy{\tabs{\Delta\h^x u\h}^2}\dx\dy \ds\,.
\end{split}
\end{align}
Similar computations show
\begin{align}
\begin{split}
III_d\leq&\, C\p\sum_{k\in\Zx,\,l\in\Zy}\ly{kl}^2\norm{\g{kl}}_{W^{2,\infty}\trkla{\Om}}^2\itTh R\trkla{s}^\p \ds\\
&+C\p\sum_{k\in\Zx,\,l\in\Zy}\ly{kl}^2\norm{\g{kl}}_{L^\infty\trkla{\Om}}^2\itTh R\trkla{s}^{\p-1} \iO\Ihxy{\tabs{\Delta\h^y u\h}^2}\dx\dy \ds\,.
\end{split}
\end{align}
To control $III_b$, we compute
\begin{align}
\iO&\Ihx{\abs{\sum_{i=1}^{\dim\Uhx}\sum_{j=1}^{\dim\Uhy} Z_{ij}^x\trkla{\lx{kl}\g{kl}}\pary\ey{j}\trkla{y}\ex{i}\trkla{x}}^2}\dx\dy\nonumber\\
&=\int_{\Ox}\Ihx{\sum_{i=1}^{\dim\Uhx}\int_{\Oy}\abs{\sum_{j=1}^{\dim\Uhy} Z_{ij}^x\trkla{\lx{kl}\g{kl}}\pary\ey{j}\trkla{y}}^2\dy\,\ex{i}\trkla{x}}\dx \nonumber\\
&=\sum_{i=1}^{\dim\Uhx}\sum_{j=1}^{\dim\Uhy}\trkla{Z_{i,j+1}^x-Z_{i,j}^x}^2\trkla{\lx{kl}\g{kl}}\hx\hy^{-1}\nonumber\\
&=\chiTh\sum_{i=1}^{\dim\Uhx}\sum_{j=1}^{\dim\Uhy}\rkla{\iO\Ihy{\Ihxloc{\parx\trkla{\lx{kl}\gt{kl} u\h}\ex{i}\trkla{x}\tfrac{\ey{j+1}\trkla{y}-\ey{j}\trkla{y}}{\hy}}}\dx\dy}^2\hx^{-1}\hy^{-1}\nonumber\\
&\leq 2\chiTh\sum_{i=1}^{\dim\Uhx}\sum_{j=1}^{\dim\Uhy}\rkla{\lx{kl}\iO\Ihy{\parx\gt{kl} \Ihx{u\h\ex{i}\trkla{x}}\tfrac{\ey{j+1}\trkla{y}-\ey{j}\trkla{y}}{\hy}}\dx\dy}^2\hx^{-1}\hy^{-1}\nonumber\\
&\quad+2\chiTh\sum_{i=1}^{\dim\Uhx}\sum_{j=1}^{\dim\Uhy}\rkla{\lx{kl}\iO\Ihy{\parx u\h \Ihx{\gt{kl}\ex{i}\trkla{x}}\tfrac{\ey{j+1}\trkla{y}-\ey{j}\trkla{y}}{\hy}}\dx\dy}^2\hx^{-1}\hy^{-1}\nonumber\\
&\leq\chiTh C\sum_{i=1}^{\dim\Uhx}\sum_{j=1}^{\dim\Uhy}\rkla{\lx{kl}\iO\Ihy{\parx u\h\Ihx{\dym \gt{kl}\ex{i}\trkla{x}}\ey{j+1}\trkla{y}}\dx\dy}^2\hx^{-1}\hy^{-1}\nonumber\\
&\quad+\chiTh C\sum_{i=1}^{\dim\Uhx}\sum_{j=1}^{\dim\Uhy}\rkla{\lx{kl}\iO\Ihy{\dyp\parx u\h\Ihx{\gt{kl}\ex{i}\trkla{x}}\ey{j}\trkla{y}}\dx\dy}^2\hx^{-1}\hy^{-1}\nonumber\\
&\quad+\chiTh C\sum_{i=1}^{\dim\Uhx}\sum_{j=1}^{\dim\Uhy}\rkla{\lx{kl}\iO\Ihy{\parx\gt{kl}\Ihx{\dym u\h\ex{i}\trkla{x}}\ey{j+1}\trkla{y}}\dx\dy}^2\hx^{-1}\hy^{-1}\nonumber\\
&\quad+\chiTh C\sum_{i=1}^{\dim\Uhx}\sum_{j=1}^{\dim\Uhy}\rkla{\lx{kl}\iO\Ihy{\dyp\parx\gt{kl}\Ihx{u\h\ex{i}\trkla{x}}\ey{j}\trkla{y}}\dx\dy}^2\hx^{-1}\hy^{-1}\nonumber\\
&\leq\chiTh C\lx{kl}^2\norm{\g{kl}}_{W^{1,\infty}\trkla{\Om}}^2 \iO\Ihy{\tabs{\parx u\h}^2}\dx\dy\nonumber\\
&\quad+\chiTh C\lx{kl}^2\norm{\g{kl}}_{L^\infty\trkla{\Om}}^2 \iO\Ihy{\tabs{\dyp\parx u\h}^2}\dx\dy\nonumber\\
&\quad+\chiTh C\lx{kl}^2\norm{\g{kl}}_{W^{1,\infty}\trkla{\Om}}^2 \iO\Ihxy{\tabs{\dym u\h}^2}\dx\dy\nonumber\\
&\quad+\chiTh C\lx{kl}^2\norm{\g{kl}}_{W^{1,\infty}\trkla{\Om}}^2\iO\Ihxy{\tabs{u\h}^2}\dx\dy\,,
\end{align}
where we again used \eqref{eq:Ihlocpar} and \eqref{eq:Ih=Ihloc}.
Noting
\begin{multline}
\iO\Ihy{\tabs{\dyp\parx u\h}^2}\dx\dy =\iO\Ihxy{\Delta\h^x u\h \Delta\h^y u\h}\dy\dy \\
\leq\tfrac12\iO\Ihxy{\tabs{\Delta\h^x u\h}^2}\dx\dy +\tfrac12\iO\Ihxy{\tabs{\Delta\h^y u\h}^2}\dx\dy=\tfrac12\iO\Ihxy{\tabs{\Delta\h u\h}^2}\dx\dy\,,
\end{multline}
we obtain
\begin{align}
\begin{split}
III_b\leq & C\p\sum_{k\in\Zx,\,l\in\Zy}\lx{kl}^2 \norm{\g{kl}}_{W^{1,\infty}\trkla{\Om}}^2\itTh R\trkla{s}^\p\ds\\
&+C\p\sum_{k\in\Zx,\,l\in\Zy}\lx{kl}^2\norm{\g{kl}}_{L^\infty\trkla{\Om}}^2\itTh R\trkla{s}^{\p-1}\iO\Ihxy{\tabs{\Delta\h u\h}^2}\dx\dy\ds\,
\end{split}
\end{align}
and analogously
\begin{align}
\begin{split}
III_c\leq & C\p\sum_{k\in\Zx,\,l\in\Zy}\ly{kl}^2 \norm{\g{kl}}_{W^{1,\infty}\trkla{\Om}}^2\itTh R\trkla{s}^\p\ds\\
&+C\p\sum_{k\in\Zx,\,l\in\Zy}\ly{kl}^2\norm{\g{kl}}_{L^\infty\trkla{\Om}}^2\itTh R\trkla{s}^{\p-1}\iO\Ihxy{\tabs{\Delta\h u\h}^2}\dx\dy\ds\,
\end{split}
\end{align}
We shall apply a similar strategy  to deal with $III_g$ and $III_h$.
We start with the decomposition
\begin{multline}
\iO\Ihxy{\abs{\Delta\h\sum_{i=1}^{\dim\Uhx}\sum_{j=1}^{\dim\Uhy} Z_{ij}^x\trkla{\lx{kl}\g{kl}}\ex{i}\trkla{x}\ey{j}\trkla{y}}^2}\dx\dy\\
\leq 2\iO\Ihxy{\sum_{j=1}^{\dim\Uhy}\abs{\sum_{i=1}^{\dim\Uhx} Z_{ij}^x\trkla{\lx{kl}\g{kl}}\Delta\h^x\ex{i}\trkla{x}}^2\ey{j}\trkla{y}}\dx\dy \\
+2\iO\Ihxy{\sum_{i=1}^{\dim\Uhx}\abs{\sum_{j=1}^{\dim\Uhy} Z_{ij}^x\trkla{\lx{kl}\g{kl}}\Delta\h^y\ey{j}\trkla{y}}^2\ex{i}\trkla{x}}\dx\dy\\
=:III_{g_1}+III_{g_2}\,.
\end{multline}
In order to control the first term, we use $\ex{i-1}\trkla{x}=\ex{i}\trkla{x+\hx}$ and $\ex{i+1}\trkla{x}=\ex{i}\trkla{x-\hx}$, apply Lemma \ref{lem:disc:pi}, and compute using \eqref{eq:Ihlocpar} and \eqref{eq:Ih=Ihloc}
\begin{multline}\label{eq:energyentropie:tmp:dpi}
\sum_{i=1}^{\dim\Uhy}Z_{ij}^x\trkla{\lx{kl}\g{kl}}\Delta\h^x\ex{i}\trkla{x}=\sum_{i=1}^{\dim\Uhx} Z_{ij}^x\trkla{\lx{kl}\g{kl}}\hx^{-2}\trkla{\ex{i-1}\trkla{x}-2\ex{i}\trkla{x}+\ex{i+1}\trkla{x}}\\
=\sum_{i=1}^{\dim\Uhx}\hx^{-2}\ex{i}\trkla{x}\trkla{-2Z_{ij}^x+Z_{i-1,j}^x+Z_{i+1,j}^x}\trkla{\lx{kl}\g{kl}}\\
=\chiTh\sum_{i=1}^{\dim\Uhx}\ex{i}\trkla{x}\hx^{-1}\hy^{-1}\iO\Ihy{\lx{kl}\parx\gt{kl}\Ihx{ u\h\tfrac{\ex{i+1}\trkla{x}-2\ex{i}\trkla{x}+\ex{i-1}\trkla{x}}{\hx^2}}\ey{j}\trkla{y}}\dx\dy\\
+\chiTh\sum_{i=1}^{\dim\Uhx}\ex{i}\trkla{x}\hx^{-1}\hy^{-1}\iO\Ihy{\lx{kl}\parx u\h\Ihx{\gt{kl} \tfrac{\ex{i+1}\trkla{x}-2\ex{i}\trkla{x}+\ex{i-1}\trkla{x}}{\hx^2}}\ey{j}\trkla{y}}\dx\dy\\
=\chiTh\sum_{i=1}^{\dim\Uhx}\ex{i}\trkla{x}\hx^{-1}\hy^{-1}\iO\Ihy{\lx{kl}\parx\Delta\h^x\gt{kl}\Ihx{u\h\trkla{x-\hx,y}\ex{i}\trkla{x}}\ey{j}\trkla{y}}\dx\dy\\
+\chiTh\sum_{i=1}^{\dim\Uhx}\ex{i}\trkla{x}\hx^{-1}\hy^{-1}\iO\Ihy{\lx{kl}\parx \gt{kl}\trkla{x+\hx,y}\Ihx{\Delta\h^x u\h\ex{i}\trkla{x}}\ey{j}\trkla{y}}\dx\dy\\
+\chiTh2\sum_{i=1}^{\dim\Uhx}\ex{i}\trkla{x}\hx^{-1}\hy^{-1}\iO\Ihy{\lx{kl}\parx\dxp\gt{kl}\Ihx{\dxm u\h\ex{i}\trkla{x}}\ey{j}\trkla{y}}\dx\dy\\
+\chiTh\sum_{i=1}^{\dim\Uhx}\ex{i}\trkla{x}\hx^{-1}\hy^{-1}\iO\Ihy{\lx{kl}\parx\Delta\h^x u\h\Ihx{\gt{kl}\trkla{x-\hx,y}\ex{i}\trkla{x}}\ey{j}\trkla{y}}\dx\dy\\
+\chiTh\sum_{i=1}^{\dim\Uhx}\ex{i}\trkla{x}\hx^{-1}\hy^{-1}\iO\Ihy{\lx{kl}\parx u\h\trkla{x+\hx,y}\Ihx{\Delta\h^x \gt{kl}\ex{i}\trkla{x}}\ey{j}\trkla{y}}\dx\dy\\
+\chiTh2\sum_{i=1}^{\dim\Uhx}\ex{i}\trkla{x}\hx^{-1}\hy^{-1}\iO\Ihy{\lx{kl}\parx\dxp u\h\Ihx{\dxm \gt{kl}\ex{i}\trkla{x}}\ey{j}\trkla{y}}\dx\dy\,.
\end{multline}
Similar to \eqref{eq:estimate:g:W2}, we use \ref{item:S} and the stability of the nodal interpolation operator to compute
\begin{multline}
\norm{\parx\Delta\h^x\Ihxy{\g{kl}}}_{L^\infty\trkla{\Om}}=\norm{\parx\Ihxy{\dxp\dxm\g{kl}}}_{L^\infty\trkla{\Om}}\\
\leq \norm{\parx\dxp\dxm\g{kl}}_{L^\infty\trkla{\Om}}\leq \norm{\parx\parx\parx\g{kl}}_{L^\infty\trkla{\Om}}\leq\norm{\g{kl}}_{W^{3,\infty}\trkla{\Om}}\,.
\end{multline}
Therefore, we have
\begin{align}
\begin{split}
III_{g_1}\leq&\, \chiTh C\lx{kl}^2\norm{\g{kl}}_{W^{3,\infty}\trkla{\Om}}^2 \iO\Ihxy{\tabs{u\h}^2}\dx\dy \\
&+ \chiTh C\lx{kl} \norm{\g{kl}}_{W^{1,\infty}\trkla{\Om}}^2 \iO\Ihxy{\tabs{\Delta\h^xu\h}^2}\dx\dy\\
&+\chiTh C\lx{kl}^2\norm{\g{kl}}_{W^{2,\infty}\trkla{\Om}}^2 \iO\Ihy{\tabs{\parx u\h}^2}\dx\dy \\
&+\chiTh C\lx{kl}^2\norm{\g{kl}}^2_{L^\infty\trkla{\Om}}\iO\Ihy{\tabs{\parx\Delta\h^x u\h}^2}\dx\dy\\
&+\chiTh C\lx{kl}^2\norm{\g{kl}}_{W^{2,\infty}\trkla{\Om}}^2\iO\Ihy{\tabs{\parx u\h}^2}\dx\dy \\
&+\chiTh C\lx{kl}^2\norm{\g{kl}}_{W^{1,\infty}\trkla{\Om}}^2\iO\Ihxy{\tabs{\Delta\h^x u\h}^2}\dx\dy\,.
\end{split}
\end{align}
Similar computations show
\begin{align}
\begin{split}
III_{g_2}\leq &\,\chiTh C\lx{kl}^2\norm{\g{kl}}_{W^{3,\infty}\trkla{\Om}}^2\iO\Ihxy{\tabs{u\h}^2}\dx\dy \\
&+\chiTh C\lx{kl}^2\norm{\g{kl}}_{W^{1,\infty}\trkla{\Om}}^2\iO\Ihxy{\tabs{\Delta\h^yu\h}^2}\dx\dy\\
&+\chiTh C\lx{kl}^2\norm{\g{kl}}_{W^{2,\infty}\trkla{\Om}}^2\iO\Ihx{\tabs{\pary u\h}^2}\dx\dy \\
&+\chiTh C\lx{kl}^2\norm{\g{kl}}_{L^\infty\trkla{\Om}}^2\iO\Ihy{\tabs{\parx\Delta\h^yu\h}^2}\dx\dy\\
&+\chiTh C\lx{kl}^2\norm{\g{kl}}_{W^{2,\infty}\trkla{\Om}}^2\iO\Ihy{\tabs{\parx u\h}^2}\dx\dy \\
&+\chiTh C\lx{kl}^2\norm{\g{kl}}_{W^{1,\infty}\trkla{\Om}}^2\iO\Ihxy{\Delta\h^xu\h\Delta\h^yu\h}\dx\dy\,.
\end{split}
\end{align}
Therefore,
\begin{align}
\begin{split}
III_g\leq&\, C\lx{kl}^2\norm{\g{kl}}_{W^{3,\infty}\trkla{\Om}}^2h^\varepsilon\sum_{k\in\Zx,\,l\in\Zy}\itTh \p R\trkla{s}^\p\ds\\
&+C\lx{kl}^2\norm{\g{kl}}_{W^{1,\infty}\trkla{\Om}}^2 h^\varepsilon\sum_{k\in\Zx,\,l\in\Zy}\itTh\p R\trkla{s}^{\p-1}\iO\Ihxy{\tabs{\Delta\h u\h}^2}\dx\dy\ds\\
&+C\lx{kl}^2\norm{\g{kl}}_{L^\infty\trkla{\Om}}^2 h^\varepsilon\sum_{k\in\Zx,\,l\in\Zy}\itTh\p R\trkla{s}^{\p-1}\iO\Ihy{\tabs{\parx\Delta\h u\h}^2}\dx\dy\ds
\end{split}
\end{align}
holds true.
Analogous computations provide
\begin{align}
\begin{split}
III_h\leq&\, C\ly{kl}^2\norm{\g{kl}}_{W^{3,\infty}\trkla{\Om}}^2h^\varepsilon\sum_{k\in\Zx,\,l\in\Zy}\itTh \p R\trkla{s}^\p\ds\\
&+C\ly{kl}^2\norm{\g{kl}}_{W^{1,\infty}\trkla{\Om}}^2 h^\varepsilon\sum_{k\in\Zx,\,l\in\Zy}\itTh\p R\trkla{s}^{\p-1}\iO\Ihxy{\tabs{\Delta\h u\h}^2}\dx\dy\ds\\
&+C\ly{kl}^2\norm{\g{kl}}_{L^\infty\trkla{\Om}}^2 h^\varepsilon\sum_{k\in\Zx,\,l\in\Zy}\itTh\p R\trkla{s}^{\p-1}\iO\Ihx{\tabs{\pary\Delta\h u\h}^2}\dx\dy\ds\,.
\end{split}
\end{align}
To control $III_e+III_f$ we compute for all $i\in\tgkla{1,\ldots,\dim\Uhx}$, $j\in\tgkla{1,\ldots,\dim\Uhy}$, $k\in\Zx$, and $l\in\Zy$
\begin{align}
\begin{split}
&\iOmega\Ihxy{F^{\prime\prime}\trkla{u\h}\rkla{\tabs{Z_{ij}^x\trkla{\lx{kl}\g{kl}}}^2+\tabs{Z_{ij}^y\trkla{\ly{kl}\g{kl}}}^2}\ex{i}\trkla{x}\ey{j}\trkla{y}}\dx\dy\\
&\qquad=\chiTh F^{\prime\prime}\trkla{u_{ij}}M_{ij}^{-1}\left(\rkla{\iO\Ihy{\Ihxloc{\lx{kl} \parx\trkla{u\h\gt{kl}}\ex{i}\trkla{x}\ey{j}\trkla{y}}}\dx\dy}^2\right.\\
&\qquad\qquad+\left.\rkla{\iO\Ihx{\Ihyloc{\ly{kl} \pary\trkla{u\h\gt{kl}}\ex{i}\trkla{x}\ey{j}\trkla{y}}}\dx\dy}^2\right)\\
&\qquad=:\chiTh F^{\prime\prime}\trkla{u_{ij}}M_{ij}^{-1}\rkla{\lx{kl}^2A_{ijkl}^2 + \ly{kl}^2B_{ijkl}^2}\,.
\end{split}
\end{align}
Using the definition of $\Ihxlocop$, we obtain for the first term
\begin{align}
\begin{split}\label{eq:energyentropy:tmp:1}
A_{ijkl}^2=&\rkla{\iO\Ihy{\Ihxloc{\parx\trkla{u\h\gt{kl}}\ex{i}\trkla{x}\ey{j}\trkla{y}}}\dx\dy}^2\\
=&\rkla{\sum_{Q\in\Qh}\int_Q\Ihxy{\parx\trkla{u\h\gt{kl}}\ex{i}\trkla{x}\ey{j}\trkla{y}}\dx\dy}^2\\
\leq&\,M_{ij}\sum_{Q\in\Qh}\int_Q\Ihxy{\tabs{\parx\trkla{u\h\gt{kl}}}^2\ex{i}\trkla{x}\ey{j}\trkla{y}}\dx\dy \\
\leq&\,2M_{ij}\sum_{Q\in\Qh}\int_Q\Ihxy{\trkla{\tabs{\parx u\h}^2\tabs{\gt{kl}}^2+\tabs{ u\h}^2\tabs{\parx\gt{kl}}^2}\ex{i}\trkla{x}\ey{j}\trkla{y}}\dx\dy\\
\leq&\,C\norm{\gt{k}}_{W^{1,\infty}\trkla{\Om}}M_{ij}\sum_{Q\in\Qh}\int_Q\Ihxy{\tabs{\parx u\h}^2\ex{i}\trkla{x}\ey{j}\trkla{y}}\dx\dy\\
&+C\norm{\gt{k}}_{W^{1,\infty}\trkla{\Om}}M_{ij}\sum_{Q\in\Qh} \int_Q\Ihxy{\tabs{u\h}^2\ex{i}\trkla{x}\ey{j}\trkla{y}}\dx\dy\\
=&\,C\norm{\gt{kl}}_{W^{1,\infty}\trkla{\Om}}M_{ij}\iO\Ihy{\tabs{\parx u\h}^2\ex{i}\trkla{x}\ey{j}\trkla{y}}\dx\dy\\
&+C\norm{\gt{kl}}_{W^{1,\infty}\trkla{\Om}}M_{ij}\iO\Ihxy{\tabs{u\h}^2\ex{i}\trkla{x}\ey{j}\trkla{y}}\dx\dy\,.
\end{split}
\end{align}
Similar computations for $B_{ijkl}$ provide
\begin{align}
III_e\!+\!III_f\leq &\,C\!\itTh\!\!\!\p R\trkla{s}^{\p-1}\!\!\iO \!\Ihy{\Ihx{F^{\prime\prime}\trkla{u\h}}\tabs{\parx u\h}^2}\!+\!\Ihx{\Ihy{F^{\prime\prime}\trkla{u\h}}\tabs{\pary u\h}^2}\!\dx\!\dy\!\ds\nonumber\\
&+C\!\itTh\!\!\!\p R\trkla{s}^{\p-1}\!\!\iO\!\Ihxy{\tabs{u\h}^2 F^{\prime\prime}\trkla{u\h}}\!\dx\!\dy\!\ds\,.
\end{align}
Analogously, we compute
\begin{align}
III_i\!+\!III_j\leq &\,C\kappa\!\!\itTh\!\!\!\p R\trkla{s}^{\p-1}\!\!\iO\! \Ihy{\Ihx{G^{\prime\prime}\trkla{u\h}}\tabs{\parx u\h}^2}\!\!+\!\Ihx{\Ihy{G^{\prime\prime}\trkla{u\h}}\tabs{\pary u\h}^2}\!\dx\!\dy\!\ds\nonumber\\
&+C\kappa\!\itTh\!\!\!\p R\trkla{s}^{\p-1}\!\!\iO\!\Ihxy{\tabs{u\h}^2 G^{\prime\prime}\trkla{u\h}}\!\dx\!\dy\!\ds\,.
\end{align}
Collecting the above estimates and applying Assumptions \ref{item:stochbasis:bounds} and \ref{item:stochbasis:boundthirdderivative}, we obtain
\begin{align}
\begin{split}
III+IV\leq&\,C\itTh R\trkla{s}^{\p}\ds +C\itTh R\trkla{s}^{\p-1}\norm{\Delta\h u\h}\h^2\ds\\
&+C\itTh R\trkla{s}^{\p-1}h^\varepsilon\iO\Ihy{\tabs{\parx \Delta\h u\h}^2}\dx\dy\ds\\
&+C\itTh R\trkla{s}^{\p-1}h^\varepsilon\iO\Ihx{\tabs{\pary \Delta\h u\h}^2}\dx\dy\ds\\
&+C\itTh R\trkla{s}^{\p-1}\iO\Ihy{\Ihx{F^{\prime\prime}\trkla{u\h}}\tabs{\parx u\h}^2}\dx\dy\ds\\
&+C\itTh R\trkla{s}^{\p-1}\iO\Ihx{\Ihy{F^{\prime\prime}\trkla{u\h}}\tabs{\pary u\h}^2}\dx\dy\ds\\
&+C\itTh R\trkla{s}^{\p-1}\iO\Ihxy{\tabs{u\h}^2 F^{\prime\prime}\trkla{u\h}}\dx\dy\ds\\
&+C\kappa\itTh R\trkla{s}^{\p-1}\iO\Ihy{\Ihx{G^{\prime\prime}\trkla{u\h}}\tabs{\parx u\h}^2}\dx\dy\ds\\
&+C\kappa\itTh R\trkla{s}^{\p-1}\iO\Ihx{\Ihy{G^{\prime\prime}\trkla{u\h}}\tabs{\pary u\h}^2}\dx\dy\ds\\
&+C\kappa\itTh R\trkla{s}^{\p-1}\iO\Ihxy{\tabs{u\h}^2 G^{\prime\prime}\trkla{u\h}}\dx\dy\ds\,.
\end{split}\label{eq:tmp:III+IV}
\end{align}
While the first term on the right-hand side is a Gronwall term, the remaining terms need to be absorbed in the negative terms provided by $I$. 
For this reason, we need the following estimates:
Due to Assumption \ref{item:potential}, we have $F^{\prime\prime}\trkla{u\h}\leq C u\h^{-p-2} +C$.
Recalling the oscillation lemma \ref{lem:oscillation}, we obtain the estimates
\begin{align}
\iO\Ihy{\Ihx{F^{\prime\prime}\trkla{u\h}}\tabs{\parx u\h}^2}\dx\dy\leq C\iO\Ihy{\meanspX{u_h}\tabs{\parx u\h}^2}\dx\dy + C R\trkla{s}\,,\\
\iO\Ihx{\Ihy{F^{\prime\prime}\trkla{u\h}}\tabs{\pary u\h}^2}\dx\dy\leq C\iO\Ihx{\meanspY{u\h}\tabs{\pary u\h}^2}\dx\dy + C R\trkla{s}\,.
\end{align}
Furthermore, Assumption \ref{item:potential}, Poincar\'e's inequality and our uniform control of the mass of discrete solutions (see Remark~\ref{rem:initialdata}, Assumption \ref{item:initial}, and the norm equivalence \eqref{eq:normequivalence}) provide
\begin{multline}
\iO\Ihxy{\tabs{u\h}^2 F^{\prime\prime}\trkla{u\h}}\dx\dy\\
\leq C\iO\Ihxy{F\trkla{u\h}}\dx\dy + C\iO\Ihy{\tabs{\parx u\h}^2}\dx\dy +C\leq C R\trkla{s}
\end{multline}for $\alpha>0$.
\revy{As Lemma \ref{lem:oscillation} provides the estimate $\restr{\Ihx{u\h^{-2}}^2}{\Kx}\leq \trkla{\max_{\Kx} \Ihx{u\h^{-2}}}^2\leq  \trkla{C_{osc}\min_{\Kx} \Ihx{u\h^{-2}}}^2\leq C_{osc}^2\restr{\Ihx{u\h^{-4}}}{\Kx}$, we can apply Young's inequality and use $p>2$ and $\kappa>1$ to compute
\begin{align}\label{eq:tmp:est:gprimeprime}
\begin{split}
\kappa\iO&\Ihy{\Ihx{G^{\prime\prime}\trkla{u\h}}\tabs{\parx u\h}^2}\dx\dy\\
 &\leq C_{osc}^{-2}\iO\Ihy{\abs{\Ihx{G^{\prime\prime}\trkla{u\h}}}^2\tabs{\parx u\h}^2}\dx\dy +C\kappa^2\iO\Ihy{\tabs{\parx u\h}^2}\dx\dy\\
 &\leq \iO\Ihy{\Ihx{\tabs{u\h}^{-4}}\tabs{\parx u\h}^2}\dx\dy +C\kappa^2\iO\Ihy{\tabs{\parx u\h}^2}\dx\dy\\
 &\leq \iO\Ihy{\Ihx{\tabs{u\h}^{-p-2}+1}\tabs{\parx u\h}^2}\dx\dy +C\kappa^2\iO\Ihy{\tabs{\parx u\h}^2}\dx\dy\\
&\leq \iO\Ihy{\Ihx{\tabs{u\h}^{-p-2}}\tabs{\parx u\h}^2}\dx\dy +C\kappa^2\iO\Ihy{\tabs{\parx u\h}^2}\dx\dy\\
&\leq C \iO\Ihy{\meanspX{u\h}\tabs{\parx u\h}^2}\dx\dy +C\kappa^2 R\trkla{s}
\end{split}
\end{align}
and }
\begin{align}
\kappa\iO\Ihx{\Ihy{G^{\prime\prime}\trkla{u\h}}\tabs{\pary u\h}^2}\dx\dy\leq C \iO\Ihx{\meanspY{u\h}\tabs{\pary u\h}^2}\dx\dy +C\kappa^2 R\trkla{s}\,.
\end{align}
\revy{In the last line of \eqref{eq:tmp:est:gprimeprime}, we used that Lemma \ref{lem:oscillation} allows us control $\Ihx{\tabs{u\h}^{p-2}}$ by its mean value.}
Noting the definition of $G^{\prime\prime}\trkla{u\h}$, we obtain 
\begin{align}
\iO\Ihxy{\tabs{u\h}^2 G^{\prime\prime}\trkla{u\h}}\dx\dy\leq \iO 1\dx\dy\leq C\,.
\end{align}
Therefore, the last term in \eqref{eq:tmp:III+IV} can be controlled by $C\kappa\itTh R\trkla{s}^{\p-1}\ds$ for $\alpha>0$.
This allows us to rewrite \eqref{eq:tmp:III+IV} for $\kappa>1$ as
\begin{align}
III+IV\leq&\, C\kappa^2 \itTh R\trkla{s}^\p\ds + C\itTh R\trkla{s}^{\p-1}\norm{\Delta\h u\h}\h^2\ds\nonumber\\
&+C\itTh R\trkla{s}^{\p-1}h^\varepsilon\iO\Ihy{\tabs{\parx\Delta\h u\h}^2}\dx\dy\ds\nonumber\\
&+C\itTh R\trkla{s}^{\p-1}h^{\varepsilon}\iO\Ihx{\tabs{\pary\Delta\h u\h}^2}\dx\dy\ds\\
&+C\itTh R\trkla{s}^{\p-1}\iO\Ihy{\meanspX{u\h}\tabs{\parx u\h}^2}\dx\dy\ds\nonumber\\
&+C\itTh R\trkla{s}^{\p-1}\iO\Ihx{\meanspY{u\h}\tabs{\pary u\h}^2}\dx\dy\ds\,.\nonumber
\end{align}
Using \eqref{eq:semidiscreteSTFE:2}, we compute
\begin{align}\label{eq:energyentropy:proof:tmp1}
\begin{split}
&\trkla{D\mathcal{E}\h+\kappa D\mathcal{S}\h}\otimes\trkla{D\mathcal{E}\h+\kappa D\mathcal{S}\h}\trkla{\Phi_{h,kl}^x,\Phi_{h,kl}^x}\\
\leq&\,2 \left(\iO\Ihy{\parx u\h\sum_{i=1}^{\dim\Uhx}\sum_{j=1}^{\dim\Uhy}Z_{ij}^x\trkla{\lx{kl}\g{kl}}\parx\ex{i}\trkla{x}\ey{j}\trkla{y}}\dx\dy \right.\\
&+\iO\Ihx{\pary u\h\sum_{i=1}^{\dim\Uhx}\sum_{j=1}^{\dim\Uhy}Z_{ij}^x\trkla{\lx{kl}\g{kl}}\ex{i}\trkla{x}\pary\ey{j}\trkla{y}}\dx\dy\\
&+\iO\Ihxy{F^\prime\trkla{u\h}\sum_{i=1}^{\dim\Uhx}\sum_{j=1}^{\dim\Uhy} Z_{ij}^x\trkla{\lx{kl}\g{kl}}\ex{i}\trkla{x}\ey{j}\trkla{y}}\dx\dy\\
&+\left.h^\varepsilon\iO\Ihxy{\Delta\h u\h\sum_{i=1}^{\dim\Uhx}\sum_{j=1}^{\dim\Uhy}Z_{ij}^x\trkla{\lx{kl}\g{kl}}\Delta\h\trkla{\ex{i}\trkla{x}\ey{j}\trkla{y}}}\dx\dy\right)^2\\
&+2\kappa^2 \rkla{\iO\Ihxy{G^\prime\trkla{u\h}\sum_{i=1}^{\dim\Uhx}\sum_{j=1}^{\dim\Uhy} Z_{ij}^x\trkla{\lx{kl}\g{kl}}\ex{i}\trkla{x}\ey{j}\trkla{y}}\dx\dy}^2\\
=&\,2\rkla{\iO\Ihxy{p\h\sum_{i=1}^{\dim\Uhx}\sum_{j=1}^{\dim\Uhy} Z_{ij}^x\trkla{\lx{kl}\g{kl}}\ex{i}\trkla{x}\ey{j}\trkla{y}}\dx\dy}^2 \\
&+2\kappa^2 \rkla{\iO\Ihxy{G^\prime\trkla{u\h}\sum_{i=1}^{\dim\Uhx}\sum_{j=1}^{\dim\Uhy} Z_{ij}^x\trkla{\lx{kl}\g{kl}}\ex{i}\trkla{x}\ey{j}\trkla{y}}\dx\dy}^2\\
=&\chiTh2\lx{kl}^2\rkla{\iO\Ihy{\Ihxloc{\parx\trkla{u\h\gt{kl} }p\h}}\dx\dy}^2 \\
&+\chiTh 2\kappa^2\lx{kl}^2\rkla{\iO\Ihy{\Ihxloc{\parx\trkla{u\h\gt{kl}} G^{\prime}\trkla{u\h}}}\dx\dy}^2\\
=:&\,\chiTh2\lx{kl}^2A_I+\chiTh2\kappa^2\lx{kl}^2A_{II}\,,
\end{split}
\end{align}
where we used \eqref{eq:def:Zx} in the last step. 
To control the first term on the right-hand side of \eqref{eq:energyentropy:proof:tmp1}, we use \eqref{eq:Ihlocpar} and \eqref{eq:Ih=Ihloc} to compute
\begin{multline}
A_I\leq 2\rkla{\iO\Ihy{\parx u\h\Ihx{\gt{kl} p\h}}\dx\dy}^2 +2 \rkla{\iO\Ihy{\parx\gt{kl} \Ihx{u\h p\h}}}^2=:A_{I_a}+A_{I_b}\,.
\end{multline}
Recalling \eqref{eq:semidiscreteSTFE:2}, \ref{item:potential}, the integration by parts formula used in \eqref{eq:energyentropie:tmp:dpi}, Hölder's inequality, and the positivity of $u\h$, we obtain\pagebreak
\begin{align}
\begin{split}
A_{I_a}\leq&\,C\rkla{\iO \Ihy{\parx u\h\Ihx{\gt{kl} \Delta_h u\h}}\dx\dy}^2\\
&+C\rkla{\iO\Ihy{\tabs{\parx u\h}\Ihx{\tabs{\gt{kl}} \trkla{u\h^{-p-1}+1}}}\dx\dy}^2\\
&+C\rkla{h^\varepsilon\iO\Ihy{\parx u\h\Ihx{\gt{kl}\Delta\h\Delta\h u\h}}\dx\dy}^2\\
\leq&\,C\rkla{\iO\Ihy{\tabs{\parx u\h}^2}\dx\dy}\norm{\gt{kl}}_{L^\infty\trkla{\Om}}^2 \norm{\Delta\h u\h}\h^2\\
&+C\rkla{\iO\Ihy{\tabs{\parx u\h}\Ihx{\tabs{\gt{kl}} u\h^{-p/2}u\h^{-p/2-1}}}\dx\dy}^2\\
&+C\rkla{\iO\Ihy{\tabs{\parx u\h}\Ihx{\tabs{\gt{kl}} }}\dx\dy}^2\\
&+C\norm{\g{kl}}_{L^\infty\trkla{\Om}}^2h^\varepsilon\iO\Ihy{\tabs{\parx\Delta\h^x u\h}^2}\dx\dy\, h^\varepsilon\iO\Ihxy{\tabs{\Delta\h u\h}^2}\dx\dy \\
&+C\norm{\parx\parx\g{kl}}_{L^\infty\trkla{\Om}}^2 h^\varepsilon\iO\Ihy{\tabs{\parx u\h}^2}\dx\dy \,h^\varepsilon\iO\Ihxy{\tabs{\Delta\h u\h}^2}\dx\dy\\
&+C\norm{\parx\g{kl}}_{L^\infty\trkla{\Om}}^2h^\varepsilon\iO\Ihxy{\tabs{\Delta\h^x u\h}^2}\dx\dy\,h^\varepsilon\iO\Ihxy{\tabs{\Delta\h u\h}^2}\dx\dy\\
&+C\norm{\g{kl}}_{L^\infty\trkla{\Om}}^2 h^\varepsilon\iO\Ihy{\tabs{\parx \Delta\h^y u\h}^2}\dx\dy\,h^\varepsilon\iO\Ihxy{\tabs{\Delta\h u\h}^2}\dx\dy\\
&+C\norm{\pary\pary\g{kl}}_{L^\infty\trkla{\Om}}^2 h^\varepsilon\iO\Ihy{\tabs{\parx u\h}^2}\dx\dy\,h^\varepsilon\iO\Ihxy{\tabs{\Delta\h u\h}^2}\dx\dy\\
&+C\norm{\pary\g{kl}}_{L^\infty\trkla{\Om}}^2 h^\varepsilon\iO\Ihxy{\Delta\h^x u\h\Delta\h^y u\h}\dx\dy\,h^\varepsilon\iO\Ihxy{\tabs{\Delta\h u\h}^2}\dx\dy\,,
\end{split}
\end{align}
where we used $\iO\Ihy{\tabs{\parx\dyp u\h}^2}\dx\dy=\iO\Ihxy{\Delta\h^x u\h \Delta\h^y u\h}\dx\dy$. 
Using a discrete version of Hölder's inequality, we obtain 
\begin{align}
\iO\Ihxy{\Delta\h^xu\h\Delta\h^yu\h}\dx\dy\leq \norm{\Delta\h^xu\h}\h\norm{\Delta\h^yu\h}\h\leq\norm{\Delta\h u\h}\h^2\,.
\end{align}
Therefore, we have
\begin{align}
A_{I_a}\leq&\,C R\trkla{s}\norm{\g{kl}}_{L^\infty\trkla{\Om}}^2\norm{\Delta\h u\h}\h^2\nonumber\\
&+ C \norm{\g{kl}}_{L^\infty\trkla{\Om}}^2\rkla{\iO\Ihy{\tabs{\parx u\h}^2\Ihx{u\h^{-p-2}}}\dx\dy} \rkla{\iO\Ihxy{u\h^{-p}}\dx\dy}\nonumber\\
&+ C \norm{\g{kl}}_{L^\infty\trkla{\Om}}^2 \iO\Ihy{\tabs{\parx u\h}^2}\dx\dy\nonumber\\
&+C\norm{\g{kl}}_{W^{2,\infty}\trkla{\Om}}^2 R\trkla{s}\rkla{h^\varepsilon\iO\Ihy{\tabs{\parx\Delta\h u\h}^2}\dx\dy + R\trkla{s}}\nonumber\\
\begin{split}
\leq&\,C R\trkla{s}\norm{\g{kl}}_{W^{2,\infty}\trkla{\Om}}^2\left(\norm{\Delta\h u\h}\h^2 +\iO\Ihy{\tabs{\parx u\h}^2\Ihx{u\h^{-p-2}}}\dx\dy\right.\\
&\left.\qquad\qquad\qquad\qquad\qquad\qquad\qquad +h^\varepsilon\iO\Ihy{\tabs{\parx\Delta\h u\h}^2}\dx\dy + R\trkla{s}\right)\,.\end{split}
\label{eq:energyentropy:tmp:hilbertschmidt1}
\end{align}
Using \eqref{eq:semidiscreteSTFE:2}, Hölder's inequality, \ref{item:potential}, the computations used in \eqref{eq:energyentropie:tmp:dpi}, Poincar\'e's inequality, and $R\trkla{s}\geq\alpha$, we obtain the estimate
\begin{align}
\begin{split}\label{eq:energyentropy:tmp:hilberschmidt2}
A_{I_b}\leq&\,C\rkla{\iO\Ihy{\parx\gt{kl}\Ihx{u\h\Delta\h u\h}}\dx\dy}^2 + C\rkla{\iO\Ihy{\parx\gt{kl}\Ihx{u\h F^\prime\trkla{u\h}}}\dx\dy}^2\\
&+C\rkla{h^\varepsilon\iO\Ihy{\parx \gt{kl}\Ihx{u\h\Delta\h\Delta\h u\h}}\dx\dy}^2\\
\leq&\,C\norm{\parx\g{kl}}_{L^\infty\trkla{\Om}}^2\norm{u\h}\h^2\norm{\Delta\h u\h}\h^2\\
&+C\norm{\parx\g{kl}}_{L^\infty\trkla{\Om}}^2\rkla{\iO\Ihxy{u\h^{-p}}\dx\dy}^2+C\norm{\parx\g{kl}}_{L^\infty\trkla{\Om}}^2\rkla{\iO u\h\dx\dy}^2\\
&+C h^\varepsilon\norm{\parx\parx\parx\g{kl}}_{L^\infty\trkla{\Om}}^2\norm{u\h}\h^2 h^\varepsilon\norm{\Delta\h u\h}\h^2\\
&+C\norm{\parx\g{kl}}_{L^\infty\trkla{\Om}}^2h^\varepsilon\norm{\Delta\h^xu\h}\h^2 h^\varepsilon\norm{\Delta\h u\h}\h^2\\
&+C\norm{\parx\parx\g{kl}}_{L^\infty\trkla{\Om}}^2 h^\varepsilon \iO\Ihy{\tabs{\parx u\h}^2}\dx\dy\, h^\varepsilon\norm{\Delta\h u\h}\h^2\\
&+Ch^\varepsilon\norm{\parx\pary\pary\g{kl}}_{L^\infty\trkla{\Om}}^2\norm{u\h}\h^2 h^\varepsilon\norm{\Delta\h u}\h^2\\
&+C\norm{\parx\g{kl}}_{L^\infty\trkla{\Om}}^2 h^\varepsilon\norm{\Delta\h^yu\h}\h^2h^\varepsilon\norm{\Delta\h u\h}\h^2\\
&+C\norm{\parx\pary\g{kl}}_{L^\infty\trkla{\Om}}^2 h^\varepsilon\iO\Ihx{\tabs{\pary u\h}^2}\dx\dy\,h^\varepsilon\norm{\Delta\h u\h}\h^2\\
\leq&\,C\rkla{\norm{\g{kl}}_{W^{2,\infty}\trkla{\Om}}^2 +h^\varepsilon\norm{\g{kl}}_{W^{3,\infty}\trkla{\Om}}^2} R\trkla{s} \rkla{\norm{\Delta\h u\h}\h^2+R\trkla{s}}\,.
\end{split}
\end{align}
Here, we used that $\iO u\h\dx\dy=\iO u\h^0\dx\dy$ is uniformly bounded $\Prob$-almost surely (cf.~Assumption \ref{item:initial} and Remark~\ref{rem:initialdata}).
Using \eqref{eq:Ihlocpar} and \eqref{eq:Ih=Ihloc} and  $G^\prime(u)=u^{-1}$, we obtain
\begin{align}\label{eq:energyentropy:tmp:hilberschmidt3}
A_{II}\leq C\norm{\g{kl}}_{L^\infty\trkla{\Om}}^2\iO\Ihy{\Ihx{G^{\prime\prime}\trkla{u\h}}\tabs{\parx u\h}^2}\dx\dy + C\norm{\parx\g{kl}}_{L^\infty\trkla{\Om}}^2\,.
\end{align}
Therefore, we have
\begin{align}
V\leq&\,C\sum_{k\in\Zx,\,l\in\Zy} \lx{kl}^2\rkla{\norm{\g{kl}}_{W^{2,\infty}\trkla{\Om}}^2+h^\varepsilon\norm{\g{kl}}_{W^{3,\infty}\trkla{\Om}}^2}\itTh R\trkla{s}^\p\ds\nonumber\\
&+C\sum_{k\in\Zx,\,l\in\Zy} \lx{kl}^2\rkla{\norm{\g{kl}}_{W^{2,\infty}\trkla{\Om}}^2+h^\varepsilon\norm{\g{kl}}_{W^{3,\infty}\trkla{\Om}}^2}\itTh R\trkla{s}^{\p-1}\norm{\Delta\h u\h}\h^2\ds\nonumber\\
&+C\sum_{k\in\Zx,\,l\in\Zy} \lx{kl}^2\norm{\g{kl}}_{W^{2,\infty}\trkla{\Om}}^2\itTh R\trkla{s}^{\p-1}\iO\Ihy{\tabs{\parx u\h}^2\Ihx{u\h^{-p-2}}}\dx\dy\ds\nonumber\\
&+C\sum_{k\in\Zx,\,l\in\Zy} \lx{kl}^2\norm{\g{kl}}_{W^{2,\infty}\trkla{\Om}}^2\itTh R\trkla{s}^{\p-1}h^\varepsilon\iO\Ihy{\tabs{\parx\Delta\h u\h}^2}\dx\dy\ds\nonumber\\
&+C\kappa^2\sum_{k\in\Zx,\,l\in\Zy} \lx{kl}^2\norm{\g{kl}}_{L^\infty\trkla{\Om}}^2\itTh R\trkla{s}^{\p-1}\iO\Ihy{\Ihx{G^{\prime\prime}\trkla{u\h}}\tabs{\parx u\h}^2}\dx\dy\ds\,.
\end{align}
Recalling Assumptions \ref{item:stochbasis:bounds} and \ref{item:stochbasis:boundthirdderivative} and Lemma \ref{lem:oscillation} and mimicking the computations in \eqref{eq:tmp:est:gprimeprime}, we obtain for $\kappa>1$
\begin{align}
\begin{split}
V\leq &\,C\kappa^4\itTh R\trkla{s}^\p \ds+ C\itTh R\trkla{s}^{\p-1}\norm{\Delta\h u\h}\h^2\ds\\
&+ C\itTh R\trkla{s}^{\p-1}\iO\Ihy{\meanspX{u\h}\tabs{\parx u\h}^2}\dx\dy\ds\\
&+ C\itTh R\trkla{s}^{\p-1}\iO\Ihx{\meanspY{u\h}\tabs{\pary u\h}^2}\dx\dy\ds\\
&+C\itTh R\trkla{s}^{\p-1} h^\varepsilon\iO\Ihy{\tabs{\parx\Delta\h u\h}^2}\dx\dy\ds\,.
\end{split}
\end{align}
Analogously, we compute
\begin{align}
\begin{split}
VI\leq &\,C\kappa^4\itTh R\trkla{s}^\p \ds+ C\itTh R\trkla{s}^{\p-1}\norm{\Delta\h u\h}\h^2\ds\\
&+ C\itTh R\trkla{s}^{\p-1}\iO\Ihy{\meanspX{u\h}\tabs{\parx u\h}^2}\dx\dy\ds\\
&+ C\itTh R\trkla{s}^{\p-1}\iO\Ihx{\meanspY{u\h}\tabs{\pary u\h}^2}\dx\dy\ds\\
&+C\itTh R\trkla{s}^{\p-1} h^\varepsilon\iO\Ihx{\tabs{\pary\Delta\h u\h}^2}\dx\dy\ds\,.
\end{split}
\end{align}
In order to obtain bounds for the expected value of stochastic integral, we rewrite
\begin{align}
\begin{split}
II=&\sum_{k\in\Zx,\,l\in\Zy}\itTh\p R\trkla{s}^{\p-1} \iO\Ihy{\Ihxloc{\parx\trkla{\lx{kl}\gt{kl} u\h} p\h}}\dx\dy\dbx{kl}\\
&+\sum_{k\in\Zx,\,l\in\Zy}\itTh\p R\trkla{s}^{\p-1} \iO\Ihx{\Ihyloc{\pary\trkla{\ly{kl}\gt{kl} u\h} p\h}}\dx\dy\dby{kl}\\
&+\kappa\sum_{k\in\Zx,\,l\in\Zy}\itTh\p R\trkla{s}^{\p-1} \iO\Ihy{\Ihxloc{\parx\trkla{\lx{kl}\gt{kl} u\h} G^\prime\trkla{u\h}}}\dx\dy\dbx{kl}\\
&+\kappa\sum_{k\in\Zx,\,l\in\Zy}\itTh\p R\trkla{s}^{\p-1} \iO\Ihx{\Ihyloc{\pary\trkla{\ly{kl}\gt{kl} u\h} G^\prime\trkla{u\h}}}\dx\dy\dby{kl}\\
=:&\,II_a+II_b+II_c+II_d
\end{split}
\end{align}
and apply the Burkholder-Davis-Gundy inequality.
For this reason, we introduce the Hilbert-Schmidt operators 
\begin{subequations}
\begin{align}
\mathfrak{T}_1\trkla{s}\trkla{\omega}&:= \chiTh R\trkla{s}^{\p-1}\sum_{k\in\Zx,\,l\in\Zy}\iO\Ihy{\Ihxloc{\parx\trkla{\skla{\g{kl},\omega}_{L^2} u\h\gt{kl}} p\h}}\dx\dy\,,\\
\mathfrak{T}_3\trkla{s}\trkla{\omega}&:=\chiTh \kappa R\trkla{s}^{\p-1}\sum_{k\in\Zx,\,l\in\Zy}\iO\Ihy{\Ihxloc{\parx\trkla{\skla{\g{kl},\omega}_{L^2} u\h\gt{kl}} G^\prime\trkla{u\h}}}\dx\dy\,,
\end{align}
mapping $Q^{1/2}_xL^2\trkla{\Om}$ onto $\mathds{R}$  and the Hilbert-Schmidt operators
\begin{align}
\mathfrak{T}_2\trkla{s}\trkla{\omega}&:= \chiTh R\trkla{s}^{\p-1}\sum_{k\in\Zx,\,l\in\Zy}\iO\Ihx{\Ihyloc{\pary\trkla{\skla{\g{kl},\omega}_{L^2} u\h\gt{kl}} p\h}}\dx\dy\,,\\
\mathfrak{T}_4\trkla{s}\trkla{\omega}&:=\chiTh\kappa R\trkla{s}^{\p-1}\sum_{k\in\Zx,\,l\in\Zy}\iO\Ihx{\Ihyloc{\pary\trkla{\skla{\g{kl},\omega}_{L^2} u\h\gt{kl}} G^\prime\trkla{u\h}}}\dx\dy\,
\end{align}
\end{subequations}
mapping $Q^{1/2}_yL^2\trkla{\Om}$ onto $\mathds{R}$.
Recalling \eqref{eq:energyentropy:tmp:hilbertschmidt1} and \eqref{eq:energyentropy:tmp:hilberschmidt2}, we obtain
\begin{align}
\begin{split}
&\rkla{\sum_{k,\,l\in\mathds{Z}}\abs{\mathfrak{T}_1\trkla{s}\trkla{Q^{1/2}_x \g{kl}}}^2}^{1/2} \\
&=\chiTh R\trkla{s}^{\p-1}\rkla{\sum_{k\in\Zx,\,l\in\Zy}\lx{kl}^2\rkla{\iO\Ihy{\Ihxloc{\parx\trkla{u\h\gt{kl}} p\h}}\dx\dy}^2}^{1/2}\\
&\leq \chiTh C R\trkla{s}^{\p-1}\left(\sum_{k\in\Zx,\,l\in\Zy}\lx{kl}^2\rkla{\norm{\g{kl}}_{L^\infty\trkla{\Om}}^2+h^\varepsilon\norm{\g{kl}}_{W^{3,\infty}\trkla{\Om}}^2}\right.\\
&~\times\left.\rkla{\!\norm{\Delta\h u\h}\h^2+\!\iO\Ihy{\tabs{\parx u\h}^2\Ihx{ u\h^{-p-2}}}\dx\dy\!+\! h^\varepsilon\!\!\iO\!\Ihy{\tabs{\parx\Delta\h u\h}^2}\dx\dy \!+\!R\trkla{s}}^{\!2}\right)^{\!1/2}\\
&\revy{\leq\chiTh C R\trkla{s}^{\p-1}\left(\norm{\Delta\h u\h}\h^2+ \iO\Ihy{\tabs{\parx u\h}^2\Ihx{ u\h^{-p-2}}}\dx\dy \right.}\\
&\revy{\qquad\qquad\qquad\qquad\qquad\qquad\qquad\qquad\qquad+\left.h^\varepsilon\iO\Ihy{\tabs{\parx\Delta\h u\h}^2}\dx\dy +R\trkla{s}\right)\,.}
\end{split}
\end{align}
Similarly, we have
\begin{multline}
\rkla{\sum_{k,\,l\in\mathds{Z}}\abs{\mathfrak{T}_2\trkla{s}\trkla{Q^{1/2}_y \g{kl}}}^2}^{1/2}\\
\leq \chiTh C R\trkla{s}^{\p-1}\rkla{\norm{\Delta\h u\h}\h^2+\iO\Ihx{\tabs{\pary u\h}^2\Ihy{ u\h^{-p-2}}}\dx\dy}\\
+\chiTh C R\trkla{s}^{\p-1}\rkla{h^\varepsilon\iO\Ihx{\tabs{\pary\Delta\h u\h}^2}\dx\dy +R\trkla{s}}\,.
\end{multline}
Following the computations in \eqref{eq:energyentropy:tmp:hilberschmidt3}, we obtain
\begin{align}
\rkla{\sum_{k,\,l\in\mathds{Z}}\abs{\mathfrak{T}_3\trkla{s}\trkla{Q^{1/2}_x \g{kl}}}^2}^{1/2}&\leq\chiTh C\kappa R\trkla{s}^{\p-1}\rkla{\iO\Ihy{\Ihx{G^{\prime\prime}\trkla{u\h}}\tabs{\parx u\h}^2}\dx\dy +C}\,,\\
\rkla{\sum_{k,\,l\in\mathds{Z}}\abs{\mathfrak{T}_4\trkla{s}\trkla{Q^{1/2}_y \g{kl}}}^2}^{1/2}&\leq\chiTh C\kappa R\trkla{s}^{\p-1}\rkla{\iO\Ihx{\Ihy{G^{\prime\prime}\trkla{u\h}}\tabs{\pary u\h}^2}\dx\dy +C}\,.
\end{align}
Therefore, the Burkholder-Davis-Gundy inequality yields together with Young's inequality and $R\trkla{s}\geq\alpha$
\begin{align}
\begin{split}
\mathds{E}&\ekla{\sup_{t\in\tekla{0,\Tmax}}\abs{ II_a}}\\
&\leq C\mathds{E}\left[\left(\iTmaxTh R\trkla{s}^{2\p-2}\left(\norm{\Delta\h u\h}\h^2+\iO\Ihy{\tabs{\parx u\h}^2\Ihx{ u\h^{-p-2}}}\dx\dy\right.\right.\right.\\
&\qquad+\left.\left.\left.h^\varepsilon\iO\Ihy{\tabs{\parx\Delta\h u\h}^2}\dx\dy +R\trkla{s}\right)\ds\right)^{1/2}\right]\\
&\leq \tfrac18\mathds{E}\ekla{\sup_{t\in\tekla{0,\Tmax\wedge T\h}} R\trkla{t}^\p} \\
&\quad+ C\alpha^{-1}\mathds{E}\left[\iTmaxTh R\trkla{s}^{\p-1}\left(\norm{\Delta\h u\h}\h^2+\iO\Ihy{\tabs{\parx u\h}^2\Ihx{ u\h^{-p-2}}}\dx\dy\right.\right.\\
&\qquad\qquad+\left.\left.h^\varepsilon\iO\Ihy{\tabs{\parx\Delta\h u\h}^2}\dx\dy +R\trkla{s}\right)\ds\right]\,,
\end{split}
\end{align}
\begin{align}
\begin{split}
\mathds{E}&\ekla{\sup_{t\in\tekla{0,\Tmax}}\abs{ II_b}}\leq \tfrac18\mathds{E}\ekla{\sup_{t\in\tekla{0,\Tmax\wedge T\h}} R\trkla{t}^\p} \\
&\qquad+ C\alpha^{-1}\mathds{E}\left[\iTmaxTh R\trkla{s}^{\p-1}\left(\norm{\Delta\h u\h}\h^2+\iO\Ihx{\tabs{\pary u\h}^2\Ihy{ u\h^{-p-2}}}\dx\dy\right.\right.\\
&\qquad\qquad+\left.\left.h^\varepsilon\iO\Ihx{\tabs{\pary\Delta\h u\h}^2}\dx\dy +R\trkla{s}\right)\ds\right]\,,
\end{split}
\end{align}
and
\begin{multline}
\mathds{E}\ekla{\sup_{t\in\tekla{0,\Tmax}}\tabs{II_c}+\sup_{t\in\tekla{0,\Tmax}}\tabs{II_d}}\leq \tfrac14\mathds{E}\ekla{\sup_{t\in\tekla{0,\Tmax\wedge T\h}} R\trkla{t}^{\p-1}}\\
+C\alpha^{-1}\kappa^2\mathds{E}\ekla{\iTmaxTh R\trkla{s}^{\p-1} \rkla{\iO\Ihy{\Ihx{G^{\prime\prime}\trkla{u\h}\tabs{\parx u\h}^2}}}}\\
+C\alpha^{-1}\kappa^2\mathds{E}\ekla{\iTmaxTh R\trkla{s}^{\p-1} \rkla{\iO\Ihx{\Ihy{G^{\prime\prime}\trkla{u\h}}\tabs{\pary u\h}^2}\dx\dy+C}}\,.
\end{multline}
Collecting the above results, we obtain
\begin{align}
\begin{split}
\expected{\sup_{t\in\tekla{0,\Tmax}} \tabs{II}}\leq&\,\tfrac12\expected{\sup_{t\in\tekla{0,\Tmax}} R\trkla{t}^\p} +C\alpha^{-1}\expected{\int_{0}^{T\h} R\trkla{s}^{\p-1} \norm{\Delta\h u\h}\h^2\dx\dy\ds}\\
&+C\alpha^{-1}\expected{\int_0^{T\h} R\trkla{s}^{\p-1}\iO\Ihy{\meanspX{u\h}\tabs{\parx u\h}^2}\dx\dy\ds}\\
&+C\alpha^{-1}\expected{\int_0^{T\h} R\trkla{s}^{\p-1}\iO\Ihx{\meanspY{u\h}\tabs{\pary u\h}^2}\dx\dy\ds}\\
&+C\alpha^{-1}\expected{\int_0^{t\h}R\trkla{s}^{\p-1} h^\varepsilon\iO\Ihy{\tabs{\parx\Delta\h u\h}^2}\dx\dy\ds}\\
&+C\alpha^{-1}\expected{\int_0^{T\h} R\trkla{s}^{\p-1}h^\varepsilon\iO\Ihx{\tabs{\pary\Delta\h u\h}^2}\dx\dy\ds}\\
&+C\alpha^{-1}\kappa^4\expected{\int_0^{T\h} R\trkla{s}^\p\ds}\,,
\end{split}
\end{align}
where we again used Lemma \ref{lem:oscillation} and mimicked  \eqref{eq:tmp:est:gprimeprime}.
Collecting the intermediate results established above, we obtain for $\kappa$ sufficiently large
\begin{multline}
\expected{\sup_{t\in\tekla{0,\Tmax}} R\trkla{t}^\p} + \expected{\int_0^{T\h} R\trkla{s}^{\p-1}\iO\Ihy{\meanGinvX{u\h}\tabs{\parx p\h}^2}\dx\dy\ds}\\
+\expected{\int_0^{T\h} R\trkla{s}^{\p-1}\iO\Ihx{\meanGinvY{u\h}\tabs{\pary u\h}^2}\dx\dy\ds}\\
 +\expected{\int_0^{T\h} R\trkla{s}^{\p-1}\norm{\Delta\h u\h}\h^2\ds} +\expected{\int_0^{T\h} R\trkla{s}^{\p-1}h^\varepsilon\iO\Ihy{\tabs{\parx\Delta\h u\h}^2}\dx\dy\ds}\\
 +\expected{\int_0^{T\h} R\trkla{s}^{\p-1}h^\varepsilon\iO\Ihx{\tabs{\pary\Delta\h u\h}^2}\dx\dy\ds} \\
 + \expected{\int_0^{T\h} R\trkla{s}^{\p-1}\iO\Ihy{\meanspX{u\h}\tabs{\parx u\h}^2}\dx\dy\ds}\\
 + \expected{\int_0^{T\h} R\trkla{s}^{\p-1}\iO\Ihx{\meanspY{u\h}\tabs{\pary u\h}^2}\dx\dy\ds}\\
 \leq \expected{R\trkla{0}^\p} +C\kappa^4\expected{\int_0^{T\h} R\trkla{s}^\p\ds}\,.
\end{multline}
Applying Gronwall's lemma concludes the proof.
\end{proof}
\subsection{Hölder continuity in time}
For compactness in time, the first step is to establish uniform Hölder continuity  for the stochastic integral.
In particular, we will prove the following lemma.
\begin{lemma}\label{lem:stochint}
Let $\Tmax>0$, $\p>1$, $\nu\in\trkla{0,\tfrac12}$ and let the Assumptions \ref{item:S}, \ref{item:initial}, \ref{item:potential}, \ref{item:stoch}, \ref{item:regularization}, and \ref{item:stochbasis:boundthirdderivative} hold true.
If $2\nu \p>1$ holds, for solutions $\trkla{u\h,p\h}$ to \eqref{eq:semidiscreteSTFE} the stochastic integral
\begin{align}
I\h\trkla{t}:=\sum_{i=1}^{\dim\Uhx}\sum_{j=1}^{\dim\Uhy}\sum_{k\in\Zx,\,l\in\Zy} I_{ijkl}\trkla{t}\ex{i}\trkla{\tilde{x}}\ey{j}\trkla{\tilde{y}}
\end{align}
with
\begin{align}
\begin{split}
I_{ijkl}\trkla{t}:=&\,M_{ij}^{-1}\itTh\iO \Ihy{\Ihxloc{\parx\trkla{\lx{kl} u\h\gt{kl}}\ex{i}\trkla{x}\ey{i}\trkla{y}}}\dx\dy\dbx{kl}\\
&+M_{ij}^{-1}\itTh\iO \Ihx{\Ihyloc{\pary\trkla{\lx{kl} u\h\gt{kl}}\ex{i}\trkla{x}\ey{i}\trkla{y}}}\dx\dy\dby{kl}
\end{split}
\end{align}
is contained in $L^{2\p}\trkla{\Omega;C^{\beta}\trkla{\tekla{0,\Tmax};L^2\trkla{\Om}}}$ with $\beta:=\nu-\tfrac1{2\p}$.
\end{lemma}
\begin{proof}
According to Lemma 2.1 in \cite{Flandoli1995} it suffices to show that
\begin{multline}
\hat{Z}^x\trkla{s}\trkla{\omega}:=\chiTh \sum_{i=1}^{\dim\Uhx}\sum_{j=1}^{\dim\Uhy}M_{ij}^{-1}\ex{i}\trkla{\tilde{x}}\ey{j}\trkla{\tilde{y}}\\
\times\iO\Ihy{\Ihxloc{\parx\trkla{u\h\sum_{k\in\Zx,\,l\in\Zy}\skla{\g{kl},\omega}_{L^2}\gt{kl}}\ex{i}\trkla{x}\ey{j}\trkla{y}}}\dx\dy 
\end{multline}
and
\begin{multline}
\hat{Z}^y\trkla{s}\trkla{\omega}:=\chiTh \sum_{i=1}^{\dim\Uhx}\sum_{j=1}^{\dim\Uhy}M_{ij}^{-1}\ex{i}\trkla{\tilde{x}}\ey{j}\trkla{\tilde{y}}\\
\times\iO\Ihx{\Ihyloc{\pary\trkla{u\h\sum_{k\in\Zx,\,l\in\Zy}\skla{\g{kl},\omega}_{L^2}\gt{kl}}\ex{i}\trkla{x}\ey{j}\trkla{y}}}\dx\dy 
\end{multline}
are progressively measurable and contained in $L^{2\p}\trkla{\Omega\times\trkla{0,\Tmax};L_2\trkla{L^2\trkla{\Om};L^2\trkla{\Om}}}$ with a uniform bound in $h$ to establish $I\h\in L^{2\p}\trkla{\Omega;W^{\nu,2\p}\trkla{{0,\Tmax};L^2\trkla{\Om}}}$.
Then the continuous embedding 
\begin{align}
W^{\nu,2\p}\trkla{{0,\Tmax};L^2\trkla{\Omega}}\hookrightarrow C^{\nu-\tfrac1{2\p}}\trkla{\tekla{0,\Tmax};L^2\trkla{\Om}}
\end{align}
completes the proof.
Recalling the computations from \eqref{eq:energyentropy:tmp:1} and using \ref{item:stochbasis:bounds}, we immediately obtain the bounds
\begin{align}
\norm{\hat{Z}^x}_{L_2\trkla{L^2\trkla{\Om};L^2\trkla{\Om}}}^2\leq \chiTh C\iO\Ihy{\tabs{\parx u\h}^2}\dx\dy +C\iO\Ihxy{\tabs{u\h}^2}\dx\dy\,,\\
\norm{\hat{Z}^y}_{L_2\trkla{L^2\trkla{\Om};L^2\trkla{\Om}}}^2\leq \chiTh C\iO\Ihx{\tabs{\pary u\h}^2}\dx\dy +C\iO\Ihxy{\tabs{u\h}^2}\dx\dy\,.
\end{align}\\
Progressive measurability is satisfied due to the pathwise continuity of the $u\h$ $\Prob$-almost surely.
Hence, the result follows by Proposition \ref{prop:energyentropyestimate}.
\end{proof}
In order to show compactness in time, we shall use Lemma \ref{lem:stochint} to establish the Hölder continuity of $u\h$ as a mapping from $[0,T_{max}]$ into appropriate Sobolev spaces.

\begin{lemma}\label{lem:hoelderInTime}
Let $\Tmax>0$ and let the Assumptions \ref{item:S}, \ref{item:initial}, \ref{item:potential}, \ref{item:stoch}, \ref{item:regularization}, and \ref{item:stochbasis:boundthirdderivative} hold true.
Then, for $\p$ sufficiently large, a solution $u\h$ to \eqref{eq:semidiscreteSTFE} is uniformly bounded in $L^\sigma\trkla{\Omega;C^{1/4}\trkla{\tekla{0,\Tmax};\trkla{H^1_\per\trkla{\Om}}^\prime}}$ for $\sigma<8/5$, i.e.
\begin{align}
\expected{\rkla{\sup_{t_1,t_2\in\tekla{0,\Tmax}}\frac{\norm{u\h\trkla{t_2}-u\h\trkla{t_1}}_{\trkla{H^1_\per\trkla{\Om}}^\prime}}{\tabs{t_2-t_1}^{1/4}}}^\sigma}\leq C\,.
\end{align}
\end{lemma}
\begin{proof}
Denoting the standard $L^2$-projection onto $\Uhxy$ by $\projop$, we obtain
\begin{multline}
\norm{u\h\trkla{t_2}-u\h\trkla{t_1}}_{\trkla{H^1_\per\trkla{\Om}}^\prime}=\sup_{0\neq\psi\in H^1_\per\trkla{\Om}}\rkla{\norm{\psi}_{H^1\trkla{\Om}}^{-1}\abs{\iO \rkla{u\h\trkla{t_2}-u\h\trkla{t_1}}\psi\dx\dy}}=\\
\leq\, \sup_{0\neq\psi\in H^1_\per\trkla{\Om}}\rkla{\norm{\psi}_{H^1\trkla{\Om}}^{-1} \abs{\iO\Ihxy{\rkla{u\h\trkla{t_2}-u\h\trkla{t_1}}\proj{\psi}}\dx\dy}}\\
+\, \sup_{0\neq\psi\in H^1_\per\trkla{\Om}}\rkla{\norm{\psi}_{H^1\trkla{\Om}}^{-1} \abs{\iO\trkla{\ids-\Ihxyop}\tgkla{\rkla{u\h\trkla{t_2}-u\h\trkla{t_1}}\proj{\psi}}\dx\dy}}\\
=:\, \sup_{0\neq\psi\in H^1_\per\trkla{\Om}}\rkla{\norm{\psi}_{H^1\trkla{\Om}}^{-1}\tabs{I}}+\sup_{0\neq\psi\in H^1_\per\trkla{\Om}}\rkla{\norm{\psi}_{H^1\trkla{\Om}}^{-1}\tabs{II}}\,.
\end{multline}
To derive an estimate for $\tabs{I}$, we start with the identity
\begin{multline}\label{eq:tmp:hoelder:weakform}
\iO\Ihxy{\trkla{u\h\trkla{t_2}-u\h\trkla{t_1}}\phi\h}\dx\dy+\int_{t_1\wedge T\h}^{t_2\wedge T\h}\iO\Ihy{\meanGinvX{u\h}\parx p\h\parx\phi\h}\dx\dy\ds\\
+\int_{t_1\wedge T\h}^{t_2\wedge T\h}\iO\Ihx{\meanGinvY{u\h}\pary p\h\pary\phi\h}\dx\dy\ds=\iO\Ihxy{\trkla{I\h\trkla{t_2}-I\h\trkla{t_1}}\phi\h}\dx\dy
\end{multline}
resulting from \eqref{eq:semidiscreteSTFE:1} for $0\leq t_1< t_2$.
Using \eqref{eq:tmp:hoelder:weakform} with $\phi\h=\proj{\psi}$ and Hölder's inequality, we deduce the estimate
\begin{align}\label{eq:hoelder:I}
\begin{split}
\abs{I}\leq&\,\rkla{\int_{t_1\wedge T\h}^{t_2\wedge T\h}\iO\Ihy{\meanGinvX{u\h}\tabs{\parx p\h}^2}\dx\dy\ds}^{1/2}\\
&\qquad\times\rkla{\int_{t_1\wedge T\h}^{t_2\wedge T\h}\iO \Ihy{\meanGinvX{u\h}\tabs{\parx \proj{\psi}}^2}\dx\dy\ds}^{1/2}\\
&+\rkla{\int_{t_1\wedge T\h}^{t_2\wedge T\h}\iO\Ihx{\meanGinvY{u\h}\tabs{\pary p\h}^2}\dx\dy\ds}^{1/2}\\
&\qquad\times\rkla{\int_{t_1\wedge T\h}^{t_2\wedge T\h}\iO \Ihx{\meanGinvY{u\h}\tabs{\pary \proj{\psi}}^2}\dx\dy\ds}^{1/2}\\
&+C\norm{I\h\trkla{t_2}-I\h\trkla{t_1}}_{L^2\trkla{\Om}}\norm{\proj{\psi}}_{L^2\trkla{\Om}}\,.
\end{split}
\end{align}
In order to derive bounds for the first term on the right-hand side, we use the estimate
\begin{align}
\begin{split}
\left(\int_{t_1\wedge T\h}^{t_2\wedge T\h}\right.&\left.\iO\Ihy{\meanGinvX{u\h}\tabs{\parx\proj{\psi}}^2}\dx\dy\ds\right)^{1/2}\\
&\leq C\rkla{\int_{t_1 \wedge T\h}^{t_2\wedge T\h}\norm{u\h}_{L^\infty\trkla{\Om}}^2\norm{\parx \proj{\psi}}_{L^2\trkla{\Om}}^2\ds}^{1/2}\\
&\leq C\rkla{\trkla{t_2-t_1}^{3/4}\norm{u\h}_{L^8\trkla{0,\Tmax;L^\infty\trkla{\Om}}}^2 }^{1/2}\norm{\psi}_{H^1\trkla{\Om}}\\
&\leq C\trkla{t_2-t_1}^{3/8}\norm{\psi}_{H^1\trkla{\Om}}\norm{\Delta\h u\h}_{L^2\trkla{0,\Tmax;L^2\trkla{\Om}}}^{1/4}\norm{u\h}_{L^\infty\trkla{0,\Tmax;H^1\trkla{\Om}}}^{3/4}\\
&\qquad+C\trkla{t_2-t_1}^{3/8}\norm{\psi}_{H^1\trkla{\Om}} \norm{u\h}_{L^\infty\trkla{0,\Tmax;H^1\trkla{\Om}}}\,.
\end{split}
\end{align}
In the last step, we used the inequality
\begin{align*}
\norm{u\h}_{L^8\trkla{0,\Tmax;L^\infty\trkla{\Om}}}\!\leq\! C\norm{\Delta\h u\h}_{L^2\trkla{0,\Tmax;L^2\trkla{\Om}}}^{1/4}\norm{u\h}_{L^\infty\trkla{0,\Tmax;H^1\trkla{\Om}}}^{3/4} \!+C\norm{u\h}_{L^\infty\trkla{0,\Tmax;H^1{\trkla{\Om}}}},
\end{align*}
which follows from the discrete Gagliardo--Nirenberg inequality proven in Corollary \ref{cor:embedding}.
Using similar computations for the second term on the right-hand side of \eqref{eq:hoelder:I} and estimating the remaining term with help of Lemma \ref{lem:stochint}, we obtain
\begin{align}
\begin{split}
&\expected{\sup_{0\neq\psi\in H^1_\per\trkla{\Om}}\rkla{\norm{\psi}_{H^1\trkla{\Om}}^{-1}\tabs{I}}^\sigma}\\
&\,\leq C\trkla{t_2-t_1}^{3\sigma/8}\mathds{E}\left[ \rkla{\int_{t_1\wedge T\h}^{t_2\wedge T\h}\iO\Ihy{\meanGinvX{u\h}\tabs{\parx p\h}^2}\dx\dy\ds}^{\sigma/2}\right.\\
&\qquad\left.\times\rkla{\norm{\Delta\h u\h}_{L^2\trkla{0,\Tmax;L^2\trkla{\Om}}}^{\sigma/4}\norm{u\h}_{L^\infty\trkla{0,\Tmax;H^1\trkla{\Om}}}^{3\sigma/4}+\norm{u\h}_{L^\infty\trkla{0,\Tmax;H^1\trkla{\Om}}}^\sigma}\right]\\
&\quad+ C\trkla{t_2-t_1}^{3\sigma/8}\mathds{E}\left[ \rkla{\int_{t_1\wedge T\h}^{t_2\wedge T\h}\iO\Ihx{\meanGinvY{u\h}\tabs{\pary p\h}^2}\dx\dy\ds}^{\sigma/2}\right.\\
&\qquad\left.\times\rkla{\norm{\Delta\h u\h}_{L^2\trkla{0,\Tmax;L^2\trkla{\Om}}}^{\sigma/4}\norm{u\h}_{L^\infty\trkla{0,\Tmax;H^1\trkla{\Om}}}^{3\sigma/4}+\norm{u\h}_{L^\infty\trkla{0,\Tmax;H^1\trkla{\Om}}}^\sigma}\right]\\
&\quad+C\trkla{t_2-t_1}^{\sigma\beta}\,.
\end{split}
\end{align}
Applying Young's inequality, we end up with
\begin{align}
\begin{split}
&\expected{\sup_{0\neq\psi\in H^1_\per\trkla{\Om}}\rkla{\norm{\psi}_{H^1\trkla{\Om}}^{-1}\tabs{I}}^\sigma}\\
&\,\leq C\trkla{t_2-t_1}^{3\sigma/8}\expected{\int_{t_1\wedge T\h}^{t_2\wedge T\h}\iO\Ihy{\meanGinvX{u\h}\tabs{\parx p\h}^2}\dx\dy\ds}\\
&\quad+C\trkla{t_2-t_1}^{3\sigma/8}\expected{\int_{t_1\wedge T\h}^{t_2\wedge T\h}\iO\Ihx{\meanGinvY{u\h}\tabs{\pary p\h}^2}\dx\dy\ds}\\
&\quad+C\trkla{t_2-t_1}^{3\sigma/8}\expected{\norm{\Delta\h u\h}_{L^2\trkla{0,\Tmax;L^2\trkla{\Om}}}^{\sigma/\trkla{4-2\sigma}}\norm{u\h}_{L^\infty\trkla{0,\Tmax;H^1\trkla{\Om}}}^{3\sigma/\trkla{4-2\sigma}}+\norm{u\h}_{L^\infty\trkla{0,\Tmax;H^1\trkla{\Om}}}^{2/\trkla{2-\sigma}}}\\
&\quad+C\trkla{t_2-t_1}^{\sigma\beta}\,.
\end{split}
\end{align}
As we assumed $\sigma<8/5$, we have $\sigma/\trkla{4-2\sigma}<2$.
Therefore, we may apply Young's inequality once again to obtain
\begin{multline}
\norm{\Delta\h u\h}_{L^2\trkla{0,\Tmax;L^2\trkla{\Om}}}^{\sigma/\trkla{4-2\sigma}}\norm{u\h}_{L^\infty\trkla{0,\Tmax;H^1\trkla{\Om}}}^{3\sigma/\trkla{4-2\sigma}}\\
\leq C\norm{\Delta\h u\h}_{L^2\trkla{0,\Tmax;L^2\trkla{\Om}}}^2 +C \norm{u\h}_{L^\infty\trkla{0,\Tmax;H^1\trkla{\Om}}}^{6\sigma/\trkla{8-5\sigma}}\,.
\end{multline}
In view of Proposition \ref{prop:energyentropyestimate}, we conclude
\begin{align}
\expected{\sup_{0\neq\psi\in H^1_\per\trkla{\Om}}\rkla{\norm{\psi}_{H^1\trkla{\Om}}^{-1}\tabs{I}}^\sigma}\leq C\trkla{t_2-t_1}^{3\sigma/8}+C\trkla{t_2-t_1}^{\sigma\beta}
\end{align}
for $\p$ large enough.

Using \eqref{eq:Ihxy:lp}, an inverse estimate, and the stability properties of $\projop$, we obtain
\begin{align}
\tabs{II}\leq C h\norm{u\h\trkla{t_2}-u\h\trkla{t_1}}_{L^2\trkla{\Om}}\norm{\psi}_{H^1\trkla{\Om}}\,.
\end{align}
To obtain an estimate for $\norm{u\h\trkla{t_2}-u\h\trkla{t_1}}_{L^2\trkla{\Om}}$, we again start with the identity \eqref{eq:tmp:hoelder:weakform} and choose $\phi\h=\trkla{u\h\trkla{t_2}-u\h\trkla{t_1}}$.
This provides the estimate
\begin{align}
\begin{split}
\norm{u\h\trkla{t_2}-u\h\trkla{t_1}}\h^2\leq& \abs{\int_{t_1\wedge T\h}^{t_2\wedge T\h}\iO\Ihy{\meanGinvX{u\h}\parx p\h\parx\trkla{u\h\trkla{t_2}-u\h\trkla{t_1}}}\dx\dy\ds}\\
&+\abs{\int_{t_1\wedge T\h}^{t_2\wedge T\h}\iO\Ihx{\meanGinvY{u\h}\pary p\h\pary\trkla{u\h\trkla{t_2}-u\h\trkla{t_1}}}\dx\dy\ds}\\
&+\norm{I\h\trkla{t_2}-I\h\trkla{t_1}}_{L^2\trkla{\Om}}\norm{u\h\trkla{t_2}-u\h\trkla{t_1}}_{L^2\trkla{\Om}}\\
=:&\, N_1+N_2+N_3\,.
\end{split}
\end{align}
Using Hölder's inequality and Young's inequality, we compute
\begin{align}
\begin{split}
N_1\leq& \int_{t_1\wedge T\h}^{t_2\wedge T\h} \rkla{\iO\Ihy{\abs{\meanGinvX{u\h}\parx p\h}^{4/3}}\dx\dy}^{3/4}\\
&\qquad\times\rkla{\iO\Ihy{\tabs{\parx\trkla{u\trkla{t_2}-u\h\trkla{t_1}}}^4}\dx\dy}^{1/4}\ds\\
\leq&\,C\trkla{t_2-t_1}^{1/2}\int_{t_1\wedge T\h}^{t_2\wedge T\h}\rkla{\iO\Ihy{\abs{\meanGinvX{u\h}\parx p\h}^{4/3}}\dx\dy}^{3/2}\ds\\
&+C\trkla{t_2-t_1}^{-1/2}\int_{t_1\wedge T\h}^{t_2\wedge T\h}\rkla{\iO\Ihy{\tabs{\parx\trkla{u\h\trkla{t_2}-u\h\trkla{t_1}}}^4}\dx\dy}^{1/2}\ds\\
=:&\,\trkla{t_2-t_1}^{1/2}N_{1_a}+ \trkla{t_2-t_1}^{-1/2}N_{1_b}\,.
\end{split}
\end{align}
For the first term, we obtain from Hölder's inequality
\begin{align}
\begin{split}
N_{1_a}\leq&\,C \int_{t_1\wedge T\h}^{t_2\wedge T\h}\rkla{\iO\Ihy{\meanGinvX{u\h}\tabs{\parx p\h}^2}\dx\dy}\rkla{\iO\Ihy{\abs{\meanGinvX{u\h}}^2}\dx\dy}^{1/2}\ds\\
\leq &\,C \int_{t_1\wedge T\h}^{t_2\wedge T\h} R\trkla{s}\iO\Ihy{\meanGinvX{u\h}\tabs{\parx p\h}^2}\dx\dy\ds\,,
\end{split}
\end{align}
where we used
\begin{multline}
\rkla{\iO\Ihy{\abs{\meanGinvX{u\h}}^2}\dx\dy}^{1/2}
\leq C\rkla{\iO \tabs{u\h}^4\dx\dy}^{1/2}\leq C\norm{u\h}_{H^1\trkla{\Om}}^2\leq CR\trkla{s}\,.
\end{multline}
Concerning $N_{1_b}$, we obtain from a discrete version of the Gagliardo--Nirenberg inequality (cf. Corollary \ref{cor:embedding})
\begin{align}
\begin{split}
N_{1_b}\leq&\,C\trkla{t_2-t_1}\norm{u\h\trkla{t_2}-u\h\trkla{t_1}}_{W^{1,4}\trkla{\Om}}^2\\
\leq&\,C\trkla{t_2-t_1} \norm{\Delta\h u\h\trkla{t_2}-\Delta\h u\h\trkla{t_1}}\h\norm{u\h\trkla{t_2}-u\h\trkla{t_1}}_{H^1\trkla{\Om}} \\
&+C\trkla{t_2-t_1}\norm{u\h\trkla{t_2}-u\h\trkla{t_1}}_{H^1\trkla{\Om}}^2\,.
\end{split}
\end{align}
Estimates for $N_2$ can be derived in a similar manner.
To derive bounds for $N_3$, we use the results from Lemma \ref{lem:stochint} and obtain
\begin{align}
\begin{split}\label{eq:nikolskii:tmp:1}
\expected{N_3}\leq&\, \delta\expected{\norm{u\h\trkla{t_2}-u\h\trkla{t_1}}\h^2}+C_\delta\expected{\norm{I\h\trkla{t_2}-I\h\trkla{t_1}}_{L^2\trkla{\Om}}^2}\\
\leq&\delta\expected{\norm{u\h\trkla{t_2}-u\h\trkla{t_1}}\h^2} + C_\delta \trkla{t_2-t_1}^{2\beta}\,.
\end{split}
\end{align}
as the first term on the right-hand side of \eqref{eq:nikolskii:tmp:1} can be absorbed for $\delta$ small enough, we obtain
\begin{align}
\begin{split}
&\!\!\!\!\expected{h^\sigma \norm{u\h\trkla{t_2}-u\h\trkla{t_1}}_{L^2\trkla{\Om}}^\sigma}\\
&\leq C\trkla{t_2-t_1}^{\sigma/4} h^\sigma\expected{\int_{t_1\wedge T\h}^{t_2\wedge T\h} R\trkla{s}\iO\Ihy{\meanGinvX{u\h}\tabs{\parx p\h}^2}\dx\dy\ds}\\
&\quad+C\trkla{t_2-t_1}^{\sigma/4} h^\sigma\expected{\int_{t_1\wedge T\h}^{t_2\wedge T\h} R\trkla{s}\iO\Ihx{\meanGinvY{u\h}\tabs{\pary p\h}^2}\dx\dy\ds}\\
&\quad + C \trkla{t_2-t_1}^{\sigma/4}h^\sigma\expected{\norm{\Delta\h u\h\trkla{t_2}-\Delta\h u\h\trkla{t_1}}\h^{\sigma/2} \norm{u\h\trkla{t_2}-u\h\trkla{t_1}}_{H^1\trkla{\Om}}^{\sigma/2}} \\
&\quad+C\trkla{t_2-t_1}^{\sigma/4}h^\sigma\expected{\norm{u\h\trkla{t_2}-u\h\trkla{t_1}}^\sigma_{H^1\trkla{\Om}}} +Ch^\sigma \trkla{t_2-t_1}^{\sigma/4}+Ch^\sigma\trkla{t_2-t_1}^{\sigma\beta}\\
&\leq  C h^\sigma\trkla{t_2-t_1}^{\sigma/4}+ C\trkla{t_2-t_1}^{\sigma/4}\expected{\sup_{s\in\tekla{0,\Tmax}}\norm{u\h\trkla{s}}^\sigma_{H^1\trkla{\Om}}} +Ch^\sigma\trkla{t_2-t_1}^{\sigma\beta}\,,
\end{split}
\end{align}
where we used \eqref{eq:dqLap}.
Choosing $\beta\geq1/4$ (cf.~Lemma \ref{lem:stochint}) completes the proof.
\end{proof}

\section{Passage to the limit}\label{sec:limit}
\subsection{Compactness}
As $u\h$ is only strictly positive for $h>0$, we lack $h$-independent bounds on the pressure $p\h$.
Therefore, we consider the fluxes
\begin{align}\label{eq:def:jh}
J\h^x:=\Ihy{\sqrt{\meanGinvX{u\h}}\parx p\h}\,,&&J\h^y:=\Ihx{\sqrt{\meanGinvY{u\h}}\pary p\h}\,,
\end{align}
which are uniformly bounded in $L^2\trkla{\Omega;L^2\trkla{0,\Tmax;L^2\trkla{\Om}}}$.
Note that solutions $\trkla{u\h,p\h}$ to \eqref{eq:semidiscreteSTFE} may be equivalently characterized by $\trkla{u\h,\Delta\h u\h, J\h^x, J\h^y}$.
In the following, we consider these objects in the spaces
\begin{subequations}
\begin{align}
\Xu&:=C\trkla{\tekla{0,\Tmax};L^{q}\trkla{\Om}}\,,\qquad q<\infty\,,\\
\Xddu&:=\trkla{L^2\trkla{0,\Tmax;L^2\trkla{\Om}}}_{weak}\,,\\
\XJx&:=\trkla{L^2\trkla{0,\Tmax;L^2\trkla{\Om}}}_{weak}\,,\\
\XJy&:=\trkla{L^2\trkla{0,\Tmax;L^2\trkla{\Om}}}_{weak}\,.
\end{align}
\end{subequations}

\begin{lemma}\label{lem:tightness}
Let $\Tmax>0$ be arbitrary but fixed. Let $\trkla{u\h,\,\Delta\h u\h,\,J\h^x,\,J\h^y}\h$ be a sequence of discrete solutions to \eqref{eq:semidiscreteSTFE}.
Then the families of laws $\trkla{\mu_{u\h}}\h$, $\trkla{\mu_{\Delta\h u\h}}\h$, $\trkla{\mu_{J\h^x}}\h$, and $\trkla{\mu_{J\h^y}}\h$ are tight.
\end{lemma}
\begin{proof}
From Proposition \ref{prop:energyentropyestimate} and Lemma \ref{lem:hoelderInTime}, we obtain that $\trkla{u\h}\h$ is uniformly bounded in $L^2\trkla{\Omega;L^\infty\trkla{0,\Tmax;H^1\trkla{\Om}}}\cap L^{\sigma}\trkla{\Omega;C^{1/4}\trkla{\tekla{0,\Tmax};\trkla{H^1_\per\trkla{\Om}}^\prime}}$ for $\sigma<8/5$.
Due to the well-known compactness theorem by Simon (cf. \cite{Simon1987}), the ball $\overline{B}_R$ in $L^\infty\trkla{0,\Tmax;H^1_\per\trkla{\Om}}\cap C^{1/4}\trkla{\tekla{0,\Tmax};\trkla{H^1_\per\trkla{\Om}}^\prime}$ is a compact subset of $C\trkla{\tekla{0,\Tmax};L^{q}\trkla{\Om}}$.
Furthermore, we have for any $R>0$
\begin{align}
\begin{split}
\mu_{u\h}\rkla{\Xu\textbackslash\overline{B}_R}
&\,=\prob{\norm{u\h}^{\sigma}_{L^\infty\trkla{0,\Tmax;H^1\trkla{\Om}}}+\norm{u\h}_{C^{1/4}\trkla{\tekla{0,\Tmax};\trkla{H^1_\per\trkla{\Om}}^\prime}}^{\sigma}>R^{\sigma}}\\
&\,\leq R^{-\sigma}\expected{\norm{u\h}^2_{L^\infty\trkla{0,\Tmax;H^1\trkla{\Om}}}+C+\norm{u\h}_{C^{1/4}\trkla{\tekla{0,\Tmax};\trkla{H^1_\per\trkla{\Om}}^\prime}}^{\sigma}}\,,
\end{split}
\end{align}
which shows the tightness of $\trkla{\mu_{u\h}}\h$.
As closed balls in $L^2\trkla{0,\Tmax;L^2\trkla{\Om}}$ are compact in the weak topology, the tightness of $\trkla{\mu_{\Delta\h u\h}}\h$, $\trkla{\mu_{J\h^x}}\h$ and $\trkla{\mu_{J\h^y}}\h$ is a direct consequence of Markov's inequality and the bound obtained in Proposition \ref{prop:energyentropyestimate}.
\end{proof}
Following the lines of \cite{FischerGrun2018}, we introduce the Polish space
\begin{align}
\XW:=\trkla{C\trkla{\tekla{0,\Tmax};L^2\trkla{\Om}}}^2
\end{align}
as an additional path space.
Let $\mu_{\bs{W}}:=\trkla{\mu_{W^x},\mu_{W^y}}^T$ be the law of 
\begin{align}
\bs{W}=\trkla{\sum_{k,l\in\mathds{Z}}\lx{kl}\g{kl}\bx{kl},\sum_{k,l\in\mathds{Z}}\ly{kl}\g{kl}\by{kl}}^T\,.
\end{align}
As $\XW$ is a Polish space, $\mu_{\bs{W}}$ is a Radon measure and therefore regular from the interior, i.e.
\begin{align}
\mu_{\bs{W}}\rkla{\trkla{C\trkla{\tekla{0,\Tmax};L^2\trkla{\Om}}}^2}=\sup\gkla{\mu_{\bs{W}}\trkla{K}\,:\,K\subset \trkla{C\trkla{\tekla{0,\Tmax};L^2\trkla{\Om}}}^2\text{~compact}}\,.
\end{align}
To deal with the initial data, we introduce the space $\XuInit:= H^1_{\per}\trkla{\Om}$.
Together with the tightness results of Lemma \ref{lem:tightness}, we obtain the following result.
\begin{lemma}
On the path space $\X:=\Xu\times\Xddu\times\XJx\times\XJy\times\XW\times\XuInit$ the joint laws $\mu\h$ defined by
\begin{multline}
\mu\h\trkla{A\times B\times C\times D\times E\times F}:=\\\prob{\tgkla{u\h\in A}\cap\tgkla{\Delta\h u\h\in B}\cap\tgkla{J\h^x\in C}\cap\tgkla{J\h^y\in D}\cap\tgkla{\bs{W}\in E}\cap\tgkla{u^0\in F}}\,,
\end{multline}
for $h\in(0,1]$ are tight.
\end{lemma}
Using Jakubowski's theorem (cf.~\cite{Jakubowski1998}) which is a generalization of Skorokhod's theorem (cf.~\cite{Skorokhod1956}), we obtain the following result.

\begin{proposition}\label{prop:convergence1}
Let $\trkla{u\h,\,\Delta\h u\h,\, J\h^x,\,J\h^y}$ be solutions to \eqref{eq:semidiscreteSTFE} in the sense of Lemma \ref{lem:existence} defined on the same stochastic basis $\trkla{\Omega,\mathcal{F},\trkla{\mathcal{F}_t}_{t\geq0},\Prob}$ with respect to the Wiener process $\bs{W}$.
Then there exists a subsequence which we again denote by $\trkla{u\h, \Delta\h u\h,\, J\h^x, J\h^y}$ such that there are a stochastic basis $\trkla{\tilde{\Omega},\tilde{\mathcal{F}},\tilde{\Prob}}$, a sequence of random variables
\begin{subequations}
\begin{align}
\tilde{u}\h\,:&\,\tilde{\Omega}\rightarrow C\trkla{\tekla{0,\Tmax};L^{q}\trkla{\Om}}\quad\trkla{q<\infty}\,,\\
\widetilde{\Delta\h u\h}\,:&\,\tilde{\Omega}\rightarrow L^2\trkla{0,\Tmax;L^2\trkla{\Om}}\,,\label{eq:widetildedeltau}\\
\tilde{J}\h^x\,:&\,\tilde{\Omega}\rightarrow L^2\trkla{0,\Tmax;L^2\trkla{\Om}}\,,\\
\tilde{J}\h^y\,:&\,\tilde{\Omega}\rightarrow L^2\trkla{0,\Tmax;L^2\trkla{\Om}}\,,\\
\tilde{u}\h^0\,:&\,\tilde{\Omega}\rightarrow H^1_\per\trkla{\Om}\,,
\end{align}
\end{subequations}
a sequence of $\trkla{L^2\trkla{\Om}}^2$-valued processes $\tilde{\bs{W}}\h$ on $\tilde{\Omega}$, random variables
\begin{subequations}
\begin{align}
\tilde{u}&\,\in L^2\trkla{\tilde{\Omega};C\trkla{\tekla{0,\Tmax};L^{q}\trkla{\Om}}}\quad\trkla{q<\infty}\,,\\
\widetilde{\Delta u}&\,\in L^2\trkla{\tilde{\Omega};L^2\trkla{0,\Tmax;L^2\trkla{\Om}}}\,,\\
\tilde{J}^x&\,\in L^2\trkla{\tilde{\Omega};L^2\trkla{0,\Tmax;L^2\trkla{\Om}}}\,,\\
\tilde{J}^y&\,\in L^2\trkla{\tilde{\Omega};L^2\trkla{0,\Tmax;L^2\trkla{\Om}}}\,,\\
\tilde{u}^0&\,\in L^2\trkla{\tilde{\Omega};H^2_{\per}\trkla{\Om}}\,,
\end{align}
\end{subequations}
and an $L^2\trkla{\Om}$-valued process $\tilde{W}$ on $\tilde{\Omega}$ which satisfy the following properties:
\begin{itemize}
\item[i)] The law of $\trkla{\tilde{u}\h,\widetilde{\Delta\h u\h},\tilde{J}\h^x,\tilde{J}\h^y, \tilde{W}\h,\tilde{u}\h^0}$ on $\Xu\times\Xddu\times\XJx\times\XJy\times\XW\times\XuInit$ under $\tilde{\Prob}$ coincides for any $h$ with the law of $\trkla{u\h,\Delta\h u\h,J\h^x,J\h^y,W,u\h^0}$ under $\Prob$.
\item[ii)] The sequence $\trkla{\tilde{u}\h,\widetilde{\Delta\h u\h},\tilde{J}\h^x,\tilde{J}\h^y, \tilde{W}\h,\tilde{u}\h^0}$ converges $\tilde{\Prob}$-almost surely towards\linebreak $\trkla{\tilde{u},\widetilde{\Delta u},\tilde{J}^x,\tilde{J}^y,\tilde{W},\tilde{u}^0}$ in the topology of $\X$.
\end{itemize}
\end{proposition}
\begin{remark}
In particular, one may use the interval $\tekla{0,1}$ for $\tilde{\Omega}$, its standard Borel $\sigma$-algebra for $\tilde{\mathcal{F}}$, and the Lebesgue measure for $\tilde{\Prob}$ (cf.~\cite{Jakubowski1998}). 
\end{remark}

Similarly to $T\h$, we introduce the random stopping times 
\begin{align}
\tilde{T}\h:=\Tmax\wedge \inf\tgkla{t\geq0\,:\,\mathcal{E}\h\trkla{\tilde{u}\h\trkla{t}}\geq \Emax}\,.
\end{align}
\begin{lemma}\label{lem:stoppingtime}
Along a subsequence, the convergence $\lim_{h\searrow0} \tilde{T}\h=\Tmax$ holds $\tilde{\Prob}$-almost surely.
\end{lemma}
\begin{proof}
Following the lines of \cite{FischerGrun2018}, we compute for each $\tau\in(0,\Tmax]$
\begin{align}
\probt{\tgkla{\tilde{T}\h<\tau}}=\prob{\tgkla{T\h<\tau}}\leq C h^{\rho/(2+p)}\,.
\end{align}
Hence $\tilde{T}\h\rightarrow\Tmax$ in probability for $h\searrow0$, which implies the $\tilde{\Prob}$-almost sure convergence for a subsequence.
\end{proof}
\begin{lemma}\label{lem:identification}
\begin{subequations}
Under the Assumptions \ref{item:S}, \ref{item:initial}, \ref{item:potential}, \ref{item:stoch}, \ref{item:regularization}, and \ref{item:stochbasis:boundthirdderivative}, $\widetilde{\Delta\h u\h}$ can be identified $\tilde{\Prob}$-almost surely as the discrete Laplacian of $\tilde{u}\h$, i.e.
\begin{align}\label{eq:identify:discLap}
\widetilde{\Delta\h u\h}=\Delta\h\tilde{u}\h\,.
\end{align}
Furthermore, the flux components $\tilde{J}\h^x$ and $\tilde{J}\h^y$ can be identified $\tilde{\Prob}$-almost surely as
\begin{align}
\tilde{J}\h^x=\chiTth\Ihy{\sqrt{\meanGinvX{\tilde{u}\h}}\parx\rkla{-\Delta\h\tilde{u}\h+\Ihxy{F^\prime\trkla{\tilde{u}\h}} +h^\varepsilon \Delta\h\Delta\h\tilde{u}\h}}\,,\label{eq:identify:Jhx}\\
\tilde{J}\h^y=\chiTth\Ihx{\sqrt{\meanGinvY{\tilde{u}\h}}\pary\rkla{-\Delta\h\tilde{u}\h+\Ihxy{F^\prime\trkla{\tilde{u}\h}} +h^\varepsilon \Delta\h\Delta\h\tilde{u}\h}}\label{eq:identify:Jhy}\,.
\end{align}
\end{subequations}
\end{lemma}
\begin{proof}

As $\Delta\h u\h$ depends continuously on $u\h$ (cf.~\eqref{eq:dqLap}), \eqref{eq:identify:discLap} follows by equality of laws.\\
As for every fixed $h>0$, the functions $u\h$ and $\tilde{u}\h$ are almost surely in $C\trkla{\tekla{0,\Tmax};C\trkla{\overline{\Om}}}$, the stopping times $T\h\trkla{\omega}$ and $\tilde{T}\h\trkla{\tilde{\omega}}$ are also continuous functions of $u\h$ and $\tilde{u}\h$, respectively.
By inverse estimates (cf. Theorem 4.5.11 in \cite{BrennerScott}) and the oscillation lemma \ref{lem:oscillation}, the same holds true for the terms on the right-hand side of \eqref{eq:identify:Jhx} and \eqref{eq:identify:Jhy}.
In particular, the expectation
\begin{multline}
\tilde{\mathds{E}}\left[\left|\int_0^{\Tmax}\iO \tilde{J}\h^x\phi\dx\dy\dt\right.\right.\\
-\left.\left.\int_0^{\Tmax}\chiTth\Ihy{\sqrt{\meanGinvX{\tilde{u}\h}}  \parx\trkla{-\Delta\h\tilde{u}\h+\Ihxy{F^\prime\trkla{\tilde{u}\h}}+h^\varepsilon\Delta\h\Delta\h\tilde{u}\h}}\right|\right]
\end{multline}
coincides with 
\begin{multline}
\tilde{\mathds{E}}\left[\left|\int_0^{\Tmax}\iO {J}\h^x\phi\dx\dy\dt\right.\right.\\
-\left.\left.\int_0^{\Tmax}\chiTh\Ihy{\sqrt{\meanGinvX{{u}\h}}  \parx\trkla{-\Delta\h {u}\h+\Ihxy{F^\prime\trkla{{u}\h}}+h^\varepsilon\Delta\h\Delta\h {u}\h}}\right|\right]
\end{multline}
for arbitrary $\phi\in C^\infty\trkla{\tekla{0,\Tmax};\trkla{C_\per^\infty\trkla{\overline{\Om}}}}$.
As the latter one is equal to zero, this gives the claim w.r.t.~$\tilde{J}\h^x$.
The argumentation for $\tilde{J}\h^y$ is the same.

\end{proof}
\begin{corollary}\label{cor:tildedeltau}
Let the assumptions of Proposition \ref{prop:convergence1} and Lemma \ref{lem:identification} hold true.
Then the limit function $\widetilde{\Delta u}$ introduced in \eqref{eq:widetildedeltau} can be identified with the Laplacian of $\tilde{u}$ $\tilde{\Prob}$-almost surely.
\end{corollary}
\begin{proof}
Choosing $\phi\in C_\per^\infty\trkla{\overline{\Om}}$, a Taylor expansion provides $\dxm\dxp\phi+\dym\dyp\phi\rightarrow \Delta\phi$ in $L^\infty\trkla{\Om}$.
Therefore, shifting the discrete Laplacian onto the test function before passing to the limit provides the desired result.
\end{proof}
We proceed by showing that $\tilde{\bs{W}}$ and $\tilde{\bs{W}}\h$ are $Q$-Wiener processes adapted to suitably defined filtrations $\trkla{\tilde{\mathcal{F}}_t}_{t\geq0}$ and $\trkla{\tilde{\mathcal{F}}_{h,t}}_{t\geq0}$, respectively.
We define $\trkla{\tilde{\mathcal{F}}_t}_{t\geq0}$ to be the $\tilde{\Prob}$-augmented canonical filtration associated with $\trkla{\tilde{u},\tilde{\bs{W}},\tilde{u}^0}$, i.e.
\begin{align}
\tilde{\mathcal{F}}_t:=\sigma\trkla{\sigma\trkla{r_t\tilde{u},r_t\tilde{\bs{W}}}\cup\tgkla{N\in\tilde{\mathcal{F}}\,:\,\tilde{\Prob}\trkla{N}=0}\cup \sigma\trkla{\tilde{u}^0}}\,.
\end{align}
Here, $r_t$ is the restriction of a function defined on $\tekla{0,\Tmax}$ to the interval $\tekla{0,t}$ with $t\in\tekla{0,\Tmax}$.
Analogously, we introduce the filtrations $\trkla{\tilde{\mathcal{F}}_{h,t}}_{t\geq0}$ as the $\tilde{\Prob}$-augmented canonical filtration associated with $\trkla{\tilde{u}\h,\tilde{\bs{W}}\h,\tilde{u}\h^0}$
\begin{align}
\tilde{\mathcal{F}}_{h,t}:=\sigma\trkla{\sigma\trkla{r_t\tilde{u}\h,r_t\tilde{\bs{W}}\h}\cup\tgkla{N\in\tilde{\mathcal{F}}\,:\,\tilde{\Prob}\trkla{N}=0}\cup \sigma\trkla{\tilde{u}\h^0}}\,.
\end{align}

\begin{lemma}\label{lem:QWiener}
The processes $\tilde{\bs{W}}\h$ and $\tilde{\bs{W}}$ are Q-Wiener processes adapted to the filtrations $\trkla{\tilde{\mathcal{F}}_{h,t}}_{t\geq0}$ and $\trkla{\tilde{\mathcal{F}}_{t}}_{t\geq0}$. They can be written as
\begin{align}
\tilde{\bs{W}}\h\trkla{t}&=\sum_{\alpha\in\tgkla{x,\,y}}\sum_{k,\,l\in\Z}\lalpha{kl}\g{kl}\bthalpha{kl}\kbalpha&\text{and}&&\tilde{\bs{W}}\trkla{t}&=\sum_{\alpha\in\tgkla{x,\,y}}\sum_{k,\,l\in\Z}\lalpha{kl}\g{kl}\btalpha{kl}\kbalpha\,.
\end{align}
Here, $\trkla{\bthalpha{kl}}_{\alpha\in\tgkla{x,\,y},\,k,\,l\in\Z}$ and  $\trkla{\btalpha{kl}}_{\alpha\in\tgkla{x,\,y},\,k,\,l\in\Z}$ are families of independently and identically distributed Brownian motions with respect to $\trkla{\tilde{\mathcal{F}}_{h,t}}_{t\geq0}$ and $\trkla{\tilde{\mathcal{F}}_t}_{t\geq0}$.
\end{lemma}
For a proof we refer to \cite{FischerGrun2018}.
Combining the results of Proposition \ref{prop:convergence1} and Lemma \ref{lem:identification} with the discrete Gagliardo--Nirenberg inequality allows us to establish improved convergence results.
\begin{lemma}
Let $\tilde{u}\h$ and $\Delta\h\tilde{u}\h$ be the random variables identified in Proposition \ref{prop:convergence1}. 
Furthermore, let the Assumptions \ref{item:S}, \ref{item:initial}, \ref{item:potential}, \ref{item:stoch}, \ref{item:regularization}, and \ref{item:stochbasis:boundthirdderivative} hold true and let $q\in\trkla{1,\infty}$. Then $\tilde{u}\h$ converges strongly along a subsequence towards $\tilde{u}$ in $L^2\trkla{0,\Tmax;W^{1,q}\trkla{\Om}}$ $\tilde{\Prob}$-almost surely.
\end{lemma}
\begin{proof}
We follow the lines of Lemma 5.1 in \cite{Metzger2020}.
Using Hölder's inequality, we compute
\begin{align}
\norm{\tilde{u}\h-\tilde{u}}_{L^2\trkla{0,\Tmax;W^{1,q}\trkla{\Om}}}\leq C\norm{\tilde{u}\h-\tilde{u}}_{L^2\trkla{0,\Tmax;H^1\trkla{\Om}}}^{1/q}\norm{\tilde{u}\h-\tilde{u}}_{L^2\trkla{0,\Tmax;W^{1,2q-2}\trkla{\Om}}}^{\trkla{q-1}/q}\,.
\end{align}
Due to the discrete Gagliardo--Nirenberg inequality (cf.~Lemma \ref{lem:embedding}), we have\, $\tilde{u}\in L^2\trkla{0,\Tmax;W_\per^{1,q}\trkla{\Om}}$ $\tilde{\Prob}$-almost surely.
As $\tilde{u}\h$ is $\tilde{\Prob}$-almost surely in $\Uhxy\cap L^\infty\trkla{0,\Tmax;H^1_\per\trkla{\Om}}$ with $\Delta\h\tilde{u}\h\in L^2\trkla{0,\Tmax;L^2\trkla{\Om}}$, we may use the discrete Gagliardo--Nirenberg inequality (cf.~Corollary \ref{cor:embedding}) to show that $\tilde{u}$ is also $\tilde{\Prob}$-almost surely in $L^2\trkla{0,\Tmax;W_\per^{1,q}\trkla{\Om}}$.
Therefore, it suffices to show that $\tilde{u}\h$ converges strongly towards $\tilde{u}$ in $L^2\trkla{0,\Tmax;H^1\trkla{\Om}}$.
Using the triangle inequality, we derive
\begin{align}
\norm{\tilde{u}\h-\tilde{u}}_{L^2\trkla{0,\Tmax;H^1\trkla{\Om}}}\leq \norm{\tilde{u}\h-\Ritz{\tilde{u}}}_{L^2\trkla{0,\Tmax;H^1\trkla{\Om}}}+\norm{\Ritz{\tilde{u}}-\tilde{u}}_{L^2\trkla{0,\Tmax;H^1\trkla{\Om}}}\,,
\end{align}
where $\Ritzop$ is the Ritz projection operator defined in \eqref{eq:def:Ritz}.
As $\Ritz{\tilde{u}}$ converges strongly towards $\tilde{u}$ in $H^1\trkla{\Om}$ and since $\tilde{u}\h$ is $L^\infty\trkla{H^1}$-regular, it only remains to show that the first term on the right-hand side vanishes.
We define $A\h\,:\,\Uhxy\rightarrow \Uhxy\cap H^1_*\trkla{\Om}$ via $\iO\trkla{A\h\phi\h}\psi\h\dx\dy=\iO\nabla\phi\h\cdot\nabla\psi\h\dx\dy$ for all $\phi\h,\psi\h\in\Uhxy$ and compute
\begin{multline}
\norm{\nabla\tilde{u}\h-\nabla\Ritz{\tilde{u}}}_{L^2\trkla{0,T;L^2\trkla{\Om}}}^2\leq\abs{\int_0^{\Tmax}\iO\rkla{A\h\trkla{\tilde{u}\h-\Ritz{\tilde{u}}}}\cdot\rkla{\tilde{u}\h-\Ritz{\tilde{u}}}\dx\dy}\\
\leq\norm{A\h\trkla{\tilde{u}\h-\Ritz{\tilde{u}}}}_{L^2\trkla{0,\Tmax;L^2\trkla{\Om}}}\norm{\tilde{u}\h-\Ritz{\tilde{u}}}_{L^2\trkla{0,\Tmax;L^2\trkla{\Om}}}\\
\leq C\norm{\tilde{u}\h-\Ritz{\tilde{u}}}_{L^2\trkla{0,\Tmax;L^2\trkla{\Om}}}\,.
\end{multline}
Together with the strong convergence of $\tilde{u}\h$ in $L^2\trkla{0,\Tmax;L^2\trkla{\Om}}$ $\tilde{\Prob}$-almost surely, which we have from Proposition \ref{prop:convergence1}, and the strong convergence of $\Ritz{\tilde{u}}$ towards $\tilde{u}$, we complete the proof.
\end{proof}

\subsection{Convergence of the deterministic terms}
In this section, we identify the limit functions $\tilde{J}\h^x$ and $\tilde{J}\h^y$ introduced in Proposition \ref{prop:convergence1} and use the a priori estimates from Proposition \ref{prop:energyentropyestimate} to establish additional (weak) convergence properties. 
In order to identify the limit of the fluxes, we consider the discrete pressure
\begin{align*}
\Uhxy\ni\tilde{p}\h:=-\chiTth\Delta\h\tilde{u}\h+ \chiTth \Ihxy{F^\prime\trkla{\tilde{u}\h}}+\chiTth h^\varepsilon\Delta\h\Delta\h\tilde{u}\h
\end{align*}
$\tilde{\Prob}$-almost everywhere.
In addition, we introduce the sets
\begin{subequations}
\begin{align}
S_\delta:=&\,\gkla{\trkla{\tilde{\omega},t,\trkla{x,y}}\in\tilde{\Omega}\times\trkla{0,\Tmax}\times\Om~:~\tilde{u}\trkla{\omega,t,x,y}>\delta}\,,\\
S_\delta\trkla{\tilde{\omega},t}:=&\gkla{\trkla{x,y}\in\Om~:~\tilde{u}\trkla{\tilde{\omega},t,\trkla{x,y}}>\delta}\,,\\
\begin{split}
S_\delta^{\Qh}:=&\left\{\trkla{\tilde{\omega},t,\trkla{x,y}}\in\tilde{\Omega}\times\trkla{0,\Tmax}\times\Om~:~\exists \Q\in\Qh~\text{s.t.}~\trkla{x,y}\in\Q~\right.\\
&\qquad\qquad\qquad\qquad\qquad\qquad\qquad\qquad\qquad\left.\text{and}~\restr{\tilde{u}\h\trkla{\tilde{\omega},t,\cdot}}{\Q}>\delta\right\}\,,
\end{split}\\
S_\delta^{\Qh}\trkla{\tilde{\omega},t}:=&\gkla{\trkla{x,y}\in\Om~:~\exists \Q\in\Qh~\text{s.t.}~\trkla{x,y}\in\Q~\text{and}~\restr{\tilde{u}\h\trkla{\tilde{\omega},t,\cdot}}{\Q}>\delta}\,.\label{eq:SdeltaQhot}
\end{align}
\end{subequations}
On these superlevel sets, we will be able to identify the limit functions of $\tilde{p}\h$, $\tilde{J}^x\h$, and $\tilde{J}\h^y$.
In particular, the following lemma holds true.
\begin{lemma}\label{lem:convergence2}
Let $\tilde{u}\h$ and $\tilde{u}$ be the random variables identified in Proposition \ref{prop:convergence1}. 
Furthermore, let the Assumptions \ref{item:S}, \ref{item:initial}, \ref{item:potential}, \ref{item:stoch}, \ref{item:regularization}, and \ref{item:stochbasis:boundthirdderivative} hold true.
Then, there exists a subsequence, again denoted by $\tilde{u}\h$ such that for all $q<\infty$ and $\overline{q}<\tfrac{2q}{q-2}$ the following convergence properties hold true:
\begin{subequations}
\begin{align}
\tilde{u}\h&\rightarrow\tilde{u}&&\text{in~}L^q\trkla{\tilde{\Omega};C\trkla{\tekla{0,\Tmax};L^q\trkla{\Om}}}\,,\label{eq:conv:strong1}\\
\tilde{u}\h&\rightarrow\tilde{u}&&\text{in~}L^{\overline{q}}\trkla{\tilde{\Omega};L^2\trkla{0,\Tmax;W^{1,q}\trkla{\Om}}}\,,\label{eq:conv:strong2}\\
\tilde{u}\h &\weakstar \tilde{u}&&\revy{\text{in~}L^{2\p}_{\text{weak-}\trkla{*}}\trkla{\tilde{\Omega};L^\infty\trkla{0,\Tmax;H^1_\per\trkla{\Om}}}\,,}\label{eq:conv:u:weakstar}\\
\Delta\h\tilde{u}\h&\weak\Delta\tilde{u}&&\text{in~}L^2\trkla{\tilde{\Omega};L^2\trkla{0,\Tmax;L^2\trkla{\Om}}}\,,\label{eq:conv:weak:discLap}\\
\Ihy{\sqrt{\meanGinvX{\tilde{u}\h}}}&\rightarrow \tilde{u}&&\text{in~} L^{q}\trkla{\tilde{\Omega};L^\infty\trkla{{0,\Tmax};L^{q}\trkla{\Om}}}\label{eq:conv:mobx}\,,\\
\Ihx{\sqrt{\meanGinvY{\tilde{u}\h}}}&\rightarrow \tilde{u}&&\text{in~} L^{q}\trkla{\tilde{\Omega};L^\infty\trkla{{0,\Tmax};L^{q}\trkla{\Om}}}\,,\label{eq:conv:moby}
\end{align}
\end{subequations}
\revy{where we have used the notation as in Chapter 0.3 of \cite{FeireislNovotny2009} to denote the dual space of $L^{2\p/\trkla{2\p-1}}\trkla{\tilde{\Omega};L^1\trkla{0,\Tmax;\trkla{H^1_\per\trkla{\Om}}^\prime}}$.}\\
In addition, we have
\begin{subequations}
\begin{align}
\charac{S_{\delta/4}^{\Qh}}\charac{S_\delta}\parx\tilde{p}\h&\weak \charac{S_\delta}\parx\tilde{p}&&\text{in~}L^2\trkla{\tilde{\Omega};L^2\trkla{0,\Tmax;L^2\trkla{\Om}}}\,,\label{eq:conv:px}\\
\charac{S_{\delta/4}^{\Qh}}\charac{S_\delta}\pary\tilde{p}\h&\weak \charac{S_\delta}\pary\tilde{p}&&\text{in~}L^2\trkla{\tilde{\Omega};L^2\trkla{0,\Tmax;L^2\trkla{\Om}}}\,,\label{eq:conv:py}
\end{align}
\end{subequations}
with $\tilde{p}=-\Delta \tilde{u}+F^\prime\trkla{\tilde{u}}\text{~on~}S_\delta$ for all $\delta>0$ and
\begin{subequations}
\begin{align}
\tilde{J}\h^x&\weak \tilde{J}^x&&\text{in~}L^2\trkla{\tilde{\Omega};L^2\trkla{0,\Tmax;L^2\trkla{\Om}}}\,,\label{eq:conv:Jx}\\
\tilde{J}\h^y&\weak \tilde{J}^y&&\text{in~}L^2\trkla{\tilde{\Omega};L^2\trkla{0,\Tmax;L^2\trkla{\Om}}}\label{eq:conv:Jy}\,,
\end{align}
\end{subequations}
where $\tilde{J}^x$ and $\tilde{J}^y$ are the limit functions introduced in Proposition \ref{prop:convergence1}.
For every $\delta>0$, we are able to identify these limit functions on the superlevel sets $S_\delta$ as $\tilde{J}^x\!= \tilde{u}\parx\trkla{-\Delta\tilde{u}+F^\prime\trkla{\tilde{u}}}$ and $\tilde{J}^y= \tilde{u}\pary\trkla{-\Delta\tilde{u}+F^\prime\trkla{\tilde{u}}}$.
\end{lemma}
\begin{proof}
By Proposition \ref{prop:convergence1}, we have in particular $\tilde{u}\h\rightarrow\tilde{u}$. 
Choosing $\p$ sufficiently large in Proposition \ref{prop:energyentropyestimate} and combining Proposition \ref{prop:convergence1} with Vitali's convergence theorem and with the bounds on $\tilde{u}\h$ and on $\Delta\h\tilde{u}\h$ (see Proposition~\ref{prop:energyentropyestimate}),
we obtain the strong convergence of $\tilde{u}\h$ towards $\tilde{u}$ in $L^q\trkla{\tilde{\Omega};C\trkla{\tekla{0,\Tmax};L^q\trkla{\Omega}}}$, which gives \eqref{eq:conv:strong1}.
To establish \eqref{eq:conv:strong2}, we use the continuous Gagliardo--Nirenberg inequality and Lemma \ref{lem:embedding} to show that
\begin{align}
\begin{split}
\norm{\tilde{u}\h-\hat{u}}_{L^2\trkla{0,\Tmax;W^{1,q}\trkla{\Om}}}^{\hat{q}} \leq&\, C\norm{\Delta\h \tilde{u}\h}_{L^2\trkla{0,\Tmax;L^2\trkla{\Om}}}^{\hat{q}\trkla{q-2}/q}\norm{\tilde{u}\h}_{L^\infty\trkla{0,\Tmax;H^1\trkla{\Om}}}^{\hat{q}-\hat{q}\trkla{q-2}/q} \\
&+C\norm{ \tilde{u}}_{L^2\trkla{0,\Tmax;H^2\trkla{\Om}}}^{\hat{q}\trkla{q-2}/q}\norm{\tilde{u}}_{L^\infty\trkla{0,\Tmax;H^1\trkla{\Om}}}^{\hat{q}-\hat{q}\trkla{q-2}/q}\\
&+ C\norm{\tilde{u}\h}_{L^2\trkla{0,\Tmax;H^1\trkla{\Om}}}^{\hat{q}}+ C\norm{\tilde{u}}_{L^2\trkla{0,\Tmax;H^1\trkla{\Om}}}^{\hat{q}}
\end{split}
\end{align}
for all $\overline{q}<\hat{q}<\tfrac{2q}{q-2}$.
As this choice of $\hat{q}$ in particular implies that $\hat{q}\trkla{q-2}/q<2$, we may use Hölder's inequality to show that $\expectedt{\norm{\tilde{u}\h-\tilde{u}}_{L^2\trkla{0,\Tmax;W^{1,q}\trkla{\Om}}}^{\hat{q}}}$ is uniformly bounded.
As we already established the $\tilde{\Prob}$-almost sure convergence in Proposition \ref{prop:convergence1}, an application of Vitali's convergence theorem provides the result.\\
From the bounds on $\tilde{u}\h$ stated in Proposition \ref{prop:energyentropyestimate}, we obtain the weak convergence $\tilde{u}\h\weakstar \mathfrak{u}$ in \revy{$L^{2\p}_{\text{weak-}\trkla{*}}\trkla{\tilde{\Omega};L^\infty\trkla{0,\Tmax;H^1_\per\trkla{\Om}}}$} along a subsequence.
As the strong convergence of $\tilde{u}\h$ towards $\tilde{u}$ is already established in \eqref{eq:conv:strong1}, we are able to identify $\mathfrak{u}$ and $\tilde{u}$.\\
To establish the weak convergence expressed in \eqref{eq:conv:weak:discLap}, we again start with the uniform $L^2\trkla{\tilde{\Omega};L^2\trkla{0,\Tmax;L^2\trkla{\Om}}}$-bounds on $\Delta\h\tilde{u}\h$ stated in Proposition \ref{prop:convergence1}.
These bounds provide the existence of a subsequence converging weakly in $L^2\trkla{\tilde{\Omega};L^2\trkla{0,\Tmax;L^2\trkla{\Om}}}$ towards some limit function $\mathfrak{v}$.
As we also have $\Delta\h\tilde{u}\h\weak \Delta \tilde{u}$ $\tilde{\Prob}$-almost surely (cf.~Proposition \ref{prop:convergence1} and Corollary \ref{cor:tildedeltau}), combining the aforementioned uniform bounds and Vitali's convergence theorem provides
\begin{align}
\int_0^{\Tmax}\iO\Delta\h\tilde{u}\h\phi\dx\dy\dt\rightarrow\int_0^{\Tmax}\iO\Delta\tilde{u}\phi\dx\dy\dt
\end{align}
strongly in $L^r\trkla{\tilde{\Omega}}$ for $r<2$.
Therefore, we have $\mathfrak{v}=\Delta\h \tilde{u}$.\\

To establish \eqref{eq:conv:mobx}, it suffices to show $\expectedt{\norm{\Ihy{\sqrt{\meanGinvX{\tilde{u}\h}}}-\tilde{u}\h}_{L^\infty\trkla{0,\Tmax;L^q\trkla{\Om}}}^q}$ vanishes for $h\searrow0$.
Using Hölder's inequality, we obtain 
\begin{multline}
\expectedt{\sup_{t\in\tekla{0,\Tmax}}\iO\abs{\Ihy{\sqrt{\meanGinvX{\tilde{u}\h}}}-\tilde{u}\h}^q\dx\dy}\\
\leq \rkla{\expectedt{\sup_{t\in\tekla{0,\Tmax}}\iO\abs{\Ihy{\sqrt{\meanGinvX{\tilde{u}\h}}}-\tilde{u}\h}^2\dx\dy}}^{1/2}\\
\times\rkla{\expectedt{\sup_{t\in\tekla{0,\Tmax}}\iO\abs{\Ihy{\sqrt{\meanGinvX{\tilde{u}\h}}}-\tilde{u}\h}^{2q-2}\dx\dy}}^{1/2}\,.
\end{multline}
In the following, we will show that the first factor converges towards zero while the second factor remains bounded.
Since $\restr{\meanGinvX{\tilde{u}\h}}{\Kx}\trkla{y}\in\tekla{\min_{x\in\Kx}\tilde{u}\h^2\trkla{x,y},\max_{x\in\Kx}\tilde{u}\h^2\trkla{x,y}}$, we compute on each $Q\in\Qh$
\begin{multline}
\int_{\Q}\abs{\Ihy{\sqrt{\meanGinvX{\tilde{u}\h}}}-\tilde{u}\h}^2\dx\dy\leq \int_{\Q}\Ihy{\tabs{\max_{\Kx}\tilde{u}\h-\min_{\Kx}\tilde{u}\h}^2}\dx\dy\\
\leq h^2\int_{\Q}\Ihy{\tabs{\parx\tilde{u}\h}^2}\dx\dy\,,
\end{multline}
which provides the first result.
To show that the second integral remains bounded, we use that on each element $\Kx=\rkla{i\hx,\trkla{i+1}\hx}$ we have $\restr{\meanGinvX{\tilde{u}\h}}{\Kx}\trkla{y}=\tilde{u}\h\trkla{i\hx,y}\cdot\tilde{u}\h\trkla{\trkla{i+1}\hx,y}$.
Applying the inequality of arithmetic and geometric means and Jensen's inequality, we obtain on each $\Q:= \trkla{i\hx,\trkla{i+1}\hx}\times\Ky$
\begin{multline}
\int_{\Q}\abs{\Ihy{\sqrt{\meanGinvX{\tilde{u}\h}}}}^{2q-2}\dx\dy\leq \int_{\Q} \abs{\Ihy{\sqrt{\tilde{u}\h\trkla{i\hx,y}\tilde{u}\h\trkla{\trkla{i+1}\hx,y}}}}^{2q-2}\dx\dy\\
\leq\int_{Q}\abs{\tfrac12\Ihy{\tilde{u}\h\trkla{i\hx,y}+\tilde{u}\h\trkla{\trkla{i+1}\hx,y}}}^{2q-q}\dx\dy\\
\leq \int_{\Q}\tfrac12 \Ihy{\tilde{u}\h^{2q-2}\trkla{i\hx,y}+\tilde{u}\h^{2q-2}\trkla{\trkla{i+1}\hx,y}}\dx\dy=\int_{\Q}\Ihxy{\tilde{u}\h^{2q-2}}\dx\dy\,.
\end{multline}
Summing over all $\Q\in\Qh$ and applying \eqref{eq:normequivalence}, we obtain the result for $\p$ large enough.
The convergence expressed in \eqref{eq:conv:moby} can be shown with analogous computations.

To address \eqref{eq:conv:px}, we combine the bounds in Proposition \ref{prop:convergence1} with the estimate
\begin{align}\label{eq:tmp:bound:p}
\iO\charac{S_{\delta/4}^{\Qh}}\charac{S_\delta}\tabs{\parx\tilde{p}\h}^2\dx\dy
\leq\delta^{-2}C\iO\Ihy{\meanGinvX{\tilde{u}\h}\tabs{\parx\tilde{p}\h}^2}\dx\dy\,,
\end{align}
which indicates that $\charac{S_{\delta/4}^{\Qh}}\charac{S_\delta}\parx\tilde{p}\h$ is uniformly bounded in $L^2\trkla{\tilde{\Omega};L^2\trkla{0,\Tmax;L^2\trkla{\Om}}}$.
Therefore, there exists a subsequence converging towards a limit function $\eta_\delta$. 
In the following, we have to show that $\eta_\delta=\parx\tilde{p}$ on $S_\delta$.\\
We approximate the characteristic function $\charac{S_\delta}$ by a family of functions ${\tilde{\chi}_N\trkla{\omega,t}}_{N\in\mathds{N}}\,:\,\tilde{\Omega}\times\tekla{0,\Tmax}\rightarrow C^\infty_\per\trkla{\Om}$ satisfying
\begin{align}
\begin{split}
\tilde{\chi}_N\trkla{\omega,t,\trkla{x,y}}=&\,~0\qquad\text{for~} \trkla{x,y}\in S_\delta\trkla{\omega,t}^c\,,\\
\tilde{\chi}_N\trkla{\omega,t,\trkla{x,y}}=&\left\{\begin{matrix}
1&\text{if~}\dist\trkla{\trkla{x,y},\partial S_\delta\trkla{\omega,t}}\geq \tfrac1N\,,\\
0&\text{if~}\dist\trkla{\trkla{x,y},\partial S_\delta\trkla{\omega,t}}\leq\tfrac1{N+1}\,.
\end{matrix}\right.
\end{split}
\end{align}
To identify $\eta_\delta$ on $S_\delta$ with $\parx\tilde{p}$, we will show that
\begin{align}
\expectedt{\int_0^{\Tmax}\iO\eta_\delta\tilde{\phi}\tilde{\chi}_N\dx\dy\dt}=-\expectedt{\int_0^{\Tmax}\iO\tilde{p}\parx\trkla{\tilde{\phi}\tilde{\chi}_N}\dx\dy\dt}\,
\end{align}
for sufficiently regular test functions $\tilde{\phi}$.
In the following, we will show that this equation is valid even in the slightly more general case when $\tilde{\phi}\tilde{\chi}$ is replaced by the test function $\tilde{\zeta}\in L^\infty\trkla{\tilde{\Omega};C^\infty\trkla{\tekla{0,\Tmax};C^\infty_\per\trkla{\overline{\Om}}}}$ with $\operatorname{supp} \tilde{\zeta}\trkla{\omega,t}\subset S_\delta\trkla{\omega,t}$.

As $\tilde{u}\h$ converges strongly in $W^{1,q}_\per\trkla{\Om}$ ($q<\infty$) $\tilde{\Prob}$-almost surely for almost all $t\in\rkla{0,\Tmax}$ and hence almost everywhere in $\tilde{\Omega}\times\trkla{0,\Tmax}$, we have $\tilde{u}\h\rightarrow \tilde{u}$ in $C^{\gamma}_\per\trkla{\overline{\Om}}$ with $\gamma<1$ almost everywhere in $\tilde{\Omega}\times\rkla{0,\Tmax}$.
This allows us to apply an appropriate version of  Egorov's theorem (cf. Theorem 42 in \cite{Dinculeanu2000}) and to deduce for all $\iota>0$ the existence of a subset $\mathfrak{E}_\iota\subset\tilde{\Omega}\times\rkla{0,\Tmax}$ with measure smaller than $\iota$ such that $\tilde{u}\h$ converges uniformly in $\trkla{\tilde{\Omega}\times\trkla{0,\Tmax}}\setminus\mathfrak{E}_\iota=:\mathfrak{E}_\iota^c$.
Therefore, we compute
\begin{multline}
\expectedt{\int_0^{\Tmax}\iO\charac{S_{\delta/4}^{\Qh}}\charac{S_\delta}\parx \tilde{p}\h \tilde{\zeta}\dx\dy\dt}=\int_{\tilde{\Omega}}\int_0^{\Tmax}\iO\charac{S_{\delta/4}^{\Qh}}\parx \tilde{p}\h \tilde{\zeta}\dx\dy\dt\, \mathrm{d}\probt{\tilde{\omega}}\\
=\int_{\mathfrak{E}_\iota^c} \iO \charac{S_{\delta/4}^{\Qh}}\parx \tilde{p}\h \tilde{\zeta}\dx\dy\dt\, \mathrm{d}\probt{\tilde{\omega}} +\int_{\mathfrak{E}_\iota} \iO \charac{S_{\delta/4}^{\Qh}}\parx \tilde{p}\h \tilde{\zeta}\dx\dy\dt\, \mathrm{d}\probt{\tilde{\omega}}=: A+B\,.
\end{multline} 
Applying Hölder's inequality, we immediately obtain $\tabs{B}\leq C\iota^{1/2}$.
As $\tilde{u}\h$ converges uniformly towards $\tilde{u}$ in $\mathfrak{E}_\iota^c$, we have $\operatorname{supp} \tilde{\zeta}\trkla{\omega,t}\subset S_\delta\trkla{\omega,t}\subset S_{\delta/4}^{\Qh}\trkla{\omega,t}$ for $h$ sufficiently small.
Using the definition of $\tilde{p}\h$, we obtain
\begin{align}
\begin{split}
A=&\,-\int_{\mathfrak{E}_\iota^c} \iO \tilde{p}\h \parx\tilde{\zeta}\dx\dy\dt\, \mathrm{d}\probt{\tilde{\omega}}\\
=&\, \int_{\mathfrak{E}_\iota^c} \iO\chiTth \Delta\h\tilde{u}\h \parx\tilde{\zeta}\dx\dy\dt\, \mathrm{d}\probt{\tilde{\omega}} -\int_{\mathfrak{E}_\iota^c} \iO \chiTth \Ihxy{F^\prime\trkla{\tilde{u}\h}} \parx\tilde{\zeta}\dx\dy\dt\, \mathrm{d}\probt{\tilde{\omega}}\\
&-h^\varepsilon\int_{\mathfrak{E}_\iota^c} \iO \chiTth \Delta\h\Delta\h\tilde{u}\h \parx\tilde{\zeta}\dx\dy\dt\, \mathrm{d}\probt{\tilde{\omega}}\\
=:&\,A_1+A_2+A_3\,.
\end{split}
\end{align}
The convergence of $A_1$ towards $ \int_{\mathfrak{E}_\iota^c} \iO \Delta\tilde{u} \parx\tilde{\zeta}\dx\dy\dt\, \mathrm{d}\probt{\tilde{\omega}}$ is a direct consequence of the weak convergence \eqref{eq:conv:weak:discLap}.
To obtain the convergence of $A_2$, we use
\begin{align*}
A_2\!=\!-\!\int_{\mathfrak{E}_\iota^c}\! \iO\!\chiTth F^\prime\trkla{\tilde{u}\h} \parx\tilde{\zeta}\dx\dy\dt\, \mathrm{d}\probt{\tilde{\omega}} \!+\!\!\int_{\mathfrak{E}_\iota^c}\! \iO\!\chiTth \trkla{\ids-\Ihxyop}\tgkla{F^\prime\trkla{\tilde{u}\h}} \parx\tilde{\zeta}\dx\dy\dt\, \mathrm{d}\probt{\tilde{\omega}}\,.
\end{align*}
As $\operatorname{supp} \tilde{\zeta}\trkla{\omega,t}\subset S_{\delta/4}^{\Qh}\trkla{\omega,t}$, the second term vanishes for $h\searrow0$ due to Lemma \ref{lem:errorF} and the first term converges towards $-\int_{\mathfrak{E}_\iota^c} \iO F^\prime\trkla{\tilde{u}} \parx\tilde{\zeta}\dx\dy\dt\, \mathrm{d}\probt{\tilde{\omega}}$ due to Vitali's convergence theorem.
The treatment of $A_3$ is more delicate, as the bounds on $h^\varepsilon\Delta\h\Delta\h \tilde{u}\h$ are not obvious. 
We begin by splitting $A_3$ into
\begin{align}
\begin{split}
A_3=&\,-h^\varepsilon\int_{\mathfrak{E}_\iota^c} \iO\chiTth\Ihxy{ \Delta\h\Delta\h\tilde{u}\h \Ihxy{\parx\tilde{\zeta}}}\dx\dy\dt\, \mathrm{d}\probt{\tilde{\omega}} \\
&\,-h^\varepsilon\int_{\mathfrak{E}_\iota^c} \iO\chiTth\trkla{\ids-\Ihxyop}\gkla{ \Delta\h\Delta\h\tilde{u}\h \Ihxy{\parx\tilde{\zeta}}}\dx\dy\dt\, \mathrm{d}\probt{\tilde{\omega}}\\
&\,-h^\varepsilon\int_{\mathfrak{E}_\iota^c} \iO \chiTth\Delta\h\Delta\h\tilde{u}\h \trkla{\ids-\Ihxyop}\tgkla{\parx\tilde{\zeta}}\dx\dy\dt\, \mathrm{d}\probt{\tilde{\omega}}=:A_{3_a}+A_{3_b}+A_{3_c}\,.
\end{split}
\end{align}
Recalling \eqref{eq:defDiscLapl} and applying Hölder's inequality shows that $A_{3_a}$ vanishes, since
\begin{align}
\begin{split}
\tabs{A_{3_a}}\leq&\, h^{\varepsilon/2}\norm{h^{\varepsilon/2}\parx\Delta\h\tilde{u}\h}_{L^2\trkla{\tilde{\Omega};L^2\trkla{0,\Tmax;L^2\trkla{\Om}}}}\norm{\parx\Ihxy{\parx\tilde{\zeta}}}_{L^2\trkla{\tilde{\Omega};L^2\trkla{0,\Tmax;L^2\trkla{\Om}}}}\\
&+\,h^{\varepsilon/2}\norm{h^{\varepsilon/2}\pary\Delta\h\tilde{u}\h}_{L^2\trkla{\tilde{\Omega};L^2\trkla{0,\Tmax;L^2\trkla{\Om}}}}\norm{\pary\Ihxy{\parx\tilde{\zeta}}}_{L^2\trkla{\tilde{\Omega};L^2\trkla{0,\Tmax;L^2\trkla{\Om}}}}\\
\leq&\, h^{\varepsilon/2} C\,.
\end{split}
\end{align}
Due to \eqref{eq:dqLap} we have $h\norm{h^{\varepsilon/2}\Delta\h\Delta\h\tilde{u}\h}_{L^2\trkla{\tilde{\Omega};L^2\trkla{0,\Tmax;L^2\trkla{\Om}}}}\leq C$.
Therefore, standard estimates for $\Ihxyop$, which can be found e.g. in Theorem 4.4.20 in \cite{BrennerScott}, provide
\begin{align}
\begin{split}
\tabs{A_{3_b}}\leq&\, C h^{\varepsilon/2} h\norm{h^{\varepsilon/2}\Delta\h\Delta\h\tilde{u}\h}_{L^2\trkla{\tilde{\Omega};L^2\trkla{0,\Tmax;L^2\trkla{\Om}}}}\norm{\nabla \Ihxy{\parx\tilde{\zeta}}}_{L^2\trkla{\tilde{\Omega};L^2\trkla{0,\Tmax;L^2\trkla{\Om}}}}\\
\leq&\, C h^{\varepsilon/2}\,.
\end{split}
\end{align}
Similar considerations based on Lemma \ref{lem:Ih:error} provide
\begin{align}
\begin{split}
\tabs{A_{3_c}}\leq &\,h^{\varepsilon/2}\norm{h^{\varepsilon/2}\Delta\h\Delta\h\tilde{u}\h}_{L^2\trkla{\tilde{\Omega};L^2\trkla{0,\Tmax;L^2\trkla{\Om}}}}\norm{\trkla{\ids-\Ihxyop}\tgkla{\parx\tilde{\zeta}}}_{L^2\trkla{\tilde{\Omega};L^2\trkla{0,\Tmax;L^2\trkla{\Om}}}}\\
\leq&\, C h^{1+\varepsilon/2} \norm{\parx\tilde{\zeta}}_{L^2\trkla{\tilde{\Omega};L^2\trkla{0,\Tmax;H^2\trkla{\Om}}}}\,.
\end{split}
\end{align}
Collecting the results above, we have
\begin{align}
\begin{split}
A\rightarrow&\, \int_{\mathfrak{E}_\iota^c} \iO \Delta\tilde{u} \parx\tilde{\zeta}\dx\dy\dt\, \mathrm{d}\probt{\tilde{\omega}} -\int_{\mathfrak{E}_\iota^c} \iO F^\prime\trkla{\tilde{u}} \parx\tilde{\zeta}\dx\dy\dt\, \mathrm{d}\probt{\tilde{\omega}}\\
=&\,-\int_{\tilde{\Omega}}\int_0^{\Tmax} \iO \tilde{p}\parx\tilde{\zeta}\dx\dy\dt\, \mathrm{d}\probt{\tilde{\omega}} -\int_{\mathfrak{E}_\iota} \iO \Delta\tilde{u} \parx\tilde{\zeta}\dx\dy\dt\, \mathrm{d}\probt{\tilde{\omega}}\\
& +\int_{\mathfrak{E}_\iota} \iO F^\prime\trkla{\tilde{u}} \parx\tilde{\zeta}\dx\dy\dt\, \mathrm{d}\probt{\tilde{\omega}}\,.
\end{split}
\end{align}
As $\tilde{u}>\delta$ in $\operatorname{supp}\parx\tilde{\zeta}$, the last two integrals can be bounded by $C\iota^{1/2}$.
Therefore, we may identify $\eta_\delta$ with $\charac{S_\delta}\parx\tilde{p}$, which provides \eqref{eq:conv:px}.
The convergence expressed in \eqref{eq:conv:py} can be proven by similar computations.\\
The weak convergence expressed in \eqref{eq:conv:Jx} and \eqref{eq:conv:Jy} can be established analogously to \eqref{eq:conv:weak:discLap}.
Therefore, it remains to show that the fluxes $\tilde{J}^x$ and $\tilde{J}^y$ coincide with $\tilde{u}\parx \tilde{p}$ and $\tilde{u}\pary \tilde{p}$, respectively.
Reusing the ideas of the proof of \eqref{eq:conv:px}, we choose $\tilde{\zeta}\in L^\infty\trkla{\tilde{\Omega};C^\infty\trkla{\tekla{0,\Tmax};C_\per^\infty\trkla{\overline{\Om}}}}$ with $\operatorname{supp}\tilde{\zeta}\trkla{\omega,t}\subset S_\delta\trkla{\omega,t}$ and compute
\begin{align}
\begin{split} 
&\expectedt{\int_0^{\Tmax}\iO \tilde{J}\h^y\tilde{\zeta}\dx\dy\dt}=\expectedt{\int_0^{\Tmax}\iO \Ihy{\sqrt{\meanGinvX{\tilde{u}\h}}\parx \tilde{p}\h}\tilde{\zeta}\dx\dy\dt}\\
&\qquad=\expectedt{\int_0^{\Tmax}\iO \charac{S_{\delta/4}^{\Qh}}\charac{S_\delta} \Ihy{\sqrt{\meanGinvX{\tilde{u}\h}}\parx \tilde{p}\h}\tilde{\zeta}\dx\dy\dt}\\
&\qquad\qquad+\expectedt{\int_0^{\Tmax}\iO \trkla{1-\charac{S_{\delta/4}^{\Qh}}}\charac{S_\delta} \Ihy{\sqrt{\meanGinvX{\tilde{u}\h}}\parx \tilde{p}\h}\tilde{\zeta}\dx\dy\dt}\\
&\qquad=\expectedt{\int_0^{\Tmax}\iO \charac{S_{\delta/4}^{\Qh}}\charac{S_\delta} \Ihy{\sqrt{\meanGinvX{\tilde{u}\h}}}\parx \tilde{p}\h\tilde{\zeta}\dx\dy\dt}\\
&\qquad\qquad-\expectedt{\int_0^{\Tmax}\iO \charac{S_{\delta/4}^{\Qh}}\charac{S_\delta} \trkla{\ids-\Ihyop}\gkla{\Ihy{\sqrt{\meanGinvX{\tilde{u}\h}}}\parx \tilde{p}\h}\tilde{\zeta}\dx\dy\dt}\\
&\qquad\qquad+\expectedt{\int_0^{\Tmax}\iO \trkla{1-\charac{S_{\delta/4}^{\Qh}}}\charac{S_\delta} \Ihy{\sqrt{\meanGinvX{\tilde{u}\h}}\parx \tilde{p}\h}\tilde{\zeta}\dx\dy\dt}\\
&\qquad=:B_1+B_2+B_3\,.
\end{split}
\end{align}
The convergence 
\begin{align}
B_1\rightarrow\expectedt{\int_0^{\Tmax}\iO\charac{S_\delta}\tilde{u}\parx\tilde{p}\tilde{\zeta}\dx\dy\dt}=\expectedt{\int_0^{\Tmax}\iO\tilde{u}\parx\tilde{p}\tilde{\zeta}\dx\dy\dt}
\end{align}
follows directly from \eqref{eq:conv:mobx} and \eqref{eq:conv:px}.
Combining \eqref{eq:Ihy:lp:pointwise} (cf.~Lemma \ref{lem:Ih:error}) with standard inverse estimates (cf.~Theorem 4.5.11 in \cite{BrennerScott}) and \eqref{eq:tmp:bound:p} provides the estimate
\begin{align}
\begin{split}
B_2\leq&\,Ch\norm{\pary\Ihy{\sqrt{\meanGinvX{\tilde{u}\h}}}}_{L^2\trkla{\tilde{\Omega};L^2\trkla{0,\Tmax;L^2\trkla{\Om}}}}\norm{\parx\tilde{p}\h}_{L^2\trkla{S_\delta}}\\
\leq&\,C\delta^{-1}\rkla{\norm{\Ihy{\sqrt{\meanGinvX{\tilde{u}\h}}}-\tilde{u\h}}_{L^2\trkla{\tilde{\Omega};L^2\trkla{0,\Tmax;L^2\trkla{\Om}}}}\!\!\!+\!h\norm{\pary\tilde{u}\h}_{L^2\trkla{\tilde{\Omega};L^2\trkla{0,\Tmax;L^2\trkla{\Om}}}}}.
\end{split}
\end{align}
Therefore, $B_2$ vanishes for $h\searrow0$ due to \eqref{eq:conv:mobx}, and \eqref{eq:conv:strong1}.
Applying Hölder's inequality and the bounds established in Proposition \ref{prop:convergence1}, we obtain
\begin{align}
B_3\leq C\norm{\rkla{\charac{S_\delta}-\charac{S_{\delta/4}^{\Qh}}}\tilde{\zeta}}_{L^2\trkla{\tilde{\Omega};L^2\trkla{0,\Tmax;L^2\trkla{\Om}}}}\,.
\end{align}
Similar to the arguments used in the proof of \eqref{eq:conv:px}, we may use Egorov's theorem to obtain the existence of a subset $\mathfrak{E}_\iota\subset\tilde{\Omega}\times\rkla{0,\Tmax}$ with measure smaller than $\iota$ such that $\tilde{u}\h$ converges uniformly in $C^{0,\gamma}_\per\trkla{\overline{\Om}}$ in $\trkla{\tilde{\Omega}\times\trkla{0,\Tmax}}\setminus\mathfrak{E}_\iota=:\mathfrak{E}_\iota^c$.
As we have $\charac{S_\delta}\subset\charac{S_{\delta/4}^{\Qh}}$ in $\mathfrak{E}_\iota^c$ for $h$ small enough and $\operatorname{supp}\tilde{\zeta}\subset S_\delta$, we obtain $B_3\leq C\iota^{1/2}$ for all $\iota>0$.
As we also have $\tilde{J}\h^x\weak\tilde{J}^x$ in $L^2\trkla{\tilde{\Omega};L^2\trkla{0,\Tmax;L^2\trkla{\Om}}}$, we obtain $\tilde{J}^x=\tilde{u}\parx\trkla{-\Delta\tilde{u}+F^\prime\trkla{\tilde{u}}}$ on $S_\delta$.
The identification of $\tilde{J}^y$ on $S_\delta$ follows by similar arguments.
\end{proof}

\subsection{Convergence of the stochastic integral}
We consider for arbitrary but fixed $v\in C^\infty_\per\trkla{\overline{\Om}}$ the operator $\marth\,:\,\Omega\times\tekla{0,\Tmax}\rightarrow\R$ defined by
\begin{align}
\begin{split}
\marth\trkla{t}:=&\,\iO\Ihxy{\trkla{u\h\trkla{t}-u\h\trkla{0}} v}\dx\dy \\
&+\itTh\iO\Ihy{\sqrt{\meanGinvX{u\h}}J\h^x\parx\Ihxy{v}}\dx\dy\dtau\\
&+\itTh\iO\Ihx{\sqrt{\meanGinvY{u\h}}J\h^y\pary\Ihxy{v}}\dx\dy\dtau\\
=&\,\sum_{k\in\Zx}\sum_{l\in\Zy}\itTh\iO\Ihy{\Ihxloc{\parx\trkla{u\h\gt{kl}}\Ihxy{v}}}\dx\dy\dbx{kl}\\
&+\sum_{k\in\Zx}\sum_{l\in\Zy}\itTh\iO\Ihx{\Ihyloc{\pary\trkla{u\h\gt{kl}}\Ihxy{v}}}\dx\dy\dby{kl}\,.
\end{split}
\end{align}
\begin{remark}
In contrast to \cite{FischerGrun2018}, we define $\marth$ using functions $v$ in $C_\per^\infty\trkla{\overline{\Om}}$ instead of $H^2_\per\trkla{\Om}$. 
This seems to be necessary, as the limit process in Lemma \ref{lem:mart1} requires convergence and stability properties of the projection of $v$ that are not provided by the $L^2$-projection.
Hence, the requirements on the regularity of $v$ are also higher.
\end{remark}
By the optional stopping theorem, $\marth$ is a real valued martingale, i.e. we have
\begin{align}
\expected{\trkla{\marth\trkla{t}-\marth\trkla{s}}\Psi\trkla{r_s u\h,\, r_s \bs{W}}}=0\,
\end{align}
for all $0\leq s\leq t\leq \Tmax$ and for all $\tekla{0,1}$-valued functions $\Psi$ defined on $\Xu$.
Here, $r_s$ denotes the restriction of a function on $\tekla{0,\Tmax}$ onto $\tekla{0,s}$.
\begin{lemma}\label{lem:quadraticvariation}
For the quadratic variation of $\marth$, we have
\begin{align}
\begin{split}
\qvar{\marth}_t=&\,\sum_{k\in\Zx}\sum_{l\in\Zy}\itTh\lx{kl}^2\rkla{\iO\Ihy{\Ihxloc{\parx\trkla{u\h\gt{kl}}\Ihxy{v}}}\dx\dy}^2\ds\\
&+\sum_{k\in\Zx}\sum_{l\in\Zy}\itTh\ly{kl}^2\rkla{\iO\Ihx{\Ihyloc{\pary\trkla{u\h\gt{kl}}\Ihxy{v}}}\dx\dy}^2\ds\\
\leq&\,C\norm{v}_{H^2\trkla{\Om}}^2\itTh\norm{u\h\trkla{s}}_{H^1\trkla{\Om}}^2\ds\,.
\end{split}
\end{align}
\end{lemma}
\begin{proof}
We consider the mapping $R\trkla{u\h,v}\,:\,\Omega\times\tekla{0,\Tmax}\times \trkla{L^2\trkla{\Om}}^2 \rightarrow\R$ defined by
\begin{align}
\begin{split}
\trkla{\omega,t,\trkla{z_x,z_y}}\mapsto&\, \chiTh\trkla{t,\omega}\iO\Ihy{\Ihxloc{\parx \trkla{u\h\sum_{k\in\Zx}\sum_{l\in\Zy}\skla{\g{kl},z_x}_{L^2}\gt{kl}}\Ihxy{v}}}\dx\dy\\
&+\chiTh\trkla{t,\omega}\iO\Ihx{\Ihyloc{\pary \trkla{u\h\sum_{k\in\Zx}\sum_{l\in\Zy}\skla{\g{kl},z_y}_{L^2}\gt{kl}}\Ihxy{v}}}\dx\dy\,.
\end{split}
\end{align}
We obtain for the Hilbert-Schmidt norm
\begin{multline}
\norm{R\trkla{u\h,v}\trkla{t,\omega}}_{L_2\trkla{Q^{1/2}\trkla{L^2\trkla{\Om}}^2;\R}}^2\\
=\chiTh\trkla{t,\omega}\sum_{k\in\Zx}\sum_{l\in\Zy}\lx{kl}^2\rkla{\iO\Ihy{\Ihxloc{\parx\trkla{u\h\gt{kl}}\Ihxy{v}}}\dx\dy}^2\\
+\chiTh\trkla{t,\omega}\sum_{k\in\Zx}\sum_{l\in\Zy}\lx{kl}^2\rkla{\iO\Ihx{\Ihyloc{\pary\trkla{u\h\gt{kl}}\Ihxy{v}}}\dx\dy}^2\\
=: \chiTh\trkla{t,\omega}\sum_{k\in\Zx}\sum_{l\in\Zy}\lx{kl}^2 A^2+ \chiTh\trkla{t,\omega}\sum_{k\in\Zx}\sum_{l\in\Zy}\ly{kl}^2B^2\,.
\end{multline}
To estimate the first term, we use \eqref{eq:Ihlocpar} and \eqref{eq:Ih=Ihloc} and compute
\begin{align}
\begin{split}
A^2\leq&\,2\rkla{\iO\Ihy{\parx u\h\Ihx{\gt{kl}\Ihxy{v}}}\dx\dy}^2+2\rkla{\iO\Ihy{\parx\gt{kl}\Ihx{u\h\Ihxy{v}}}\dx\dy}^2\\
\leq& \,C\norm{\gt{kl}}_{L^\infty\trkla{\Om}}^2\norm{\parx u\h}_{L^2\trkla{\Om}}^2\norm{\Ihxy{v}}_{L^2\trkla{\Om}}^2+\norm{\parx\gt{kl}}_{L^\infty\trkla{\Om}}^2\norm{ u\h}_{L^2\trkla{\Om}}^2\norm{\Ihxy{v}}_{L^2\trkla{\Om}}^2\,.
\end{split}
\end{align}
Using \ref{item:stochbasis:bounds}, the standard error estimates for $\Ihxyop$, and a similar estimate for $B^2$, we conclude
\begin{align}
\norm{R\trkla{u\h,v}\trkla{t,\omega}}_{L_2\trkla{Q^{1/2}\trkla{L^2\trkla{\Om}}^2;\R}}^2\leq C \chiTh\trkla{t,\omega}\norm{u\h}_{H^1\trkla{\Om}}^2\norm{v}_{H^2\trkla{\Om}}^2\,.
\end{align}
Applying Lemma 2.4.3 in \cite{PrevotRoeckner} concludes the proof.
\end{proof}
In the next lemma, we will study cross variation of $\marth$ with the processes
\begin{align}
\balpha{kl}\trkla{t}=\iO\int_0^t\trkla{\lalpha{kl}}^{-1}\g{kl}\kbalpha\cdot\operatorname{d}\bs{W}\dx\dy
\end{align}
for $k,l\in\Z$ and $\alpha\in\tgkla{x,y}$.
\begin{lemma}
For $k,l\in\Z$ and $\alpha\in\tgkla{x,y}$, the cross variation $\crossvar{\marth}{\balpha{kl}}_t$ is given by 
\begin{subequations}
\begin{align}
\crossvar{\marth}{\bx{kl}}_t&=\left\{\begin{array}{cl}
\lx{kl}\itTh\iO\Ihy{\Ihxloc{\parx\trkla{u\h\gt{kl}}\Ihxy{v}}}\dx\dy\ds & \text{if~}k\in\Zx,\,l\in\Zy\,,\\
0&\text{else}\,,
\end{array}\right.\label{eq:crossvar:x}\\
\crossvar{\marth}{\by{kl}}_t&=\left\{\begin{array}{cl}
\ly{kl}\itTh\iO\Ihx{\Ihyloc{\pary\trkla{u\h\gt{kl}}\Ihxy{v}}}\dx\dy\ds & \text{if~}k\in\Zx,\,l\in\Zy\,,\\
0&\text{else}\,.
\end{array}\right.\label{eq:crossvar:y}
\end{align}
\end{subequations}
\end{lemma}
\begin{proof}
The proof follows the lines of Lemma 5.12 in \cite{FischerGrun2018}.
We will only prove \eqref{eq:crossvar:x}, as \eqref{eq:crossvar:y} follows by similar computations.\\
To compute the cross variation with $\bx{kl}$, we consider for given $k$, $l$ the mappings $S^x_\pm\,:\,\Omega\times\tekla{0,\Tmax}\times Q_x^{1/2}L^2\trkla{\Omega}\rightarrow \R$ which are defined as
\begin{align}
z\mapsto\chiTh\sum_{\overline{k}\in \Zx}\sum_{\overline{l}\in\Zy}\rkla{\iO\Ihy{\Ihxloc{\parx\trkla{u\h \skla{\g{\overline{k}\overline{l}},z}_{L^2}}\Ihxy{v}}}\dx\dy \pm\trkla{\lx{kl}}^{-1}\iO\g{kl}z\dx\dy}\,.
\end{align}
Computing the Hilbert-Schmidt norm of $S_\pm^x$, we obtain
\begin{align}
\begin{split}
\norm{S_\pm^x\trkla{u\h,v}}_{L_2\trkla{Q_x^{1/2}L^2\trkla{\Om};\R}}=\chiTh\sum_{\overline{k}\in\Zx}\sum_{\overline{l}\in\Zy}\lx{\overline{k}\overline{l}}\rkla{\iO\Ihy{\Ihxloc{\parx\trkla{u\h\gt{\overline{kl}}}\Ihxy{v}}}\dx\dy}^2\\
\pm\chiTh2\sum_{\overline{k}\in\Zx}\sum_{\overline{l}\in\Zy}\lx{\overline{kl}}\iO\Ihy{\Ihxloc{\parx\trkla{u\h\gt{\overline{kl}}}\Ihxy{v}}}\dx\dy\tfrac{\lx{\overline{kl}}}{\lx{kl}}\iO\g{kl}\g{\overline{kl}}\dx\dy\\
+\chiTh\sum_{\overline{k},\overline{l}\in\Z}\rkla{\tfrac{\lx{\overline{kl}}}{\lx{kl}}\iO\g{kl}\g{\overline{kl}}\dx\dy}^2\,.
\end{split}
\end{align}
Using $\crossvar{\marth}{\bx{kl}}_t=\tfrac14\trkla{\qvar{S^x_+\trkla{u\h,v}}_t-\qvar{S_-^x\trkla{u\h,v}}_t}$ and recalling the identity \linebreak$\iO\g{kl}\g{\overline{kl}}\dx\dy=\delta_{k\overline{k}}\delta_{l\overline{l}}$, we deduce \eqref{eq:crossvar:x}.
\end{proof}
In addition to $\marth$, the processes
\begin{multline}
\marth^2-\sum_{k\in\Zx}\sum_{l\in\Zy}\int_0^{\trkla{\cdot}\wedge T\h}\lx{kl}^2\rkla{\iO\Ihy{\Ihxloc{\parx\trkla{u\h\gt{kl}}\Ihxy{v}}}\dx\dy}^2\ds\\
-\sum_{k\in\Zx}\sum_{l\in\Zy}\int_0^{\trkla{\cdot}\wedge T\h}\ly{kl}^2\rkla{\iO\Ihx{\Ihyloc{\pary\trkla{u\h\gt{kl}}\Ihxy{v}}}\dx\dy}^2\ds
\end{multline}
and
\begin{align}
\begin{array}{cc}
\marth\bx{kl}-\lx{kl}\int_0^{\trkla{\cdot}\wedge T\h}\iO\Ihy{\Ihxloc{\parx\trkla{u\h\gt{kl}}\Ihxy{v}}}\dx\dy\ds&\text{if~}k\in\Zx,\,l\in\Zy\,,\\
\marth\bx{kl}&\text{else}\,,
\end{array}\\
\begin{array}{cc}
\marth\by{kl}-\ly{kl}\int_0^{\trkla{\cdot}\wedge T\h}\iO\Ihx{\Ihyloc{\pary\trkla{u\h\gt{kl}}\Ihxy{v}}}\dx\dy\ds&\text{if~}k\in\Zx,\,l\in\Zy\,,\\
\marth\by{kl}&\text{else}\,
\end{array}
\end{align}
are also a martingales.\\
By equality of laws, we deduce that the following processes are also $\trkla{\tilde{\mathcal{F}}_{h,t}}$-martingales:
\begin{subequations}
\begin{align}
\begin{split}\label{eq:mt1}
\martht\trkla{t}:=&\,\iO\Ihxy{\trkla{\tilde{u}\h\trkla{t}-\tilde{u}\h\trkla{0}}\Ihxy{v}}\dx\dy\\
&+\itTht\iO\Ihy{\sqrt{\meanGinvX{\tilde{u}\h}}\tilde{J}\h^x\parx\Ihxy{v}}\dx\dy\dtau\\
&+\itTht\iO\Ihx{\sqrt{\meanGinvY{\tilde{u}\h}}\tilde{J}\h^y\pary\Ihxy{v}}\dx\dy\dtau\,,
\end{split}\\
\begin{split}\label{eq:mt2}
\martht^2\trkla{t}-&\,\sum_{k\in\Zx}\sum_{l\in\Zy}\lx{kl}^2\itTht\rkla{\iO\Ihy{\Ihxloc{\parx\trkla{\tilde{u}\h\gt{kl}}\Ihxy{v}}}\dx\dy}^{2}\ds\\
&\,-\sum_{k\in\Zx}\sum_{l\in\Zy}\ly{kl}^2\itTht\rkla{\iO\Ihx{\Ihyloc{\pary\trkla{\tilde{u}\h\gt{kl}}\Ihxy{v}}}\dx\dy}^{2}\ds\,,
\end{split}
\end{align}
\begin{align}
\begin{array}{cc}
\martht\bthx{kl}-\lx{kl}\int_0^{\trkla{\cdot}\wedge \tilde{T}\h}\iO\Ihy{\Ihxloc{\parx\trkla{\tilde{u}\h\gt{kl}}\Ihxy{v}}}\dx\dy\ds&\text{if~}k\in\Zx,\,l\in\Zy\,,\\
\martht\bthx{kl}&\text{else}\,,
\end{array}\\
\begin{array}{cc}
\martht\bthy{kl}-\ly{kl}\int_0^{\trkla{\cdot}\wedge \tilde{T}\h}\iO\Ihx{\Ihyloc{\pary\trkla{\tilde{u}\h\gt{kl}}\Ihxy{v}}}\dx\dy\ds&\text{if~}k\in\Zx,\,l\in\Zy\,,\\
\martht\bthy{kl}&\text{else}\,.
\end{array}
\end{align}
\end{subequations}
Here, we used
\begin{align}
\bthalpha{kl}\trkla{t}=\iO\int_0^t\trkla{\lalpha{kl}}^{-1}\g{kl}\kbalpha\cdot\operatorname{d}\tilde{\bs{W}}\h\dx\dy
\end{align}
for $k,l\in\Z$ and $\alpha\in\tgkla{x,y}$.
Furthermore, the quadratic variation of $\martht$ and the cross variations with $\bthalpha{kl}$ are given as
\begin{align}
\begin{split}\label{eq:qvarmth}
\qvar{\martht}_t&=\,\sum_{k\in\Zx}\sum_{l\in\Zy}\itTht\lx{kl}^2\rkla{\iO\Ihy{\Ihxloc{\parx\trkla{\tilde{u}\h\gt{kl}}\Ihxy{v}}}\dx\dy}^2\ds\\
&\quad+\sum_{k\in\Zx}\sum_{l\in\Zy}\itTht\ly{kl}^2\rkla{\iO\Ihx{\Ihyloc{\pary\trkla{\tilde{u}\h\gt{kl}}\Ihxy{v}}}\dx\dy}^2\ds
\end{split}\\
\crossvar{\martht}{\bthx{kl}}_t\!&=\left\{\!\!\begin{array}{cl}
\lx{kl}\itTht\!\!\iO\Ihy{\Ihxloc{\parx\trkla{\tilde{u}\h\gt{kl}}\Ihxy{v}}}\dx\dy\ds & \text{if~}k\!\in\!\Zx,\,l\!\in\!\Zy\,,\\
0&\text{else}\,,
\end{array}\right.\\
\crossvar{\martht}{\bthy{kl}}_t\!&=\left\{\!\!\begin{array}{cl}
\ly{kl}\itTht\!\!\iO\Ihx{\Ihyloc{\pary\trkla{\tilde{u}\h\gt{kl}}\Ihxy{v}}}\dx\dy\ds & \text{if~}k\!\in\!\Zx,\,l\!\in\!\Zy\,,\\
0&\text{else}\,.
\end{array}\right.
\end{align}
\begin{lemma}\label{lem:mart1}
Let the Assumptions \ref{item:S}, \ref{item:initial}, \ref{item:potential}, \ref{item:stoch}, \ref{item:regularization}, and \ref{item:stochbasis:boundthirdderivative} hold true.
Then, for all $\tekla{0,1}$-valued continuous functions $\Psi$ defined on $\Xu\times\XW$, we have
\begin{multline}
\tilde{\mathds{E}}\left[\left(\iO\trkla{\tilde{u}\trkla{t}-\tilde{u}\trkla{s}}v\dx\dy+\int_s^t\iO\tilde{u} J^x\parx v\dx\dy\dtau \right.\right.\\
+\left.\left.\int_s^t\iO\tilde{u} J^y\pary v\dx\dy\dtau\right)\Psi\trkla{r_s\tilde{u},r_s\tilde{\bs{W}}}\right]=0
\end{multline}
for all $0\leq s\leq t\leq\Tmax$.
\end{lemma}
\begin{proof}
To establish the claim, we pass to the limit in the identity
\begin{align}
\expectedt{\rkla{\martht\trkla{t}-\martht\trkla{s}}\Psi\trkla{r_s\tilde{u}\h,r_s\tilde{\bs{W}}\h}}=0\,.
\end{align}
Due to Lemma \ref{lem:stoppingtime}, we have $\chiTth=1$ on $\tekla{s,t}$ for $h$ small enough depending on $\tilde{\omega}\in\tilde{\Omega}$.
We start with the decomposition
\begin{multline}
\iO\Ihxy{\trkla{\tilde{u}\h\trkla{t}-\tilde{u}\h\trkla{s}}v}\dx\dy\\
=\iO\trkla{\tilde{u}\h\trkla{t}-\tilde{u}\h\trkla{s}}\Ihxy{v}\dx\dy-\iO\trkla{\ids-\Ihxyop}\tgkla{\trkla{\tilde{u}\h\trkla{t}-\tilde{u}\h\trkla{s}}\Ihxy{v}}\dx\dy\,.
\end{multline}
Due to the strong convergence of $\tilde{u}\h$ in $C\trkla{\tekla{0,\Tmax};L^q\trkla{\Om}}$ $\tilde{\Prob}$-almost surely and the convergence properties of $\Ihxyop$ (cf.~Theorem 4.4.20 in \cite{BrennerScott}), the first term on the right-hand side converges $\tilde{\Prob}$-almost surely towards $\iO\trkla{\tilde{u}\trkla{t}-\tilde{u}\trkla{s}}v\dx\dy$, while the second term vanishes due to Lemma \ref{lem:Ih:error}, as the following computation shows.
\begin{align}
\begin{split}
\abs{\iO\!\trkla{\ids-\Ihxyop}\tgkla{\trkla{\tilde{u}\h\trkla{t}-\tilde{u}\h\trkla{s}}\Ihxy{v}}\dx\dy}&\leq Ch\norm{\tilde{u}\h\trkla{t}-\tilde{u}\h\trkla{s}}_{L^1\trkla{\Om}}\norm{\nabla\Ihxy{v}}_{L^\infty\trkla{\Om}}\\
&\leq Ch \iO\tilde{u}\h^0\dx\dy \norm{\nabla v}_{L^\infty\trkla{\Om}}\leq Ch\,.
\end{split}
\end{align}
Here, we used a standard inverse estimate (cf.~Theorem 4.5.11 in \cite{BrennerScott}) and the non-negativity of $\tilde{u}\h$.
In order to deal with the remaining terms, we use the decomposition
\begin{align}
\begin{split}
\int_{t_1}^{t_2}\iO&\Ihy{\sqrt{\meanGinvX{\tilde{u}\h}}\tilde{J}\h^x\parx\Ihxy{v}}\dx\dy\dt \\
=&\int_{t_1}^{t_2}\iO \Ihy{\sqrt{\meanGinvX{\tilde{u}\h}}}\tilde{J}\h^x\parx\Ihxy{v}\dx\dy\dt\\
&-\int_{t_1}^{t_2}\iO\trkla{\ids-\Ihyop}\gkla{\tilde{J}\h^x\Ihy{\sqrt{\meanGinvX{\tilde{u}\h}}\parx\Ihxy{v}}}\dx\dy\dt\\
&-\int_{t_1}^{t_2}\iO \tilde{J}\h^x\trkla{\ids-\Ihyop}\gkla{\Ihy{\sqrt{\meanGinvX{\tilde{u}\h}}}\parx\Ihxy{v}}\dx\dy\dt\\
=&\, A+B+C\,.
\end{split}
\end{align}
Applying \eqref{eq:Ihy:lp:pointwise} from Lemma \ref{lem:Ih:error} pointwise in $x$ together with an inverse estimate (cf.~Theorem 4.5.11 \cite{BrennerScott}) and applying Hölder's inequality, we obtain
\begin{align}
\begin{split}
\tabs{B}\leq &\, Ch\int_{t_1}^{t_2}\norm{\tilde{J}\h^x}_{L^2\trkla{\Om}}\norm{\pary\Ihy{\sqrt{\meanGinvX{\tilde{u}\h}}\parx\Ihxy{v}}}_{L^2\trkla{\Om}}\dt\,.
\end{split}
\end{align}
Applying \eqref{eq:Ihy:dylp:pointwise} from Lemma \ref{lem:Ih:error} and standard inverse estimates (cf.~Theorem 4.5.11 in \cite{BrennerScott}), we deduce
\begin{align}
\begin{split}
&\norm{\pary\Ihy{\sqrt{\meanGinvX{\tilde{u}\h}}\parx\Ihxy{v}}}_{L^2\trkla{\Om}}\\
&\leq \norm{\pary\rkla{\Ihy{\sqrt{\meanGinvX{\tilde{u}\h}}}\parx\Ihxy{v}}}_{L^2\trkla{\Om}} \\
&\qquad+\norm{\pary\trkla{\ids-\Ihyop}\gkla{\Ihy{\sqrt{\meanGinvX{\tilde{u}\h}}}\parx\Ihxy{v}}}_{L^2\trkla{\Om}}\\
&\leq \norm{\pary\rkla{\Ihy{\sqrt{\meanGinvX{\tilde{u}\h}}}}\parx\Ihxy{v}}_{L^2\trkla{\Om}}+\norm{\Ihy{\sqrt{\meanGinvX{\tilde{u}\h}}}\pary\parx\Ihxy{v}}_{L^2\trkla{\Om}}\\
&\qquad+C\norm{\Ihy{\sqrt{\meanGinvX{\tilde{u}\h}}}}_{L^2\trkla{\Om}}\norm{\pary\parx\Ihxy{v}}_{L^\infty\trkla{\Om}}\,.
\end{split}
\end{align}
As $v$ was chosen to be sufficiently regular to control $\norm{\pary\parx\Ihxy{v}}_{L^\infty\trkla{\Om}}$, the only problematic term is the first one on the right-hand side.
In view of \eqref{eq:conv:mobx}, we use the regularity of $v$ and apply an inverse estimate (cf.~Theorem 4.5.11 in \cite{BrennerScott}) to obtain
\begin{multline}
\norm{\pary\rkla{\Ihy{\sqrt{\meanGinvX{\tilde{u}\h}}}}\parx\Ihxy{v}}_{L^2\trkla{\Om}}\\
\leq C\norm{\pary\rkla{\Ihy{\sqrt{\meanGinvX{\tilde{u}\h}}}-\tilde{u}\h}}_{L^2\trkla{\Om}} +C\norm{\pary \tilde{u}\h}_{L^2\trkla{\Om}}\\
\leq Ch^{-1}\norm{\Ihy{\sqrt{\meanGinvX{\tilde{u}\h}}}-\tilde{u}\h}_{L^2\trkla{\Om}} +C\norm{\pary\tilde{u}\h}_{L^2\trkla{\Om}}\,.
\end{multline}
In conclusion, by Hölder's inequality, we have
\begin{align}
\begin{split}
\tabs{B}\leq&\, C h \norm{\tilde{J}\h^x}_{L^2\trkla{0,\Tmax;L^2\trkla{\Om}}}\rkla{\norm{\Ihy{\sqrt{\meanGinvX{\tilde{u}\h}}}}_{L^2\trkla{0,\Tmax;L^2\trkla{\Om}}}\!\!\!+\norm{\pary\tilde{u}\h}_{L^2\trkla{0,\Tmax;L^2\trkla{\Om}}}}\\
&+C\norm{\tilde{J}\h^x}_{L^2\trkla{0,\Tmax;L^2\trkla{\Om}}}\norm{\Ihy{\sqrt{\meanGinvX{\tilde{u}\h}}}-\tilde{u}\h}_{L^2\trkla{0,\Tmax;L^2\trkla{\Om}}}\,.
\end{split}
\end{align}
Therefore, in view of Lemma \ref{lem:convergence2}, $B$ vanishes in $L^q\trkla{\tilde{\Omega}}$ for any $q<\infty$.
In the same spirit, we use Hölder's inequality and \eqref{eq:Ihy:lp:pointwise} in Lemma \ref{lem:Ih:error} to prove that $C$ vanishes in $L^q\trkla{\tilde{\Omega}}$.
\begin{align}
\begin{split}
\tabs{C}\leq &\,\int_{t_1}^{t_2}\norm{\tilde{J}\h^x}_{L^2\trkla{\Om}}\norm{\trkla{\ids-\Ihyop}\gkla{\Ihy{\sqrt{\meanGinvX{\tilde{u}\h}}}\parx\Ihxy{v}}}_{L^2\trkla{\Om}}\dt\\
\leq&\, Ch\int_{t_1}^{t_2}\norm{\tilde{J}\h^x}_{L^2\trkla{\Om}}\norm{\Ihy{\sqrt{\meanGinvX{\tilde{u}\h}}}}_{L^2\trkla{\Om}}\norm{\pary\parx\Ihxy{v}}_{L^\infty\trkla{\Om}}\dt\\
\leq&\,Ch\,.
\end{split}
\end{align}
Therefore, it remains to analyze the convergence properties of $A$.
As $\tilde{J}\h^x$ converges weakly towards $\tilde{J}^x$ $\tilde{\Prob}$-almost surely in $L^2\trkla{0,\Tmax;L^2\trkla{\Om}}$, we may use the strong convergence of $\Ihy{\sqrt{\meanGinvX{\tilde{u}\h}}}$ in $L^\infty\trkla{0,\Tmax;L^q\trkla{\Om}}$ with $q<\infty$ (cf.~\eqref{eq:conv:mobx} in Lemma \ref{lem:convergence2}) and $\parx\Ihxy{v}$ in $L^\infty\trkla{\Om}$ to conclude that
\begin{align}
A\rightarrow \int_{t_1}^{t_2}\iO \tilde{u} \tilde{J}^x\parx v\dx\dy\dt&&\tilde{\Prob}\text{-almost surely.}
\end{align}
Analogous arguments provide the $\tilde{\Prob}$-almost sure convergence 
\begin{align}
\int_{t_1}^{t_2}\iO\Ihy{\sqrt{\meanGinvX{\tilde{u}\h}}\tilde{J}\h^x\parx\Ihxy{v}}\dx\dy\dt\rightarrow \int_{t_1}^{t_2}\iO \tilde{u} \tilde{J}^y\pary v\dx\dy\dt\,.
\end{align}
As the $\Psi$-term is continuous, uniformly bounded, and converges $\tilde{\Prob}$-almost surely, it remains to control higher moments to conclude the proof by applying Vitali's convergence theorem.

As $\tilde{u}\h$ is uniformly bounded in $L^q\trkla{\tilde{\Omega};C\trkla{\tekla{0,\Tmax};L^q\trkla{\Om}}}$ for arbitrary $q<\infty$ and the error terms $B$ and $C$, which where introduced by the nodal interpolation operators, vanish also in $L^q\trkla{\tilde{\Omega}}$, it remains to establish the integrability of appropriate higher moments of 
\begin{subequations}\label{eq:tmp:integrals}
\begin{align}
&&\int_{t_1}^{t_2}\iO \Ihy{\sqrt{\meanGinvX{\tilde{u}\h}}}\tilde{J}\h^x\parx\Ihxy{v}\dx\dy\dt\label{eq:tmp:integrals:a}\\
\text{~and~}&&\int_{t_1}^{t_2}\iO \Ihx{\sqrt{\meanGinvY{\tilde{u}\h}}}\tilde{J}\h^y\pary\Ihxy{v}\dx\dy\dt\,.\label{eq:tmp:integrals:b}
\end{align}
\end{subequations}
The desired integrability of higher moments of the first integral follows from the estimate
\begin{multline}
\abs{\int_{t_1}^{t_2}\iO \Ihy{\sqrt{\meanGinvX{\tilde{u}\h}}}\tilde{J}\h^x\parx\Ihxy{v}\dx\dy\dt}\\
\leq C\norm{\tilde{J}^x\h}_{L^2\trkla{0,\Tmax;L^2\trkla{\Om}}}\norm{\Ihy{\sqrt{\meanGinvX{\tilde{u}\h}}}}_{L^\infty\trkla{0,\Tmax;L^q\trkla{\Om}}}\norm{\parx\Ihxy{v}}_{L^\infty\trkla{\Om}}\,.
\end{multline}
Together with the bounds from Proposition \ref{prop:convergence1} and Lemma \ref{lem:convergence2}, we obtain the uniform integrability of a $q$-moment for $q>1$.
A similar argument provides the uniform integrability of the second term in \eqref{eq:tmp:integrals:b}, which concludes the proof.
\end{proof}
\begin{lemma}\label{lem:mart2}
Let the Assumptions \ref{item:S}, \ref{item:initial}, \ref{item:potential}, \ref{item:stoch}, \ref{item:regularization},  \ref{item:stochbasis:boundthirdderivative}, and \ref{item:stochbasis:convergence} hold true.
Then for all $\tekla{0,1}$-valued continuous functions $\Psi$ defined on $\Xu\times\XW$, we have
\begin{multline}
\tilde{\mathds{E}}\left[\left(\martt^2\trkla{t}-\martt^2\trkla{s} - \int_s^t\sum_{k,l\in\Z}\lx{kl}^2\rkla{\iO\tilde{u}\g{kl}\parx v\dx\dy}^2\dtau \right.\right.\\
\left.\left.-\int_s^t\sum_{k,l\in\Z}\ly{kl}^2\rkla{\iO\tilde{u}\g{kl}\pary v\dx\dy}^2\dtau\right)\Psi\trkla{r_s\tilde{u},r_s\tilde{\bs{W}}}\right]=0
\end{multline}
for all $0\leq s\leq t\leq\Tmax$, where
\begin{align}
\begin{split}
\martt\trkla{t}:=&\,\iO\trkla{\tilde{u}\trkla{t}-\tilde{u}\trkla{0}}v\dx\dy+\int_0^t\iO\tilde{u}\tilde{J}^x\parx v\dx\dy\dtau+\int_0^t\iO\tilde{u}\tilde{J}^y\pary v\dx\dy\dtau\,.
\end{split}
\end{align}
\end{lemma}
\begin{proof}
We will prove this by passing to the limit in the martingale \eqref{eq:mt2}.
Recalling the arguments from the proof of Lemma \ref{lem:mart1}, we obtain that
\begin{multline}
\martht^2\trkla{t}:=\left(\iO\Ihxy{\trkla{\tilde{u}\h\trkla{t}-\tilde{u}\h\trkla{0}}v}\dx\dy\right.\\
+\itTht\iO\Ihy{\sqrt{\meanGinvX{\tilde{u}\h}}\tilde{J}\h^x\parx\Ihxy{v}}\dx\dy\dtau\\
\left.+\itTht\iO\Ihx{\sqrt{\meanGinvY{\tilde{u}\h}} \tilde{J}\h^y\pary\Ihxy{v}}\dx\dy\dtau\right)^2
\end{multline}
converges along a subsequence $\tilde{\Prob}$-almost surely towards
\begin{multline}
\martt^2\trkla{t}=\left(\iO\trkla{\tilde{u}\trkla{t}-\tilde{u}\trkla{0}}v\dx\dy\right.\\+\int_0^t\iO\tilde{u}\tilde{J}^x\parx v\dx\dy\dtau \left.+\int_0^t\iO\tilde{u}\tilde{J}^y\pary v\dx\dy\dtau\right)^2\,.
\end{multline}
In order to deduce the convergence of the corresponding expected values, we need to establish higher regularity of $\martht$.
Starting from the representation 
\begin{align}
\begin{split}
\martht\trkla{t}=&\,-\sum_{k\in\Zx}\sum_{l\in\Zy}\lx{kl}\itTht\iO\Ihy{\Ihxloc{\parx\trkla{\tilde{u}\h\gt{kl}}\Ihxy{v}}}\dx\dy\dbthx{kl}\\
&\,-\sum_{k\in\Zx}\sum_{l\in\Zy}\ly{kl}\itTht\iO\Ihx{\Ihyloc{\pary\trkla{\tilde{u}\h\gt{kl}}\Ihxy{v}}}\dx\dy\dbthy{kl}
\end{split}
\end{align}
and combining the martingale moment inequality 
\begin{align}
\expectedt{\tabs{\martht\trkla{t}}^{2q}}\leq C_q\expectedt{\qvar{\martht}_t^q} &&\text{for any~}q>0
\end{align}
(see e.g. Proposition 3.26 in Chapter 3 of \cite{KaratzasShreve2004}) with Lemma \ref{lem:quadraticvariation} formulated for $\martht$, we obtain
\begin{align}
\begin{split}\label{eq:tmp:mht2:uniform}
\expectedt{\tabs{\martht\trkla{t}}^{2q}}\leq C\expectedt{\qvar{\martht}_t^q}\leq &\,C\norm{v}_{H^2\trkla{\Om}}^{2q}\expectedt{\rkla{\itTht\norm{\tilde{u}\h\trkla{s}}_{H^1\trkla{\Om}}^2\ds}^q}\\
\leq&\,C\norm{v}_{H^2\trkla{\Om}}^{2q}\Tmax^q\expectedt{\sup_{s\in\tekla{0,\Tmax}}\mathcal{E}\h\trkla{\tilde{u}\h}^q+C}\,.
\end{split}
\end{align}
It remains to pass to the limit in the remaining integrals in \eqref{eq:mt2}, i.e.  we  have to analyze the convergence behavior of $\qvar{\martht}_t$ for $h\searrow0$ (cf.~\eqref{eq:qvarmth}).
The first step is to show that we may neglect the interpolation operators when passing to the limit.
Applying \eqref{eq:Ihlocpar} and \eqref{eq:Ih=Ihloc}, we may rewrite the first component of $\martht$ using
\begin{multline}\label{eq:tmp:qvar1}
\iO\Ihy{\Ihxloc{\parx\trkla{\tilde{u}\h\gt{kl}}\Ihxy{v}}}\dx\dy\\
=\iO\Ihy{\parx\tilde{u}\h\Ihx{\gt{kl}v}}\dx\dy+\iO\Ihy{\parx\gt{kl}\Ihx{\tilde{u}\h v}}\dx\dy\\
=\iO\parx\tilde{u}\h\gt{kl}\Ihxy{v}\dx\dy+\iO\parx\gt{kl}\tilde{u}\h\Ihxy{v}\dx\dy\\
-\iO\trkla{\ids-\Ihyop}\tgkla{\parx\tilde{u}\h\Ihxy{\gt{kl}v}}\dx\dy-\iO\parx\tilde{u}\h\trkla{\ids-\Ihxyop}\tgkla{\gt{kl}\Ihxy{v}}\dx\dy \\
-\iO\trkla{\ids-\Ihyop}\tgkla{\parx\gt{kl}\Ihxy{\tilde{u}\h v}}\dx\dy-\iO\parx\gt{kl}\trkla{\ids-\Ihxyop}\tgkla{\tilde{u}\h\Ihxy{v}}\dx\dy\\
=:\iO\parx\trkla{\tilde{u}\h\gt{kl}}\Ihxy{v}\dx\dy + A_{kl}+B_{kl}+C_{kl}+D_{kl}\,,
\end{multline}
where we used that $\gt{kl}\in\Uhxy$.
Combining \eqref{eq:tmp:qvar1} with the binomial theorem and \eqref{eq:qvarmth}, we get
\begin{multline}\label{eq:tmp:difference}
\left|\sum_{k\in\Zx}\sum_{l\in\Zy}\lx{kl}^2\itTht\rkla{\iO\Ihy{\Ihxloc{\parx\trkla{\tilde{u}\h\gt{kl}}\Ihxy{v}}}\dx\dy}^2\dtau\right.\\
-\left.\sum_{k\in\Zx}\sum_{l\in\Zy}\lx{kl}^2\itTht\rkla{\iO\parx\trkla{\tilde{u}\h\gt{kl}}\Ihxy{v}\dx\dy}^2\dtau\right|\\
\leq C\sqrt{\qvar{\martht}_t}\rkla{\sum_{k\in\Zx}\sum_{l\in\Zy}\lx{kl}^2\itTht A_{kl}^2+B_{kl}^2+C_{kl}^2+D_{kl}^2\dtau}^{1/2}\\
+C\sum_{k\in\Zx}\sum_{l\in\Zy}\lx{kl}^2\itTht A_{kl}^2+B_{kl}^2+C_{kl}^2+D_{kl}^2\dtau\,=:\trkla{\#}\,.
\end{multline}
Using Lemma \ref{lem:Ih:error} and standard inverse estimates (cf.~Theorem 4.5.11 in \cite{BrennerScott}), we obtain the estimates
\begin{align}
A_{kl}\leq&\, Ch\norm{\parx\tilde{u}\h}_{L^2\trkla{\Om}}\norm{\pary\Ihxy{\gt{kl}v}}_{L^2\trkla{\Om}}\leq Ch\norm{\parx\tilde{u}\h}_{L^2\trkla{\Om}}\norm{\pary\trkla{\gt{kl}v}}_{L^\infty\trkla{\Om}}\,,\\
B_{kl}\leq&\,Ch \norm{\parx\tilde{u}\h}_{L^2\trkla{\Om}}\norm{\gt{kl}}_{L^4\trkla{\Om}}\norm{\nabla\Ihxy{v}}_{L^4\trkla{\Om}}\,,\\
\begin{split}
C_{kl}\leq&\,Ch\norm{\parx\pary\gt{kl}}_{L^\infty\trkla{\Om}}\norm{\Ihxy{\tilde{u}\h v}}_{L^1\trkla{\Om}}\\
\leq&\, Ch\norm{\parx\pary\g{kl}}_{L^\infty\trkla{\Om}}\norm{\tilde{u}\h}_{L^2\trkla{\Om}}\norm{\Ihxy{v}}_{L^2\trkla{\Om}}\,,\end{split}\\
D_{kl}\leq&\,C h\norm{\parx\gt{kl}}_{L^\infty\trkla{\Om}}\norm{\tilde{u}\h}_{L^2\trkla{\Om}}\norm{\nabla\Ihxy{v}}_{L^2\trkla{\Om}}\,.
\end{align}
Recalling Assumption \ref{item:stochbasis:bounds}, \eqref{eq:normequivalence}, and the regularity of $v$, we obtain
\begin{align}
\trkla{\#}\leq  Ch^{1/2}\sqrt{\qvar{\martht}_t}\rkla{\itTht \norm{\tilde{u}\h}_{H^1\trkla{\Om}}^2\dtau}^{1/2} +Ch \itTht \norm{\tilde{u}\h}_{H^1\trkla{\Om}}^2\dtau\,.
\end{align}
Therefore, the $\p$-th moment of the left-hand side of \eqref{eq:tmp:difference} vanishes, which in particular provides convergence $\tilde{\Prob}$-almost surely after restricting ourselves to appropriate subsequences.
It remains to discuss the convergence properties of 
\begin{align}
\sum_{k\in\Zx}\sum_{l\in\Zy}\lx{kl}\itTht\rkla{\iO\parx\trkla{\tilde{u}\h\gt{kl}}\Ihxy{v}\dx\dy}^2\dtau\,.
\end{align}
After integrating by parts, we may use the standard error estimates for the nodal interpolation operator (cf.~Theorem 4.4.20 in \cite{BrennerScott}) and Assumption \ref{item:stochbasis:bounds} to obtain the strong convergence of $\lx{kl}\gt{kl}$ towards $\lx{kl}\g{kl}$ in $L^\infty\trkla{\Om}$ and the strong convergence of $\parx\Ihxy{v}$ towards $\parx v$ in $L^\infty\trkla{\Om}$.
Together with \eqref{eq:conv:strong1}, this provides the $\tilde{\Prob}$-almost sure convergence.
Similar considerations provide the convergence of 
\begin{align*}
\sum_{k\in\Zx}\sum_{l\in\Zy}\lx{kl}\itTht\rkla{\iO\Ihx{\Ihyloc{\pary\trkla{\tilde{u}\h\gt{kl}}\Ihxy{v}}}\dx\dy}^2\dtau
\end{align*}
$\tilde{\Prob}$-almost surely.
As we already established the higher integrability of $\qvar{\martht}_t$ in \eqref{eq:tmp:mht2:uniform}, we may conclude by applying Vitali's theorem.
\end{proof}
In the same spirit, we get the following result.
\begin{lemma}
Let the Assumptions \ref{item:S}, \ref{item:initial}, \ref{item:potential}, \ref{item:stoch}, \ref{item:regularization},  \ref{item:stochbasis:boundthirdderivative}, and \ref{item:stochbasis:convergence} hold true.
Then for all $\tekla{0,1}$-valued continuous functions $\Psi$ defined on $\Xu\times\XW$, we have
\begin{align}
\expectedt{\rkla{\martt\trkla{t}\btalpha{kl}\trkla{t} -\martt\trkla{s}\btalpha{kl}\trkla{s} -\lalpha{kl}\int_s^t\iO\para{\alpha}\trkla{\tilde{u}\g{kl}}v\dx\dy\dtau} \Psi\trkla{r_s\tilde{u},r_s\tilde{\bs{W}}}}=0
\end{align} 
for all $k,l\in\Z$, $\alpha\in\tgkla{x,y}$, and all $s\leq t\in\tekla{0,\Tmax}$.
\end{lemma}
\begin{lemma}\label{lem:identification:stochint}
Let the Assumptions \ref{item:S}, \ref{item:initial}, \ref{item:potential}, \ref{item:stoch}, \ref{item:regularization},  \ref{item:stochbasis:boundthirdderivative}, and \ref{item:stochbasis:convergence} hold true.
Then, we have
\begin{align}
\martt\trkla{t}=\sum_{k,l\in\Z}\lx{kl}\int_0^t\iO\parx\trkla{\tilde{u}\g{kl}}v\dx\dy\dbtx{kl}+\sum_{k,l\in\Z}\ly{kl}\int_0^t\iO\pary\trkla{\tilde{u}\g{kl}}v\dx\dy\dbty{kl}\,.
\end{align}
\end{lemma}
\begin{proof}
As a martingale with vanishing quadratic variation is almost surely constant, it is sufficient to show that  
\begin{multline}
0=\qvar{\martt\trkla{\cdot}}_T+\qvar{\sum_{\alpha\in\tgkla{x,y}}\sum_{k,l\in\Z}\int_0^{\trkla{\cdot}}\lalpha{kl}\iO\para{\alpha}\trkla{\tilde{u}\g{kl}}v\dx\dy\dbtalpha{kl}}_T\\
-2\crossvar{\martt\trkla{\cdot}}{\sum_{\alpha\in\tgkla{x,y}}\sum_{k,l\in\Z}\int_0^{\trkla{\cdot}}\lalpha{kl}\iO\para{\alpha}\trkla{\tilde{u}\g{kl}}v\dx\dy\dbtalpha{kl}}_T\,.
\end{multline}
To compute the last term on the right-hand side, we use the cross variation formula, which can be found e.g. in Lemma 2.16 in Section 3.2 of \cite{KaratzasShreve2004}, to obtain
\begin{multline}
\crossvar{\martt\trkla{\cdot}}{\sum_{\alpha\in\tgkla{x,y}}\sum_{k,l\in\Z}\int_0^{\trkla{\cdot}}\lalpha{kl}\iO\para{\alpha}\trkla{\tilde{u}\g{kl}}v\dx\dy\dbtalpha{kl}}_T\\
=\sum_{\alpha\in\tgkla{x,y}}\sum_{k,l\in\Z}\int_0^T\lalpha{kl}\iO\para{\alpha}\trkla{\tilde{u}\g{kl}}v\dx\dy\operatorname{d}\!\crossvar{\martt\trkla{\cdot}}{\btalpha{kl}\trkla{\cdot}}_s\,.
\end{multline}
Following the arguments in \cite{FischerGrun2018}, it is possible to show that the process $s\mapsto \crossvar{\martt\trkla{\cdot}}{\btalpha{kl}}_s$ is absolutely continuous $\tilde{\Prob}$-almost surely and consequently
\begin{align}
\operatorname{d}\crossvar{\martt\trkla{\cdot}}{\btalpha{kl}}_s=\lalpha{kl}\iO\para{\alpha}\trkla{\tilde{u}\trkla{s}\g{kl}}v\dx\dy\ds\,.
\end{align}
Using the identities
\begin{multline}
\qvar{\martt\trkla{\cdot}}_T=\sum_{\alpha\in\tgkla{x,y}}\sum_{k,l\in\Z}\lalpha{kl}\int_0^T\rkla{\iO\para{\alpha}\trkla{\tilde{u}\g{kl}}v\dx\dy}^2\ds\\
=\qvar{\sum_{\alpha\in\tgkla{x,y}}\sum_{k,l\in\Z}\lalpha{kl}\int_0^{\trkla{\cdot}}\iO\para{\alpha}\trkla{\tilde{u}\g{kl}}v\dx\dy\dbtalpha{kl}}_T\,,
\end{multline}
we conclude
\begin{align}
\qvar{\martt\trkla{\cdot}-\sum_{\alpha\in\tgkla{x,y}}\sum_{k,l\in\Z}\lalpha{kl}\int_0^{\trkla{\cdot}}\iO\para{\alpha}\trkla{\tilde{u}\g{kl}}v\dx\dy\dbtalpha{kl}}_T\equiv0\,.
\end{align}
\end{proof}
Having established the previous results, we are now in the position to prove Theorem \ref{th:existence}.
\begin{proof}[Proof of Theorem \ref{th:existence}]
From Proposition \ref{prop:convergence1}, Corollary \ref{cor:tildedeltau}, and Lemma \ref{lem:QWiener}, we infer the existence of a stochastic basis $\trkla{\tilde{\Omega},\tilde{\mathcal{F}},\trkla{\tilde{\mathcal{F}}_t}_{t\geq0},\tilde{\Prob}}$, of a Wiener process 
\begin{align}
\tilde{\bs{W}}\trkla{t}=\sum_{\alpha\in\tgkla{x,\,y}}\sum_{k,\,l\in\Z}\lalpha{kl}\g{kl}\btalpha{kl}\kbalpha\,,
\end{align}
and of random variables
\begin{subequations}
\begin{align}
\label{eq:5101}
\begin{split}
\tilde{u}\in&\, L^q\trkla{\tilde{\Om};L^\infty\trkla{0,\Tmax;H^1_\per\trkla{\Om}}}\cap L^2\trkla{\tilde{\Omega};L^2\trkla{0,\Tmax;H^2_\per\trkla{\Om}}}\\
&\qquad\qquad\cap L^\sigma\trkla{\tilde{\Omega};C^{1/4}\trkla{\tekla{0,\Tmax};\trkla{H^1_\per\trkla{\Om}}^\prime}}\,,
\end{split}\\
\tilde{J}^x\in&\,L^2\trkla{\tilde{\Omega};L^2\trkla{0,\Tmax;L^2\trkla{\Om}}}\,,\\
\tilde{J}^y\in&\,L^2\trkla{\tilde{\Omega};L^2\trkla{0,\Tmax;L^2\trkla{\Om}}}\
\end{align}
\end{subequations}
with $q<\infty$ and $\sigma<8/5$.
As shown in Lemma \ref{lem:convergence2}, these random variables satisfy
\begin{subequations}
\begin{align}
\tilde{J}^x=&\,\tilde{u}\parx\trkla{-\Delta\tilde{u}+F^\prime\trkla{\tilde{u}}}&&\tilde{\Prob}\text{-almost surely in~} \tekla{\tilde{u}>0}\,,\\
\tilde{J}^y=&\,\tilde{u}\pary\trkla{-\Delta\tilde{u}+F^\prime\trkla{\tilde{u}}}&&\tilde{\Prob}\text{-almost surely in~} \tekla{\tilde{u}>0}\,.
\end{align}
\end{subequations}
Furthermore, we have $\Lambda=\tilde{\Prob}\circ\trkla{\tilde{u}^0}^{-1}$ by construction.
Lemma \ref{lem:mart1} implies that
\begin{align}
\begin{split}
\martt\trkla{t}=&\,\iO\trkla{\tilde{u}\trkla{t}-\tilde{u}^0}v\dx\dy+\int_0^t\iO\tilde{u}\tilde{J}^x\parx v\dx\dy\ds+\int_0^t\iO\tilde{u}\tilde{J}^y\pary v\dx\dy\ds
\end{split}
\end{align}
is an $\trkla{\tilde{\mathcal{F}}_t}_{t\geq0}$-martingale, and by Lemma \ref{lem:identification:stochint}, we obtain
\begin{multline}
\iO\trkla{\tilde{u}\trkla{t}-\tilde{u}^0}v\dx\dy+\int_0^t\iO\tilde{u}\tilde{J}^x\parx v\dx\dy\ds+\int_0^t\iO\tilde{u}\tilde{J}^y\pary v\dx\dy\ds\\
=\sum_{k,l\in\Z}\lx{kl}\int_0^t\iO\parx\trkla{\tilde{u}\g{kl}}v\dx\dy\dbtx{kl}+\sum_{k,l\in\Z}\ly{kl}\int_0^t\iO\pary\trkla{\tilde{u}\g{kl}}v\dx\dy\dbty{kl}\,.
\end{multline}
It remains to establish the energy estimate \eqref{eq:gg100}.
Starting from Proposition \ref{prop:energyentropyestimate}, using Fatou's lemma and the definition of the fluxes \eqref{eq:def:jh}, we find
\begin{multline}\label{eq:gg101}
\expectedt{\liminf_{h\searrow0}\rkla{\sup_{t\in\tekla{0,\Tmax}} \rkla{\tfrac12\iO\Ihy{\tabs{\parx \tilde{u}\h}^2}\!+\Ihx{\tabs{\pary\tilde{u}\h}^2}\dx\dy +\!\!\iO\Ihxy{F\trkla{\tilde{u}\h}}\dx\dy}^\p }}\\
+\expectedt{\liminf_{h\searrow0}\int_0^{\Tmax}\iO\tabs{\tilde{J}\h^x}^2+\tabs{\tilde{J}\h^y}^2\dx\dy\dt}\\
\leq \liminf_{h\searrow0}\expectedt{\sup_{t\in\tekla{0,\Tmax}}\mathcal{E}\h\trkla{\tilde{u}\h}^\p+\int_0^{\Tmax}\iO\Ihy{\tabs{\tilde{J}\h^x}^2} +\Ihx{\tabs{\tilde{J}\h^y}^2}\dx\dy\dt}\\
\leq C\trkla{\p,\,u^0,\,\Tmax}\,,
\end{multline}
where we used $\tfrac12 h^\varepsilon\iO\Ihxy{\tabs{\Delta\h\tilde{u}\h}^2}\dx\dy\geq0$.
By the norm equivalence \eqref{eq:normequivalence}, we obtain
\begin{multline}
\expectedt{\liminf_{h\searrow0}\rkla{\sup_{t\in\tekla{0,\Tmax}}\rkla{\tfrac12\iO\tabs{\nabla\tilde{u}\h}^2\dx\dy+\iO\Ihxy{F\trkla{\tilde{u}\h}}\dx\dy}^\p }}\\
+\expectedt{\liminf_{h\searrow0}\int_0^{\Tmax}\iO\tabs{\tilde{J}\h^x}^2+\tabs{\tilde{J}\h^y}^2\dx\dy\dt}\leq C\trkla{\p,\,u^0,\,\Tmax}\,.
\end{multline}
Similarly as in the proof of Theorem 3.2 in \cite{FischerGrun2018}, we may use the lower semi-continuity in appropriate topologies to find that
\begin{align}
\expected{\sup_{t\in\tekla{0,\Tmax}}\rkla{\tfrac12\iO\tabs{\nabla\tilde{u}\h}^2\dx\dy}^\p +\int_0^{\Tmax}\iO\tabs{\tilde{J}\h^x}^2+\tabs{\tilde{J}\h^y}^2\dx\dy\dt}\leq C\trkla{\p,\,u^0,\,\Tmax}\,.
\end{align}
So let us focus on the remaining term $\iO\Ihxy{F\trkla{\tilde{u}\h}}\dx\dy$.\\

\underline{\textbf{Step 1:}} There is a positive constant $C$ such that
\begin{align}
\iO F\trkla{\tilde{u}\h}\dx\dy\leq\iO\Ihxy{F\trkla{\tilde{u}\h}}\dx\dy +C\,.\label{eq:gg102}
\end{align}
Indeed, with the notation of Assumption \ref{item:potential}, we find $F$ to be convex for $s\in\tekla{0,\hat{u}}$ where $\hat{u}$ is given by 
\begin{align}
\hat{u}=\sqrt[p+2]{\frac{\tilde{c}_1}{\tilde{c}_2}}\,\label{eq:gg103}
\end{align}
with the constants $\tilde{c}_1$ and $\tilde{c}_2$ introduced in \ref{item:potential}.
Introducing $\hat{\delta}:=\hat{u}/C_{\text{osc}}$ with $C_{\text{osc}}$ being the constant in \eqref{eq:oscillation} and using the definition for $S_{\hat{\delta}}^{\Qh}$ from \eqref{eq:SdeltaQhot}, we find that
\begin{align}
\max_{\trkla{x,y}\in\Om\setminus S_{\hat{\delta}}^{\Qh}} \tilde{u}\h\leq \hat{u}\,.
\end{align}
Therefore, we have
\begin{align}
F\trkla{\tilde{u}\h}\leq\Ihxy{F\trkla{\tilde{u}\h}}&&\text{on~}\Om\setminus S_{\hat{\delta}}^{\Qh}\label{eq:gg104}
\end{align}
due to the convexity of $F$ on $\trkla{0,\hat{u}}$.

Using \ref{item:potential} once more, we find
\begin{align}
0\leq F\trkla{\tilde{u}\h}\leq C\trkla{\hat{\delta}^{-p}+1}&&\text{on~} S_{\hat{\delta}}^{\Qh}\,.\label{eq:gg105}
\end{align}
Combining \eqref{eq:gg104} and \eqref{eq:gg105}, \eqref{eq:gg102} is established.\\

\underline{\textbf{Step 2:}} Estimate on $\expectedt{\esssup_{t\in\tekla{0,\Tmax}}\rkla{\iO F\trkla{\tilde{u}}\dx\dy}^\p}$.\\

From \eqref{eq:conv:strong1} and \eqref{eq:conv:strong2} together with Sobolev's embedding result, we infer $\tilde{\Prob}$-a.s.~that for almost all $t\in\tekla{0,\Tmax}$
\begin{align}
\tilde{u}\h\trkla{\tilde{\omega},t,\cdot}\rightarrow\tilde{u}\trkla{\tilde{\omega},t,\cdot}&&\text{in~} C^\gamma\trkla{\overline{\Om}}\,.
\end{align}
Hence, $F\trkla{\tilde{u}\h\trkla{\tilde{\omega},t,\cdot}}\rightarrow F\trkla{\tilde{u}\trkla{\tilde{\omega},t,\cdot}}$ pointwise which implies together with Fatou's lemma
\begin{align}
\esssup_{t\in\tekla{0,\Tmax}}\iO F\trkla{\tilde{u}\trkla{\tilde{\omega},t,\cdot}}\dx\dy\leq \liminf_{h\searrow0} \sup_{t\in\tekla{0,\Tmax}} \iO F\trkla{\tilde{u}\h\trkla{\tilde{\omega},t,\cdot}}\dx\dy.\,\label{eq:step2}
\end{align}
As both sides of inequality \eqref{eq:step2} are non-negative, we can take the $\p$-th power on both sides.
Taking the expectation concludes this step.\\

\underline{\textbf{Step 3:}} Estimate on $\expectedt{\sup_{t\in\tekla{0,\Tmax}}\rkla{\iO F\trkla{\tilde{u}}\dx\dy}^\p}$.\\

We fix $t_0\in\tekla{0,\Tmax}$ arbitrarily and choose $\tilde{\omega}\in\tilde{\Omega}$ such that $\tilde{u}\trkla{\tilde{\omega},\cdot,\cdot}\in C\trkla{\tekla{0,\Tmax};L^q\trkla{\Om}}$ and $M\trkla{\tilde{\omega}}:=\esssup_{t\in\tekla{0,\Tmax}}\iO F\trkla{\tilde{u}}\dx\dy<\infty$.
The latter limitations are satisfied by almost all $\tilde{\omega}\in\tilde{\Omega}$ due to Proposition \ref{prop:convergence1} and Step 2.\\
We will now show that $\tilde{u}\trkla{\tilde{\omega},\,t_0,\,\cdot}$ vanishes only on a set of measure zero.
Therefore, we take a sequence $\trkla{s_n}_{n\in\N}$ in $\tekla{0,\Tmax}$ which satisfies $s_n\rightarrow t_0$ for $n\nearrow\infty$, 
\begin{align}
\iO F\trkla{\tilde{u}\trkla{\tilde{\omega},\,s_n,\,x,\,y}}\dx\dy\leq M\trkla{\tilde{\omega}}\,,
\end{align}
and $\tilde{u}\trkla{\tilde{\omega},\,s_n,\,\cdot}\rightarrow \tilde{u}\trkla{\tilde{\omega},\,t_0,\,\cdot}$ pointwise almost everywhere.
Assuming that the set $\mathcal{N}$, where $\tilde{u}\trkla{\tilde{\omega},\,t_0,\,\cdot}$ vanishes, has positive measure, we obtain by Egorov's theorem the existence of a set $\mathcal{N}^\delta\subset\mathcal{N}\subset\Om$ such that $\mu\trkla{\mathcal{N}^\delta}>\trkla{1-\delta}\mu\trkla{\mathcal{N}}$ for $\delta\in\trkla{0,1}$ such that $\tilde{u}\trkla{\tilde{\omega},\,s_n,\,\cdot}\rightarrow\tilde{u}\trkla{\tilde{\omega},\,t_0,\,\cdot}$ uniformly in $\mathcal{N}^\delta$.
This provides 
\begin{align}
\int_{\mathcal{N}^\delta} F\trkla{\tilde{u}\trkla{\tilde{\omega},s_n,x,y}}\dx\dy\stackrel{n\rightarrow\infty}{\longrightarrow}+\infty\,.
\end{align}
However, at the same time we have
\begin{align}
\int_{\mathcal{N}^\delta} F\trkla{\tilde{u}\trkla{\tilde{\omega},s_n,x,y}}\dx\dy\leq \iO F\trkla{\tilde{u}\trkla{\tilde{\omega},s_n,x,y}}\dx\dy \leq M\trkla{\tilde{\omega}}\,.
\end{align}
This contradiction provides $\mu\trkla{\mathcal{N}}=0$ and $F\trkla{\tilde{u}\trkla{\tilde{\omega},s_n,\cdot}}\rightarrow F\trkla{\tilde{u}\trkla{\tilde{\omega},t_0,\cdot}}$ pointwise almost everywhere.
Hence, by applying Fatou's lemma, we obtain
\begin{multline}
\iO F\trkla{\tilde{u}\trkla{\tilde{\omega},t_0,x,y}}\dx\dy\leq \liminf_{n\rightarrow \infty}\iO F\trkla{\tilde{u}\trkla{\tilde{\omega},s_n,x,y}}\dx\dy\\
\leq M\trkla{\tilde{\omega}} = \esssup_{t\in\tekla{0,\Tmax}} \iO F\trkla{\tilde{u}\trkla{\tilde{\omega},t,x,y}}\dx\dy\,.
\end{multline}
Again, taking the $\p$-th power and the expectation concludes this step.

Together with \eqref{eq:gg101} and \eqref{eq:gg102}, the energy inequality \eqref{eq:gg100} is established.
Finally, we may combine \eqref{eq:5101}  and  \eqref{eq:gg100} to deduce the positivity properties of $\tilde u$ claimed in the theorem. 
\end{proof}
\begin{remark}
As the estimate \eqref{eq:gg100} is only of a qualitative character we did not strive for an optimal result.
In fact, it is --- being based on \eqref{eq:gg105} --- a rather coarse estimate.
If more information is available on $F$, for instance number and height of local maxima, much better estimates are available --- based on appropriate convexity arguments.
\end{remark}

\section{Conclusion}\label{sec:conclusion}
We have proven the existence of martingale solutions to stochastic thin-film equations with conservative linear multiplicative noise in two space dimensions. As our result covers driving noise both in the \Ito- and in the Stratonovich-sense, we expect it to be a starting point to construct solutions for solely surface tension driven thin-film evolution (i.e. $F\equiv 0$) subject to compactly supported initial data.\footnote{In the case of \Ito-noise, no formal a priori-estimates are known for the case of compactly supported initial data.} This raises questions about the impact of noise on the evolution of the solution's support (``finite speed of propagation'' and ``waiting time phenomena''). It is well-known in the (deterministic) theory of thin-film equations that analytical concepts in two space dimensions carry over to higher dimensions while the argumentation in dimension $d=1$ takes advantage of the Sobolev embedding $H^1\hookrightarrow C^{1/2}$ and is therefore much less involved. Hence, a generalization to $d=3$ (with the perspective of applications to models for phase separation) is feasible as well. Finally, the implementation of numerical schemes related to the finite element approach presented here will provide further insight into the impact of noise on thin-film evolution.

\begin{appendix}
\section{Auxiliary results}\label{sec:appendix}
\begin{lemma}\label{lem:Ih:error}
Let $\Qh$ satisfy \ref{item:S} and let $\ids$ be the identity operator.
Furthermore, let $p\in[1,\infty)$, $q,r\in\tekla{1,\infty}$, $q^*:=\frac{q}{q-1}$, and $r^*:=\frac{r}{r-1}$.
Then the estimates 
\begin{subequations}
\begin{align}
\norm{\trkla{\ids-\Ihxop}\tgkla{f^x\h g^x\h}}_{L^p\trkla{\Ox}}&\leq C\hx^2 \norm{\parx f^x\h}_{L^{pq}\trkla{\Ox}}\norm{\parx g^x\h}_{L^{pq^*}\trkla{\Ox}\label{eq:Ihx:lp:pointwise}}\\
\norm{\parx\trkla{\ids-\Ihxop}\tgkla{f^x\h g^x\h}}_{L^p\trkla{\Ox}}&\leq C\hx \norm{\parx f^x\h}_{L^{pq}\trkla{\Ox}}\norm{\parx g^x\h}_{L^{pq^*}\trkla{\Ox}\label{eq:Ihx:dxlp:pointwise}}\\
\norm{\trkla{\ids-\Ihyop}\tgkla{f^y\h g^y\h}}_{L^p\trkla{\Oy}}&\leq C\hy^2 \norm{\pary f^y\h}_{L^{pq}\trkla{\Oy}}\norm{\pary g^y\h}_{L^{pq^*}\trkla{\Oy}}\label{eq:Ihy:lp:pointwise}\\
\norm{\pary\trkla{\ids-\Ihyop}\tgkla{f^y\h g^y\h}}_{L^p\trkla{\Oy}}&\leq C\hy \norm{\pary f^y\h}_{L^{pq}\trkla{\Oy}}\norm{\pary g^y\h}_{L^{pq^*}\trkla{\Oy}}\label{eq:Ihy:dylp:pointwise}
\end{align}
hold true for all $f^x\h,g^x\h\in\Uhx$ and $f^y\h,g^y\h\in\Uhy$. In addition the estimates
\begin{align}
\begin{split}
\norm{\trkla{\ids-\Ihxyop}\tgkla{f\h g\h}}_{L^p\trkla{\Om}}&\leq C \hx^2 \norm{\parx f\h}_{L^{pq}\trkla{\Om}}\norm{\parx g\h}_{L^{pq^*}\trkla{\Om}}\\
&\qquad+C \hy^2\norm{\pary f\h}_{L^{pr}\trkla{\Om}}\norm{\pary g\h}_{L^{pr^*}\trkla{\Om}}\,,\label{eq:Ihxy:lp}
\end{split}\\
\norm{\trkla{\ids-\Ihxop}\tgkla{f\h g\h}}_{L^p\trkla{\Om}}&\leq C\hx^2 \norm{\parx f\h}_{L^{pq}\trkla{\Om}}\norm{\parx g\h}_{L^{pq^*}\trkla{\Om}}\,,\label{eq:Ihx:lp}\\
\norm{\Ihy{\trkla{\ids-\Ihxop}\tgkla{f\h g\h}}}_{L^p\trkla{\Om}}&\leq C\hx^2 \norm{\parx f\h}_{L^{pq}\trkla{\Om}}\norm{\parx g\h}_{L^{pq^*}\trkla{\Om}}\,,\\
\norm{\parx\trkla{\ids-\Ihxop}\tgkla{f\h g\h}}_{L^p\trkla{\Om}}&\leq C\hx \norm{\parx f\h}_{L^{pq}\trkla{\Om}}\norm{\parx g\h}_{L^{pq^*}\trkla{\Om}}\,,\\
\norm{\trkla{\ids-\Ihyop}\tgkla{f\h g\h}}_{L^p\trkla{\Om}}&\leq C\hy^2 \norm{\pary f\h}_{L^{pr}\trkla{\Om}}\norm{\pary g\h}_{L^{pr^*}\trkla{\Om}}\,,\\
\norm{\Ihx{\trkla{\ids-\Ihyop}\tgkla{f\h g\h}}}_{L^p\trkla{\Om}}&\leq C\hy^2 \norm{\pary f\h}_{L^{pr}\trkla{\Om}}\norm{\pary g\h}_{L^{pr^*}\trkla{\Om}}\,,\\
\norm{\pary\trkla{\ids-\Ihyop}\tgkla{f\h g\h}}_{L^p\trkla{\Om}}&\leq C\hy \norm{\pary f\h}_{L^{pr}\trkla{\Om}}\norm{\pary g\h}_{L^{pr^*}\trkla{\Om}}\label{eq:Ihy:wp}
\end{align}
\end{subequations}
hold true for all $f\h,\,g\h\in\Uhxy$.
\end{lemma}
\begin{proof}
The estimates \eqref{eq:Ihx:lp:pointwise}--\eqref{eq:Ihy:dylp:pointwise} are proven in Lemma 2.1 in \cite{Metzger2020}.
In order to prove \eqref{eq:Ihxy:lp}, we reduce the problem to the one dimensional setting using
\begin{align}
\norm{\trkla{\ids-\Ihxyop}\tgkla{f\h g\h}}_{L^p\trkla{\Om}}^p \leq \norm{\trkla{\ids-\Ihyop}\tgkla{f\h g\h}}_{L^p\trkla{\Om}}^p +\norm{\Ihy{\trkla{\ids-\Ihxop}\tgkla{f\h g\h}}}_{L^p\trkla{\Om}}^p=:A+B
\end{align}
and apply \eqref{eq:Ihx:lp:pointwise}--\eqref{eq:Ihy:dylp:pointwise}.
This provides
\begin{align}
A\leq C\hy^{2p}\int_{\Ox}\norm{\pary f\h}_{L^{pq}\trkla{\Oy}}^p\norm{\pary g\h}_{L^{pq^*}\trkla{\Oy}}^p\dx\leq C\hy^{2p}\norm{\pary f\h}_{L^{pq}\trkla{\Om}}^p\norm{\pary g\h}_{L^{pq^*}\trkla{\Om}}^p\,.
\end{align}
Using Jensen's inequality, \eqref{eq:normequivalence}, and Hölder's inequality, we obtain for $r\in\trkla{1,\infty}$
\begin{align}
\begin{split}
B\leq&\,\int_{\Oy}\!\!\Ihy{\int_{\Ox}\abs{\trkla{\ids-\Ihxop}\tgkla{f\h g\h}}^p\!\dx}\dy\leq C\hx^{2p}\!\int_{\Oy}\!\!\Ihy{\norm{\parx f\h}_{L^{pr}\trkla{\Ox}}^p\norm{\parx g\h}_{L^{pr^*}\trkla{\Ox}}^p}\dy\\
\leq&\, C\hx^{2p}\rkla{\iO\Ihy{\tabs{\parx f\h}^{pr}}\dx\dy}^{1/r} \rkla{\iO\Ihy{\tabs{\parx g\h}^{pr^*}}\dx\dy}^{1/r^*} \\
\leq& C\hx^{2p}\norm{\parx f\h}_{L^{pr}\trkla{\Om}}^p\norm{\parx g\h}_{L^{pr^*}\trkla{\Om}}^p\,.
\end{split}
\end{align}
The corresponding estimate for $r\in\tgkla{1,\infty}$ is straightforward.
The estimates \eqref{eq:Ihx:lp}--\eqref{eq:Ihy:wp} can be derived from \eqref{eq:Ihx:lp:pointwise}--\eqref{eq:Ihy:dylp:pointwise} in a similar manner.
\end{proof}
\begin{lemma}\label{lem:embedding}
Let Assumption \ref{item:S} hold true and let the operator $A\h\,:\,\Uhxy\rightarrow\Uhxy\cap H^1_*\trkla{\Om}$ be defined by
\begin{align}
\iO \trkla{A\h\phi\h} \psi\h \dx\dy = \iO\nabla\phi\h\cdot\nabla\psi\h\dx\dy\qquad\qquad\forall\psi\h\in\Uhxy\,.
\end{align}
Then there exists a constant $C<0$, such that for all $\phi\h\in\Uhxy$ the estimates
\begin{subequations}
\begin{align}
\norm{\phi\h}_{L^\infty\trkla{\Om}}&\leq C \norm{A\h\phi\h}_{L^2\trkla{\Om}}^{\kappa}\norm{\phi\h}_{H^1\trkla{\Om}}^{1-\kappa} + C\norm{\phi\h}_{H^1\trkla{\Om}}\,,\\
\norm{\phi\h}_{W^{1,p}\trkla{\Om}}&\leq C \norm{A\h\phi\h}_{L^2\trkla{\Om}}^{\nu}\norm{\phi\h}_{H^1\trkla{\Om}}^{1-\nu} +C\norm{\phi\h}_{H^1\trkla{\Om}}\,, 
\end{align}
hold true with $\kappa=\tfrac12$, $\nu=\tfrac{3p-6}{2p}$ and $p\in\tekla{2,6}$ for $d=3$, and $\kappa=\tfrac14$, $\nu=\tfrac{p-2}{p}$ and $p\in[2,\infty)$ for $d=2$.
\end{subequations}
\end{lemma}
\begin{proof}
A similar result for continuous, piecewise linear finite element functions defined on a simplicial triangulation has been proven in Lemma A.1 in \cite{Metzger2020}.
This proof relies solely on standard error estimates for second order problems (see e.g.~Chapter 3 in \cite{Ciarlet2002}), standard inverse estimates for finite element functions (see e.g.~Theorem 4.5.11 in \cite{BrennerScott}), and the properties of the Cl\'ement interpolation operator (see e.g.~Lemma 1.127 in \cite{Ern2004}).
As the first two results remain valid for quadrilateral finite element spaces satisfying \ref{item:S} and the Cl\'ement interpolation operator, which was originally only defined for simplicial elements, can be replaced by the more general operator proposed in \cite{Bernardi1998} that satisfies the same error estimates as the Cl\'ement interpolation operator, we refer the reader to \cite{Metzger2020}.
\end{proof}
\begin{lemma}\label{lem:AhDiscLap}
Let Assumption \ref{item:S} hold true and let the operator $A\h\,:\,\Uhxy\rightarrow\Uhxy\cap H^1_*\trkla{\Om}$ be defined by
\begin{align}
\iO \trkla{A\h\phi\h} \psi\h \dx\dy = \iO\nabla\phi\h\cdot\nabla\psi\h\dx\dy\qquad\qquad\forall\psi\h\in\Uhxy\,.
\end{align}
Then there exists a constant $C<0$, such that for all $\phi\h\in\Uhxy$ the estimate
\begin{align}
\norm{A\h\phi\h}_{L^2\trkla{\Om}}\leq C\norm{\Delta\h\phi\h}_{L^2\trkla{\Om}}\,,
\end{align}
where $\Delta\h$ denotes the discrete Laplacian defined in \eqref{eq:defDiscLapl}.
\end{lemma}
\begin{proof}
Noting that the definitions of $\Delta\h^x$ and $\Delta\h^y$ in \eqref{eq:defDiscLapl} imply
\begin{align}
\begin{split}
-\int_{\Ox}\Ihx{\Delta\h^x\phi\h\psi\h}\dx&=\int_{\Ox}\parx\phi\h\parx\psi\h\dx\qquad\qquad\text{pointwise for all~}y\in\Oy\,,\\
-\int_{\Oy}\Ihy{\Delta\h^y\phi\h\psi\h}\dy&=\int_{\Oy}\pary\phi\h\pary\psi\h\dy\qquad\qquad\text{pointwise for all~}x\in\Ox\,,
\end{split}
\end{align}
we compute 
\begin{align}
\begin{split}
\norm{A\h\phi\h}_{L^2\trkla{\Om}}^2&=\iO\parx\phi\h\parx A\h\phi\h\dx\dy+\iO\pary\phi\h\pary A\h\phi\h\dx\dy\\
&=\iO\Ihx{\trkla{-\Delta\h^x\phi\h} A\h\phi\h}\dx\dy+\iO\Ihy{\trkla{-\Delta\h^y\phi\h} A\h\phi\h}\dx\dy\\
&\leq C\norm{\Delta\h^x\phi\h}_{L^2\trkla{\Om}}\norm{A\h\phi\h}_{L^2\trkla{\Om}} + C\norm{\Delta\h^y\phi\h}_{L^2\trkla{\Om}}\norm{A\h\phi\h}_{L^2\trkla{\Om}}\,,
\end{split}
\end{align}
which concludes the proof.
\end{proof}
The following corollary is a direct consequence of Lemma \ref{lem:embedding} and Lemma \ref{lem:AhDiscLap}.
\begin{corollary}\label{cor:embedding}
Let Assumption \ref{item:S} hold true and let $\Delta\h$ be the discrete Laplacian defined in \eqref{eq:defDiscLapl}.
Then there exists a constant $C<0$, such that for all $\phi\h\in\Uhxy$ the estimates
\begin{subequations}
\begin{align}
\norm{\phi\h}_{L^\infty\trkla{\Om}}&\leq C \norm{\Delta\h\phi\h}_{L^2\trkla{\Om}}^{\kappa}\norm{\phi\h}_{H^1\trkla{\Om}}^{1-\kappa} + C\norm{\phi\h}_{H^1\trkla{\Om}}\,,\\
\norm{\phi\h}_{W^{1,p}\trkla{\Om}}&\leq C \norm{\Delta\h\phi\h}_{L^2\trkla{\Om}}^{\nu}\norm{\phi\h}_{H^1\trkla{\Om}}^{1-\nu} +C\norm{\phi\h}_{H^1\trkla{\Om}}\,, 
\end{align}
hold true with $\kappa=\tfrac12$, $\nu=\tfrac{3p-6}{2p}$ and $p\in\tekla{2,6}$ for $d=3$, and $\kappa=\tfrac14$, $\nu=\tfrac{p-2}{p}$ and $p\in[2,\infty)$ for $d=2$.
\end{subequations}
\end{corollary}
\begin{lemma}\label{lem:errorF}
Let $G\,:\R\supset I\rightarrow\R$ be of class $W^{1,\infty}$, let a partition $\Qh$ of a domain $\Om$ satisfying \ref{item:S} be given. 
Furthermore, let $\Uhxy$ be the corresponding space defined in \eqref{eq:def:Uhxy}.
Then the estimate
\begin{align}
\norm{\Ihxy{G\trkla{u\h}}-G\trkla{u\h}}_{L^2\trkla{\Om}}^2\leq C\norm{G}_{W^{1,\infty}}^2 h^2\iO\tabs{\nabla u\h}^2\dx\dy
\end{align}
holds true for arbitrary $u\h\in\Uhxy$.
\end{lemma}
\begin{proof}
We will prove that the claim holds true on every $\Q\in\Qh$.
After an appropriate translation, we may assume that $\Q$ is given by $\Q=\operatorname{co}\tgkla{\trkla{0,0},\,\trkla{\hx,0},\,\trkla{\hx,\hy},\,\trkla{0,\hy}}$.
For arbitrary $\trkla{x,y}\in\Q$, we have
\begin{align}
\begin{split}
G\trkla{u\h\trkla{x,y}}=&\,\tfrac12G\trkla{u\h\trkla{x,0}}+\tfrac12\trkla{\pary G\trkla{u\h}}\trkla{x,\xi_1^y}\cdot y\\
&+\tfrac12G\trkla{u\h\trkla{x,\hy}}+\tfrac12\trkla{\pary G\trkla{u\h}}\trkla{x,\xi_2^y}\cdot\trkla{y-\hy}\\
=&\,\tfrac12 G\trkla{u\h\trkla{0,0}}+\tfrac12\trkla{\parx G\trkla{u\h}}\trkla{\xi_1^x,0}\cdot x+\tfrac12\trkla{\pary G\trkla{u\h}}\trkla{x,\xi_1^y}\cdot y\\
&+\tfrac12 G\trkla{u\h\trkla{\hx,\hy}}+\tfrac12\trkla{\parx G\trkla{u\h}}\trkla{\xi_2^x,\hy}\cdot\trkla{x-\hx}\\
&+\tfrac12\trkla{\pary G\trkla{u\h}}\trkla{x,\xi_2^y}\cdot\trkla{y-\hy}
\end{split}
\end{align}
with $\xi_1^x,\xi_2^x\in\tekla{0,\hx}$ and $\xi_1^y,\xi_2^y\in\tekla{0,\hy}$.
Similarly, we have
\begin{align}
\begin{split}
\trkla{\Ihxy{G\trkla{u\h}}}\trkla{x,y}=&\,\tfrac12\trkla{\Ihxy{G\trkla{u\h}}}\trkla{0,0}+\tfrac12\trkla{\parx\Ihxy{G\trkla{u\h}}}\trkla{\zeta_1^x,0}\cdot x\\
&+\tfrac12 \trkla{\pary\Ihxy{G\trkla{u\h}}}\trkla{x,\zeta_1^y}\cdot y +\tfrac12\trkla{\Ihxy{G\trkla{u\h}}}\trkla{\hx,\hy}\\
&+\tfrac12\trkla{\parx\Ihxy{G\trkla{u\h}}}\trkla{\zeta_2^x,\hy}\cdot\trkla{x-\hx}\\
&+\tfrac12\trkla{\pary\Ihxy{G\trkla{u\h}}}\trkla{x,\zeta_2^y}\cdot\trkla{y-\hy}
\end{split}
\end{align}
with $\zeta_1^x,\zeta_2^x\in\tekla{0,\hx}$ and $\zeta_1^y,\zeta_2^y\in\tekla{0,\hy}$.
Therefore, we have
\begin{align}
\begin{split}
&\tabs{G\trkla{u\h\trkla{x,y}}-\trkla{\Ihxy{G\trkla{u\h}}}\trkla{x,y}}^2\,\\
&\quad\leq C \hx^2\tabs{\trkla{\parx G\trkla{u\h}}\trkla{\xi_1^x,0}}^2 +C \hy^2\tabs{\trkla{\pary G\trkla{u\h}}\trkla{x,\xi_1^y}}^2+C\hx^2\tabs{\trkla{\parx G\trkla{u\h}}\trkla{\xi_2^x,\hy}}^2\\
&\qquad +C\hy^2\tabs{\trkla{\pary G\trkla{u\h}}\trkla{x,\xi_2^y}}^2+C\hx^2\tabs{\trkla{\parx\Ihxy{G\trkla{u\h}}}\trkla{\zeta_1^x,0}}^2\\
&\qquad +C\hy^2\tabs{\trkla{\pary\Ihxy{G\trkla{u\h}}}\trkla{x,\zeta_1^y}}^2+C\hx^2\tabs{\trkla{\parx\Ihxy{G\trkla{u\h}}}\trkla{\zeta_2^x,\hy}}^2 \\
&\qquad+C\hy^2\tabs{\trkla{\pary\Ihxy{G\trkla{u\h}}}\trkla{x,\zeta_2^y}}^2\\
&\quad\leq Ch^2\norm{G}_{W^{1,\infty}}^2\trkla{\tabs{\parx u\h\trkla{\cdot, 0}}^2+\tabs{\parx u\h\trkla{\cdot,\hy}}^2 + \tabs{\pary u\h\trkla{x,\cdot}}^2}\\
&\qquad+C h^2\norm{G}_{W^{1,\infty}}^2\trkla{\tabs{\parx u\h\trkla{\cdot,0}}^2+\tabs{\parx u\h\trkla{\cdot,\hy}}^2 +\Ihx{\tabs{\pary u\h\trkla{x,\cdot}}^2}}\,.
\end{split}
\end{align}
Integrating over $\Q$ and using the norm equivalence stated in \eqref{eq:normequivalence} provides the result.

\end{proof}
\begin{lemma}\label{lem:disc:pi}
Let Assumption \ref{item:S} hold true.
Then the identities
\begin{subequations}
\begin{align}
\begin{split}\label{eq:disc:pi:x1}
-\iOx a^x\h\trkla{x}\Ihx{b^x\h\trkla{x}\dxm c^x\h\trkla{x}}\dx =&\,\iOx a^x\h\trkla{x}\Ihx{\dxm b^x\h\trkla{x} c^x\h\trkla{x-\hx}} \dx \\
&+\iOx\dxp a^x\h\trkla{x}\Ihx{b^x\h\trkla{x}c\h^x\trkla{x}}\dx
\end{split}\\
\begin{split}\label{eq:disc:pi:x2}
\iOx a^x\h\trkla{x}\Ihx{b^x\h\trkla{x}\Delta\h^x c^x\h\trkla{x}}\dx=&\,\iOx \Delta\h^x a^x\h\trkla{x}\Ihx{b^x\h\trkla{x-\hx}c^x\h\trkla{x}}\dx\\
&+\iOx a^x\h\trkla{x+\hx}\Ihx{\Delta\h^x b^x\h\trkla{x}c^x\h\trkla{x}}\dx\\
&+2\iOx\dxp a^x\h\trkla{x}\Ihx{\dxm b^x\h\trkla{x}c\h^x\trkla{x}}\dx
\end{split}\\
\begin{split}\label{eq:disc:pi:y1}
-\iOy a^y\h\trkla{y}\Ihy{b^y\h\trkla{y}\dym c^y\h\trkla{y}}\dy=&\,\iOy a^y\h\trkla{y}\Ihy{\dym b^y\h\trkla{y} c^y\h\trkla{y-\hy}}\dy\\
& +\iOy\dyp a^y\h\trkla{y}\Ihy{b^y\h\trkla{y} c^y\h\trkla{y}}\dy
\end{split}\\
\begin{split}\label{eq:disc:pi:y2}
\iOy a^y\h\trkla{y}\Ihy{b^y\h\trkla{y}\Delta\h^yc^y\h\trkla{y}}\dy=&\,\iOy \Delta\h^y a^y\h\trkla{y}\Ihy{b^y\h\trkla{y-\hy} c^y\h\trkla{y}}\dy\\
&+\iOy a^y\h\trkla{y+\hy}\Ihy{\Delta\h^y b^y\h\trkla{y} c^y\h\trkla{y}}\dy\\
&+2\iOy\dyp a^y\h\trkla{y}\Ihy{\dym b^y\h\trkla{y}c^y\h\trkla{y}}\dy
\end{split}
\end{align}
\end{subequations}
hold true for periodic functions $a^x\h\in C_{\operatorname{per},\Thx}$, $b^x\h,c^x\h\in\Uhx$, $a^y\h\in C_{\operatorname{per},\Thy}$, and $b^y\h,c^y\h\in\Uhy$.
\end{lemma}
\begin{proof}
To show the second identity in \eqref{eq:disc:pi:x1}, we use the periodicity of the considered functions and compute
\begin{align}
\begin{split}
-\hx\iOx &\,a^x\h\trkla{x}\Ihx{b^x\h\trkla{x}\dxm c^x\h\trkla{x}}\dx\\
=&\, \iOx a\h^x\trkla{x+\hx}\Ihx{b^x\h\trkla{x+\hx}c^x\h\trkla{x}}\dx- \iOx a^x\h\trkla{x}\Ihx{b^x\h\trkla{x}c^x\h\trkla{x}}\dx\\
=&\iOx a^x\h\trkla{x+\hx}\Ihx{\trkla{b^x\h\trkla{x+\hx}-b^x\h\trkla{x}}c^x\h\trkla{x}}\dx\\
& +\iOx \trkla{a^x\h\trkla{x+\hx}-a^x\h\trkla{x}}\Ihx{b^x\h\trkla{x}c^x\h\trkla{x}}\dx\\
=&\iOx a^x\h\trkla{x}\Ihx{\trkla{b^x\h\trkla{x}-b^x\h\trkla{x-\hx}}c^x\h\trkla{x-\hx}}\dx\\
& +\iOx \trkla{a^x\h\trkla{x+\hx}-a^x\h\trkla{x}}\Ihx{b^x\h\trkla{x}c^x\h\trkla{x}}\dx\,.
\end{split}
\end{align}
We will now use \eqref{eq:disc:pi:x1} to prove \eqref{eq:disc:pi:x2}.
\begin{align}
\begin{split}
\iOx &\,a^x\h\trkla{x}\Ihx{b^x\h\trkla{x}\dxm\dxp c^x\h\trkla{x}}\dx\\
=&\,-\iOx a^x\h\trkla{x}\Ihx{\dxm b^x\h\trkla{x}\dxp c^x\h\trkla{x-\hx}}\dx -\!\iOx\dxp a^x\h\trkla{x}\Ihx{b^x\h\trkla{x}\dxp c^x\h\trkla{x}}\dx\\
=&\,-\iOx a^x\h\trkla{x}\Ihx{\dxm b^x\h\trkla{x}\dxm c^x\h\trkla{x}}\dx \\
&-\iOx\dxp a^x\h\trkla{x-\hx}\Ihx{b^x\h\trkla{x-\hx}\dxm c^x\h\trkla{x}}\dx\\
=&\iOx a^x\h\trkla{x}\Ihx{\dxm\dxm b^x\h\trkla{x}c^x\h\trkla{x-\hx}}\dx +\iOx\dxp a^x\h\trkla{x}\Ihx{\dxm b^x\h c^x\h\trkla{x}}\dx\\
&+\iOx\dxp a^x\h\trkla{x-\hx}\Ihx{\dxm b^x\h\trkla{x-\hx} c^x\h\trkla{x-\hx}}\dx \\
&+\iOx\dxp\dxp a^x\h\trkla{x-\hx}\Ihx{b^x\h\trkla{x-\hx}c^x\h\trkla{x}}\dx\\
=&\,\iOx a^x\h\trkla{x+\hx}\Ihx{\Delta\h^x b^x\h\trkla{x}c^x\h\trkla{x}}\dx +2\iOx\dxp a^x\h\trkla{x}\Ihx{\dxm b^x\h\trkla{x}c^x\h\trkla{x}}\dx\\
&+\iOx\Delta\h^x a^x\h\trkla{x}\Ihx{b^x\h\trkla{x-\hx}c^x\h\trkla{x}}\dx\,.
\end{split}
\end{align}
The identities \eqref{eq:disc:pi:y1} and \eqref{eq:disc:pi:y2} can be established in a similar manner.
\end{proof}
\end{appendix}
\section*{Acknowledgment} This work has been supported by German Research Foundation (DFG) through the project $397495103$ entitled  ``Free boundary propagation and noise: analysis and numerics of stochastic degenerate parabolic equations''.
\providecommand{\bysame}{\leavevmode\hbox to3em{\hrulefill}\thinspace}
\providecommand{\MR}{\relax\ifhmode\unskip\space\fi MR }
\providecommand{\MRhref}[2]{%
  \href{http://www.ams.org/mathscinet-getitem?mr=#1}{#2}
}
\providecommand{\href}[2]{#2}

\end{document}